\documentclass[12pt]{article}
\usepackage{amsfonts}
\usepackage{amsthm}
\usepackage{CJK}
\usepackage{amssymb}
\usepackage{amsmath}    
\usepackage{graphicx,floatrow}   
\usepackage{verbatim}   
\usepackage{color}      
\usepackage{subfigure}  
\usepackage{epsfig}
\usepackage{float}
\usepackage{epstopdf}
\usepackage{cases}
\usepackage{appendix}
\usepackage{bm}
\usepackage{multirow}
 \usepackage{booktabs}
\usepackage{algorithm}
\usepackage{algorithmic}
\usepackage{array}

\allowdisplaybreaks

\newtheorem{theorem}{Theorem}[section]
\newtheorem{lemma}[theorem]{Lemma}
\newtheorem{definition}{Definition}[section]
\theoremstyle{definition}
\newtheorem{remark}{Remark}[section]

\numberwithin{equation}{section}

\newcommand{\norm}[1]{\left\Vert#1\right\Vert}
\newcommand{\abs}[1]{\left\vert#1\right\vert}

\begin{document}
\title{Multi-phase image segmentation by the Allen--Cahn Chan--Vese model}
\date{}
\author{Chaoyu Liu \thanks{Department of Applied Mathematics, The Hong Kong Polytechnic University, Hung Hom, Hong Kong (polyucy.liu@connect.polyu.hk).}
	\and Zhonghua Qiao \thanks{Department of Applied Mathematics \& Research Institute for Smart Energy, The Hong Kong Polytechnic University,  Hung Hom, Hong Kong ({zhonghua.qiao@polyu.edu.hk}). Z. Qiao's work is partially supported by the Hong Kong Research Grant Council RFS grant RFS2021-5S03, GRF grant 15302919 and the Hong Kong Polytechnic University internal grant 4-ZZLS.}
	\and Qian Zhang\thanks{Corresponding author. Department of Applied Mathematics, The Hong Kong Polytechnic University, Hung Hom, Hong Kong ({qian77.zhang@polyu.edu.hk}). Q. Zhang's research is supported by the 2019 Hong Kong Scholar Program G-YZ2Y and the CAS AMSS-PolyU Joint Laboratory of Applied Mathematics grant 1-ZVA8.}}

\maketitle

\begin{abstract}
This paper proposes an Allen--Cahn Chan--Vese model to settle the multi-phase image segmentation. We first integrate the Allen--Cahn term and the Chan--Vese fitting energy term to establish an energy functional, whose minimum locates the segmentation contour.  The subsequent minimization process can be attributed to variational calculation on fitting intensities and the solution approximation of several Allen--Cahn equations, wherein $n$ Allen--Cahn equations are enough to partition $m = 2^n$ segments.  The derived Allen--Cahn equations are solved by  efficient numerical solvers with exponential time integrations and finite difference space discretization. The discrete maximum bound principle and energy stability of the proposed numerical schemes are proved. Finally, the capability of our segmentation method is verified in various experiments for different types of images.
\end{abstract}

\textbf{Key Words}: Multi-phase image segmentation, Allen--Cahn Chan--Vese model, graph Laplacian, maximum principle, energy stability

\section{Introduction}

Image segmentation aims to partition a given image into several disjoint regions with uniform characteristics such as colors, intensities, and textures. Benefitting from image segmentation, many critical subsequent image applications can be well addressed, for instance, recognition, labeling, and detection. Compared with two-phase image segmentation that only needs to divide the given image into two disjoint parts, multi-phase image segmentation is more practical and challenging \cite{Tai2021JSC}.

There exist many multi-phase image segmentation methods from different perspectives, e.g., region growing methods, graph cut methods, clustering methods, and active contour models.  Wherein active contour models have been studied extensively and achieved great success in the last two decades. For image segmentation with active contour models, it can be attributed to a minimization process of the objective energy functional,  and then it follows the evolution of the segmentation curve to get the desired segmentation result ultimately.
 The key in the extension from two-phase to multi-phase image segmentation is how to give the characteristic representation for each segment.   Samsons et al. extended the work in \cite{Samson2000_IJCV} and adopted $m$ level set functions to represent $m$ phases. Under this circumstance, the model generates overlap or vacuum issues for which an extra constraint is added \cite{Lie2006_MC,Lie2006_TIP}. Then Chan et al. proposed a multi-phase level set method \cite{Chan2003AMS,Vese2002IJCV} based on two-phase segmentation and multi-phase motion \cite{CV2001TIP,Zhao1996JCP}, in which $n$ level set functions can represent $m = 2^n$ disjoint partitions. This method reduces the number of the level set functions and fixes the issue of overlap or vacuum without constraints. However, the level set method is always trapped in the limit of computation efficiency and re-initialization. In \cite{ma2020fast, Wang_2017,Wang_2019}, characteristic functions are adopted to depict each phase and a concave functional is used to approximate the contour length, which can be generalized to many active contour models and applied to binary and multi-phase segmentations. Then an iterative convolution-thresholding method (ICTM) is proposed in this framework to solve active contour models in an efficient and energy stable way. But for images with noise, the segmentation results of ICTM are not so satisfying when it is applied to some noise-sensitive active contour models. 

In active contour models, the objective energy functional consists of a regularization term and a fitting term. The regularization term is used to minimize the perimeter of the closed curve. Because $\Gamma$ convergence theory \cite{Ambrosio1990_CPAM,Jung2007_siam} links the perimeter of the closed curve with the phase transition model (Modica--Mortola) \cite{Esedoglu2006_JCP}, we proposed the Allen--Cahn equation, known as the phase-field model, in conjunction with the Chan--Vese fitting term \cite{CV2001TIP} to handle two-phase image segmentation in \cite{Zhang2021_NMTMA}. The corresponding energy functional is given by
\begin{align}\label{CV2d}
E^\epsilon(u,C_1,C_2) = \int_{\Omega}\left(\epsilon\abs{\nabla u}^2 + \frac{1}{\epsilon}W(u)+F(u,C_1,C_2)\right) dx,
\end{align}
where $$F(u,C_1,C_2)= \lambda_1[(C_1-I)^2 H(u-\frac{1}{2})] + \lambda_2[(C_2-I)^2(1-H(u-\frac{1}{2}))].$$
An alternating direction minimization method {(ADMM)} was used to update the phase variable $u$, the two-phase average intensities $C_1$ and $C_2$ iteratively.
For the derived Allen--Cahn equation, the phase variable $u$ eventually evolves into two phases, 0 and 1. In this way, one phase variable $u$ can partition the given image into two phases. In this paper, we extend \eqref{CV2d} to a multi-phase image segmentation model. If there are $m$ phases to be separated, we need at least $n = \lceil \log_2{m} \rceil$ phase variables. Then, by adopting the ADMM for the minimization process, we need to solve $n $ Allen--Cahn equations.


Energy stability is an intrinsic property of the Allen--Cahn equation. Many energy stable numerical schemes have been developed in the simulation of the Allen--Cahn equation, such as the convex splitting schemes \cite{ShenCMS_2016}, stabilized schemes \cite{Li2016_SINUM,Shen2010_DCDS}, invariant energy quadratization schemes \cite{YangJCP_2016}, scalar auxiliary variable schemes \cite{ShenSAV}, and so on. The maximum bound principle is another vital character of the Allen--Cahn equation. In this paper, we employ the exponential time differencing (ETD) method to solve the derived Allen--Cahn equations. The proposed ETD schemes are proved to be unconditional energy stability and maximum bound principle preserving following the ideas in \cite{Li2019siam,Li2021sirev,Fu2022}. It guarantees the stability of our ADMM-ETD algorithm and locates the phase variables in a given range. Moreover, the application of the 2D Fast Fourier transform (FFT) for solving the ETD schemes accelerates the computation notably. For more details about higher-order ETD schemes, please refer to \cite{Hoch2010AN,JuCMS2015,Zhu2016_JSC} and the references therein.

It is worth noting that the initial contour significantly influences the final segmentation results due to the non-convexity of the objective energy functional. A feasible solution for this initialization issue is to relax the energy functional into a convex functional \cite{Cai2013SIAM,Cai2015JSC,AAMM2020, Wu2021_SP}. However, convex relaxation deprives all the non-convex parts of the original energy functional, and hence it may entail the loss of non-convex information of the segments and reduce models' capability of preserving the sharpness and neatness of edges \cite{Chan2018}.  In \cite{Liu2021,Zhang2021_NMTMA}, an inhomogeneous graph Laplacian initialization method (IGLIM) was proposed, where an edge detection result was chosen to be the initial contour. This method can alleviate the sensitivity of initialization.  For multi-phase image segmentation, a generalized initialization method of IGLIM, called the multi-phase inhomogeneous graph Laplacian initialization method (Multi-IGLIM) \cite{Liu2022}, will be employed to select the initial contours for the whole iterative process.

The rest of this paper is organized as follows. In Section 2, the four-phase image segmentation is utilized as an example to illustrate our model. Section 3 is devoted to giving the numerical method, including the choice of initial contour and the energy minimization process. The Allen--Cahn equations derived by ADMM will be solved by the ETD schemes for temporal discretization, and the central finite difference scheme for spatial discretization. The discrete maximum bound principle and energy stability of the ETD schemes are proved in this section. Section 4 contains numerical experiments for various images. We shown the effectiveness of the proposed model by comparing it with other state-of-the-art segmentation methods. The theoretical results proved in Section 3 are validated numerically. The paper ends with some conclusions in Section 5.

\section{The Allen--Cahn Chan--Vese model}
Referring to the work in \cite{Liu2021,Zhang2021_NMTMA,Yang2019_JSC}, we give the energy functional based on the Allen--Cahn Chan--Vese  (ACCV) model for the multi-phase image segmentation as follows,
\begin{align}\label{eqn:AC_model}
E^\epsilon = \int_{\Omega} \sum\limits_i^n\left(\epsilon \abs{\nabla u_i}^2 + \frac{1}{\epsilon}W(u_i)\right)  + F(u_1,u_2,\cdots,u_n,I) \mathrm{d}x, \ i = 1,\cdots,n.
\end{align}
Here $n$ is the number of phase variables, by which we can segment $m = 2^n$ phases \cite{Vese2002IJCV}.
 $W(u) = u^2(u-1)^2$ is a double-well potential function. $F$ is a fitting term involving $u_i,i=1,\cdots,n$ and the original image $I$.  Each phase variable $u_i$ eventually evolves to two phases 0 and 1. We assume the interface, i.e., the edge, lies at the contour of $\displaystyle u_i=\frac{1}{2}$. In Figure \ref{fig:48phase}, we can see the segmentation of 4 and 8 phases.

\begin{figure}[!htbp] \centering
		\subfigure[4 phases] {
			\includegraphics[width=0.47\columnwidth]{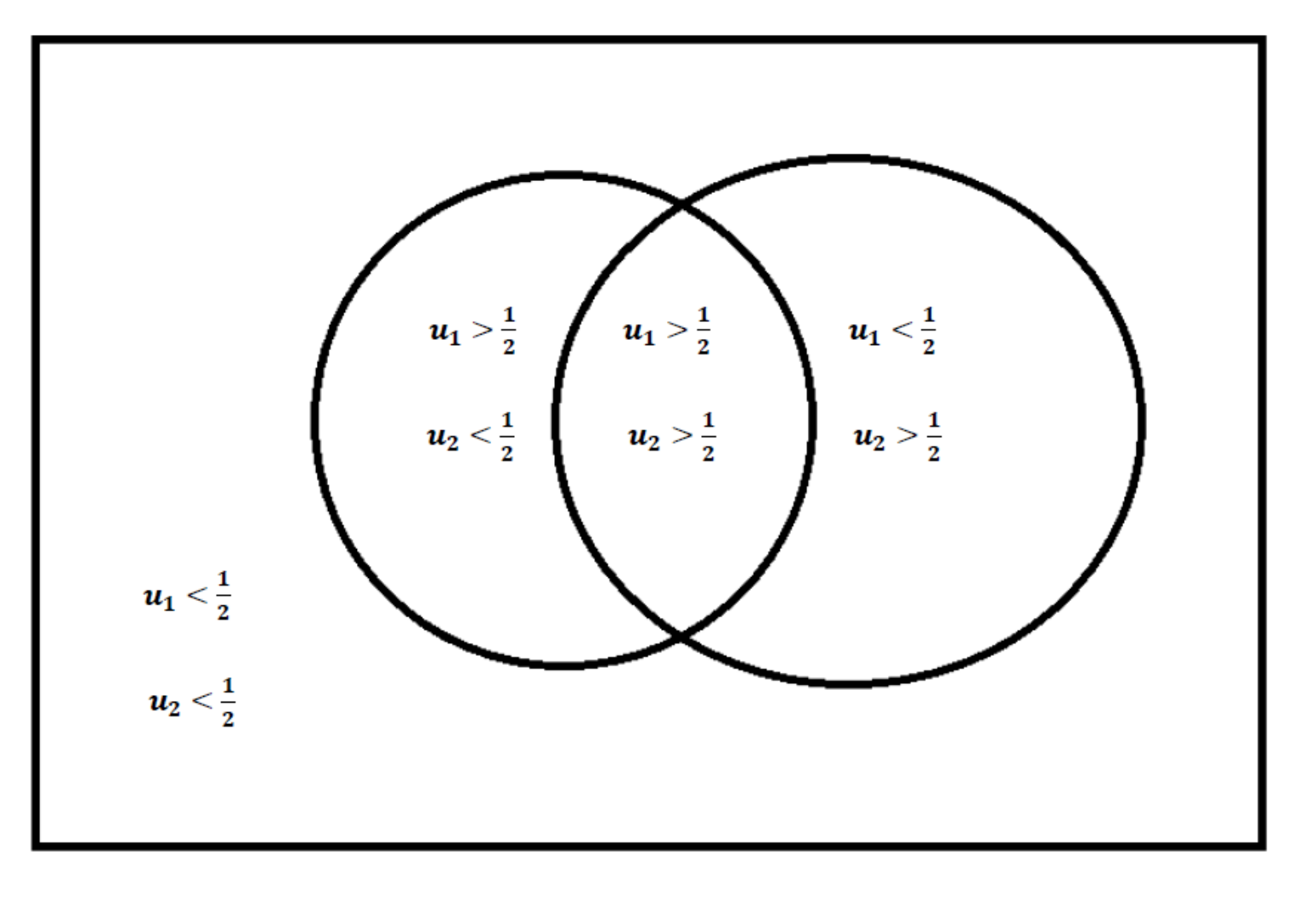}
		}
		\subfigure[8 phases] {
			\includegraphics[width=0.47\columnwidth]{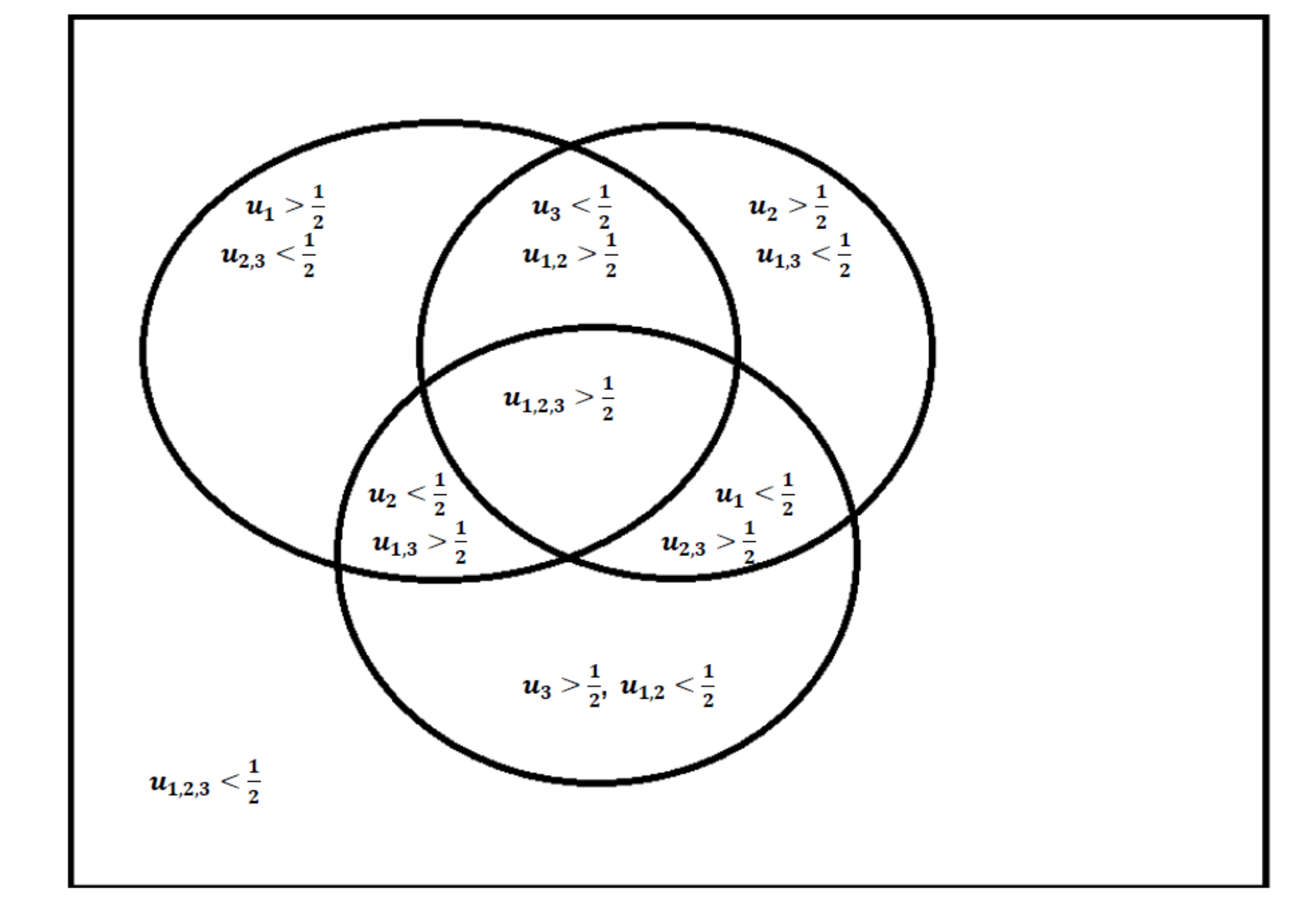}
		}
		\caption{\label{fig:48phase} Variable $u_i$ on 4 phases and 8 phases. }
	\end{figure}
For a clear description, we take $n = 2$ to do the illustration in the follow parts. All derivations can be easily extended to the case with larger $n$. We need to minimize the following energy functional
\begin{align}\label{eqn:energy_4_phase}
E(u_1,u_2,\mathbf{C}) = &\int_{\Omega}\sum\limits_{i=1}^2\left(\epsilon\abs{\nabla u_i}^2 + \frac{1}{\epsilon}W(u_i)\right)+ F(u_1,u_2,I)  \mathrm{d}x,
\end{align}
where  $\mathbf{C} = (\mathbf{C}_{11},\mathbf{C}_{12},\mathbf{C}_{21},\mathbf{C}_{22})$ and
\begin{equation}
\begin{split}\label{eqn:fitterm4}
F(u_1,u_2,I) = &\lambda  \sum_{r=1}^\omega [(I(r)- \mathbf{C}_{11}(r))^2H(u_1 - \frac{1}{2})H(u_2 - \frac{1}{2}) \\
 + &(I(r)- \mathbf{C}_{12}(r))^2H(u_1 - \frac{1}{2})\left(1-H(u_2 - \frac{1}{2} )\right) \\
 + &(I(r)- \mathbf{C}_{21}(r))^2\left(1-H(u_1 - \frac{1}{2})\right)H(u_2 - \frac{1}{2} )\\
 + & (I(r)- \mathbf{C}_{22}(r))^2\left(1-H(u_1 - \frac{1}{2})\right)\left(1-H(u_2 - \frac{1}{2} )\right)]
\end{split}
\end{equation}
with the Heaviside function $H(u)$ given by
 \begin{align*}
H(u) =  \left\{
\begin{array}{rc}
0  ,   &  \text{for}\ u < 0,\\
1    , &      \text{for}\ u \geq 0.
\end{array} \right.
\end{align*}
In the above description, we assume that the original image $I$ is a color image in space $\mathbb{R}^{M_1\times M_2\times \omega}$.  Compared with the grayscale image with only one channel $\omega=1 $, the color image has three channels (R, G, B), which gives $\omega =3$.
Therefore, the intensity averages
${\bf {C}_{11}, {C}_{12},{C}_{21},{C}_{22}}$ are constant vectors in space $\mathbb{R}^{M_1\times M_2\times 3}$.

\section{The numerical method}

In this section, we propose efficient numerical schemes to solve the ACCV model.
 An alternating direction minimization method is used to minimize the energy functional \eqref{eqn:AC_model}.

\subsection{Initialization}
 As mentioned in \cite{Liu2021,Liu2022, Zhang2021_NMTMA}, initialization is a vital step for the image segmentation.
 Here, we use the Multi-IGLIM proposed in \cite{Liu2022} to generate the
 initial contour for the multi-phase image segmentation. Multi-IGLIM comprises three stages: finding the zero-cross points of the generalized
 inhomogeneous graph Laplacian operator, the K-means clustering, and denoising based on the diagonal connectivity.
 For the sake of the completeness of this paper, we restate the essential part of the Multi-IGLIM. The readers can refer to \cite{Liu2022}
 for more detailed descriptions.

\subsubsection{Generalized inhomogeneous graph Laplacian operator}
In \cite{Liu2021}, we give the definition of the inhomogeneous graph Laplacian operator for a grayscale image, which can be easily extended to a color
image by summation of information on different channels.

Let $\Omega$ be a 2D discrete image domain, and $I$ be an image defined on it with $M_1\times M_2$ pixels. For pixel $x_0 = (i,j)\in \Omega$, the generalized inhomogeneous graph Laplacian operator $L$ for a color image is given by
\begin{equation}\label{Graph_L}
L(x_0)= \sum\limits_{\ell =1}^{8}\sum_{r=1}^{\omega} (c_{\ell} I_{i,j,r}^{\ell} -I_{i,j,r}),
\end{equation}
where $I_{i,j,r}^{\ell}$ is intensity value of the $\ell$-th neighbour point of $I_{i,j}$ in the $r$-th channel. The relevant 8 neighbors are
$$
\begin{aligned}
	& I_{i,j,r}^1 = I_{i-1,j-1,r}, I_{i,j,r}^2 = I_{i-1,j,r}, I_{i,j,r}^3 = I_{i-1,j+1,r}, I_{i,j,r}^4 = I_{i,j+1,r}, \\
	& I_{i,j,r}^5 = I_{i+1,j+1,r}, I_{i,j,r}^6 = I_{i+1,j,r},
	I_{i,j,r}^7 = I_{i+1,j-1,r}, I_{i,j,r}^8 = I_{i,j-1,r}, 
\end{aligned}
$$
and the weighted coefficient
\begin{align} \label{eqn:cl}
c_{\ell}=\frac{\sum\limits_{r=1}^{\omega}e^{\kappa(I_{i,j,r}-I_{i,j,r}^{\ell})^2}}{\sum\limits_{\ell=1}^{8} \sum\limits_{r=1}^{\omega}e^{\kappa(I_{i,j,r}-I_{i,j,r}^{\ell})^2}},
\end{align}
which is between $0$ and $1$. $\kappa$ in \eqref{eqn:cl} is a given non-negative parameter.

Subsequently, we choose a threshold value $\sigma$ to locate the zero-cross points of the generalized inhomogeneous Laplacian operator by
condition $ \abs{L(x_0)}>\sigma$.  Finally, all zero-cross points are collected in the set $\mathcal{S}$ to represent rough initial edges.

\subsubsection{The K-means clustering}
This part is devoted to dividing the rough initial edges obtained in the last stage into different phases. We utilize the K-means clustering algorithm proposed in
\cite{Macqueen1967MSP}, minimizing the sum of distances from each object to its cluster centroid for all clusters. The edges of each phase have different intensities as well as the interior region. Consequently, the K-means clustering algorithm can help segment the rough initial edges into $m$ phases. 

\subsubsection{A denoising method based on the connectivity of edge points}
Due to the influence of noise, the denoising process is necessary. Similar to \cite{Liu2021}, the diagonal connectivity is used as a guiding principle to remove most noise points from the rough initial color contours. The edge of an objective should have connectivity, which means that points of edges connect with each other, while the noise doesn't possess this property.

\begin{definition}\cite{Liu2021}
The diagonally connected points for $x_0 = (i,j)$ are defined as follows:
\begin{align*}
	&S_1 = \{(i-1,j-1),(i,j-1),(i-1,j)\}, \ S_2 = \{(i-1,j+1),(i,j+1),(i-1,j)\},\\
	&S_3 = \{(i+1,j-1),(i,j-1),(i+1,j)\}, \ S_4 = \{(i+1,j+1),(i,j+1),(i+1,j)\}.
\end{align*}
We call $x$ a diagonally connected point if both $S_1$ and $S_4$ or both $S_2$ and $S_3$ have at least one pixel that also belongs to the detected edge.
\end{definition}
To eliminate noise in the rough initial color contours, we keep all the diagonally connected points and remove the other points.
The denoising process needs to be repeated $M$ times where $M$ is a pre-setting small integer.
Now combining all above treatments,  we restate the Multi-IGLIM as follows.
\begin{algorithm}
	\caption{\leftline{\bf{Multi-IGLIM} \cite{Liu2022}}}
		\label{multi-IGLIM}
\begin{algorithmic}[1]
	\REQUIRE{$I, \kappa, \sigma, M, m$}
	\ENSURE{Initial contours for $m$ phases $v^0 = (v_1^0,v_2^0,\cdots,v_m^0)$}
	\STATE{Compute the inhomogeneous graph Laplacian value of each pixel by \eqref{Graph_L}}.
	\STATE{Set a small positive number $\sigma$. Put all together pixels $x$ whose $|L(x)|\ge \sigma$ into a set denoted as $\mathcal{S}$}.
	\STATE{Divide into $m$ separate sets all pixels of $\mathcal{S}$ by K-means according to their pixel values. Denote these sets as $\mathcal{S}_1,\cdots,\mathcal{S}_m$}.
	\STATE{Go through every pixel in $\mathcal{S}_1,\cdots,\mathcal{S}_m$ and judge whether it is diagonally connected in their set. If not, remove it from their set.}
	\STATE{Set an appropriate integer $M$, and repeat Step 4 for $M$ times.}
	\STATE{Output the characteristic function of $\mathcal{S}_i, i = 1,\cdots,m$, which is defined as $v_i^0$, as the initial contours.}
\end{algorithmic}
\end{algorithm}

\subsection{Energy minimization}
The situation of two phase variables, i.e., four phases image segmentation, is also taken as an instance to explain the energy minimization process. We will introduce the adopted minimization method,  numerical schemes for the derived Allen--Cahn equations, and prove the discrete maximum bound principle and energy stability. 

\subsubsection{The transition from $m$ phases to $n$ phase-field variables}
From the Multi-IGLIM in Algorithm \ref{multi-IGLIM}, we get the denoising edge detection results $v_i^0$ for $m$ phases.  However, the number of the initial values for the minimization process for the ACCV model \eqref{eqn:AC_model} is $n =\lceil \log_2{m} \rceil$. 
In this part, $m=4,n=2$. We first sort the results of the Multi-IGLIM, $v^0_i$, in ascending order by the perimeter for each phase. Then the initial contours for two phase variables are
\begin{align}
&u^0_1 = v^0_2 + v^0_3,\label{integration_method1}\\
&u^0_2 = v^0_2 + v^0_4.\label{integration_method2}
\end{align}
In the case of overlapped edges, $u^0_i$ will be beyond $1$ somewhere, for which we will force them to equal to $1$ directly. The integration method \eqref{integration_method1}-\eqref{integration_method2}  is not unique, and other reasonable treatment can also be considered.

\subsubsection{The alternating direction minimization method}
This part introduces the alternating direction minimization method (ADMM) to minimize the energy functional \eqref{eqn:energy_4_phase}.
The minimization procedure can be written as
\begin{align}
&\mathbf{u}^{k+1} = \text{argmin} \ E(\mathbf{u}^{k},\mathbf{C}^k), \label{eqn:iter1} \\
&\mathbf{C}^{k+1} = \text{argmin} \ E(\mathbf{u}^{k+1},\mathbf{C}^k), \label{eqn:iter3}
\end{align}
with $\mathbf{u} = (u_1,u_2).$

Since the Heaviside function $H(u)$ in \eqref{eqn:fitterm4} is discontinuous,  we give $H_{\epsilon_1}$ as a regularization of the Heaviside function $H(u)$ in practice
\begin{align*}
H_{\epsilon_1}(u) = \begin{cases}&  \frac{1}{2\epsilon_1}\left(u + \frac{\epsilon_1}{\pi} \sin(\frac{\pi u}{ \epsilon_1})\right) + \frac{1}{2}, \ \  \text{for}\abs{u} \leq \epsilon_1, \\
& 1,  \ \  \text{for}\ u > \epsilon_1, \\
&0,  \ \  \text{for}\ u < -\epsilon_1.
\end{cases}
\end{align*}
Here, the value of $\epsilon_1$ is requested to be less than $\frac{1}{2}$. The derivative of $H_{\epsilon_1}(u)$ is an approximation of Dirac function $\displaystyle\delta(u)=\frac{\mathrm{d}}{\mathrm{du}}H(u)$ given by \cite{Li2005CVPR}
\begin{align*}
\delta_{\epsilon_1}(u)  = \begin{cases}& \frac{1}{2\epsilon_1}\left(1 + \cos(\frac{\pi u}{ \epsilon_1})\right), \ \  \text{for}\abs{u} \leq \epsilon_1, \\
                     &0,  \ \  \text{for}\ \abs{u} > \epsilon_1.
        \end{cases}
\end{align*}
For a fixed $u_i, i = 1,2$ and $r = 1, \cdots, \omega$, the variational calculation gives
\begin{subequations}\label{eqn:c}
\begin{align}
&\mathbf{C}_{11}(r) = \frac{\int_{\Omega} H_{\epsilon_1}(u_1-\frac{1}{2})H_{\epsilon_1}(u_2-\frac{1}{2})I(r) \mathrm{d}x}{\int_{\Omega} H_{\epsilon_1}(u_1-\frac{1}{2})H_{\epsilon_1}(u_2-\frac{1}{2})\mathrm{d}x} , \label{eqn:split2_1}\\
&\mathbf{C}_{12}(r) = \frac{\int_{\Omega} H_{\epsilon_1}(u_1-\frac{1}{2})\left(1-H_{\epsilon_1}(u_2-\frac{1}{2})\right)I(r)\mathrm{d}x}{\int_{\Omega} H_{\epsilon_1}(u_1-\frac{1}{2})\left(1-H_{\epsilon_1}(u_2-\frac{1}{2})\right) \mathrm{d}x} , \label{eqn:split2_2}\\
&\mathbf{C}_{21}(r)= \frac{\int_{\Omega} \left( 1-H_{\epsilon_1}(u_1-\frac{1}{2})\right) H_{\epsilon_1}(u_2-\frac{1}{2})I(r) \mathrm{d}x}{\int_{\Omega} \left( 1-H_{\epsilon_1}(u_1-\frac{1}{2})\right) H_{\epsilon_1}(u_2-\frac{1}{2}) \mathrm{d}x} , \label{eqn:split2_3} \\
&\mathbf{C}_{22}(r) = \frac{\int_{\Omega} \left( 1-H_{\epsilon_1}(u_1-\frac{1}{2})\right)\left(1- H_{\epsilon_1}(u_2-\frac{1}{2})\right)I(r) \mathrm{d}x}{\int_{\Omega} \left( 1-H_{\epsilon_1}(u_1-\frac{1}{2})\right)\left( 1-H_{\epsilon_1}(u_2-\frac{1}{2})\right) \mathrm{d}x}.\label{eqn:split2_4}
\end{align}
\end{subequations}
For \eqref{eqn:iter1}, we only need to solve the following two Allen--Cahn equations to the steady state,
\begin{align}
\frac{\partial u_1}{\partial t} = 2\epsilon\Delta u_1 - \frac{1}{\epsilon}w(u_1) - f_1(u_1,u_2,\mathbf{C}),\label{MAC_1}\\
\frac{\partial u_2}{\partial t} = 2\epsilon\Delta u_2 - \frac{1}{\epsilon}w(u_2) - f_2(u_1,u_2,\mathbf{C}),\label{MAC_2}
\end{align}
with the homogeneous Neumann boundary condition.
Here,
\begin{align*}
f_1(u_1,u_2,\mathbf{C}) = & \lambda\sum\limits_{r=1}^\omega \left[(I(r)-\mathbf{C}_{11}(r))^2 -(I(r)-\mathbf{C}_{21}(r))^2 \right]H_{\epsilon_1}(u_2-\frac{1}{2})\delta_{\epsilon_1}(u_1 - \frac{1}{2})  \notag \\
&+ \left[(I(r)-\mathbf{C}_{12}(r))^2 - (I(r)-\mathbf{C}_{22}(r))^2 \right]\left(1-H_{\epsilon_1}(u_2-\frac{1}{2})\right)\delta_{\epsilon_1}(u_1 - \frac{1}{2}), \\
f_2(u_1,u_2,\mathbf{C}) = &\lambda\sum\limits_{r=1}^\omega \left[(I(r)-\mathbf{C}_{11}(r))^2 - (I(r)-\mathbf{C}_{12}(r))^2 \right]H_{\epsilon_1}(u_1-\frac{1}{2})\delta_{\epsilon_1}(u_2 - \frac{1}{2}) \notag \\
&+ \left[(I(r)-\mathbf{C}_{21}(r))^2 - (I(r)-\mathbf{C}_{22}(r))^2 \right]\left(1-H_{\epsilon_1}(u_1-\frac{1}{2})\right)\delta_{\epsilon_1}(u_2 - \frac{1}{2}).
\end{align*}
The function $w(u) = 2u(2u-1)(u-1)$ is the derivative of $ W(u)$ with respect to $u$.  

Considering the properties of maximum boundedness and energy decay for the classical Allen--Cahn equation, we would next explore the relevant properties for the proposed Allen--Cahn equations \eqref{MAC_1} and \eqref{MAC_2}.   According to the criterion in \cite{Li2021sirev}, we can verify that the proposed Allen--Cahn equations \eqref{MAC_1} and \eqref{MAC_2} satisfy the following maximum bound principle
\begin{align*}
\text{If}  \ 0 \leq \abs{u(x,t = 0)} \leq 1, \     \text{ then} \ 0 \leq \abs{u(x,t)} \leq 1, \ \forall t\in(0,\infty),
\end{align*}
for homogeneous Neumann boundary condition. Besides, the energy decay property for \eqref{eqn:energy_4_phase} holds obviously.
In the following part, we will design efficient numerical schemes preserving the corresponding properties, which are stable and reliable to get the steady state solution. In \cite{Li2019siam,Fu2022}, it is shown that the ETD method can maintain the maximum bound principle and energy stability for the Allen--Cahn equation. As a consequence, we will utilize the ETD method as temporal discretization to construct the numerical schemes for the proposed Allen--Cahn equations \eqref{MAC_1} and \eqref{MAC_2}.

\subsubsection{Exponential time differencing methods}
Using the central finite difference method to discretize \eqref{MAC_1} and \eqref{MAC_2} in space, we obtain the following
ODE system for numerical solution $U_i =[U_{i,\ell}]^{M_1M_2}_{\ell=1}\in \mathbb{R}^{M_1M_2}$,
\begin{align}\label{ODE}
\frac{\mathrm{d}U_i}{\mathrm{dt}}  + {L}_h U_i = {N}(U_i), \ \  i = 1,2,
\end{align}
where
\begin{align*}
L_h = -2\epsilon D_h  + SI_d, \quad N(U_i) = SU_i - \frac{1}{\epsilon}w(U_i) - f_i(u_1,u_2,\mathbf{C}).
\end{align*}
Here, $D_h$ is a 2D discrete Laplacian matrix under the homogeneous Neumann boundary condition,
\begin{align*}
  D_{ h } = \frac{1}{h^2}(I_{M_2}\otimes\Lambda_{M_1}+\Lambda_{M_2}\otimes I_{M_1}),  \ h \text{ is the spatial step},
\end{align*}
where $I_{M_i}$ is an $M_i\times M_i$ identity matrix and
$$
\Lambda_{ M_i } = \left[ \begin{array} { c c c c c c } - 1 & 1 & & & & 0\\ 1 & - 2 & 1 & & & \\ & & \ddots & \ddots & \ddots & \\ & & & 1 & - 2 & 1 \\ 0 & & & & 1 & - 1 \end{array} \right]_{ M_i \times M_i }, \quad i=1, 2.
$$
$I_d\in \mathbb{R}^{M_1M_2 \times M_1M_2}$  is an identity matrix. The positive constant $S$ is called the stabilizer.

Multiplying \eqref{ODE} with an integrating factor $e^{L_ht}$, and integrating from $t^n$ to $t^{n+1}$ give
\begin{align*}
U_i(t^{n+1}) = e^{-L_h\Delta t}U_i(t^n) + e^{-L_h\Delta t}\int_0^{\Delta t}e^{L_hs}N(U_i(t^n+s))ds.
\end{align*}
Then two approximations for the nonlinear term $N(U(t^n+s))$ are proposed.
 \begin{itemize}
\item ETD1 : Approximating $N(U_i(t^n+s))$ by $N(U_i(t^n))$, we have a first-order scheme
\begin{align}\label{ETD1}
U_i^{n+1} = \phi_0(L_h\Delta t) U_i^n +\Delta t \phi_1(L_h\Delta t) N(U_i^n).
\end{align}
\item ETDRK2: Approximating $N(U_i(t^n+s))$ with $\displaystyle(1-\frac{s}{\Delta t})N(U_i(t^n)) + \frac{s}{\Delta t}N(\hat{U}_i^{n+1})$,
wherein $\hat{U}_i^{n+1}$ is a first-order approximation obtained from ETD1. As a consequence, the ETDRK2 scheme is given by
    \begin{equation}\left\{
    \begin{split}
    &\hat{U}_i^{n+1} = \phi_0(L_h\Delta t) U_i^n +\Delta t \phi_1(L_h\Delta t) N(U^n_i),\\
    &{U}_i^{n+1} = \phi_0(L_h\Delta t) \hat{U}_i^{n+1} +\Delta t \phi_2(L_h\Delta t)( N(\hat{U}^{n+1}_i)- N(U^n_i)). \label{ETD2}\\
         \end{split}\right.
     \end{equation}
\end{itemize}
In \eqref{ETD1} and \eqref{ETD2}, $\phi_j,j = 0,1,2$ are defined as
\begin{align*}
\phi_0(a) = e^{-a}, \quad  \phi_1(a) = \frac{1-e^{-a}}{a}, \quad \phi_2 = \frac{e^{-a}-1+a}{a^2}.
\end{align*}
The calculation process of $\phi_i(L_h\Delta t)$ can be simplified by 2D discrete cosine transform (DCT) efficiently.
The time complexity is $\mathcal{O}((M_1M_2)^2log(M_1M_2))$ per time step. For more detailed illustrations,
please refer to \cite{Hoch2010AN,Ju2015JSC}.

We summarize the algorithm described above as follows.
\begin{algorithm}\renewcommand{\arraystretch}{2.0}
	\caption{\leftline{\bf{ADMM-ETD solver}}}
		\label{ACCV}
\begin{algorithmic}[1]
	\REQUIRE{Initial values $\mathbf{u}^0$ and $\mathbf{C}^0$}
	\ENSURE{Convergence result $u^s$}
	\STATE{With $\mathbf{u}^{k}, \mathbf{C}^k$, solve the Allen--Cahn equations by the numerical scheme \eqref{ETD1} (or \eqref{ETD2}) till the steady state $\mathbf{u}^{k+1}$}.
	\STATE{Calculate the new $\mathbf{C}^{k+1}$  by \eqref{eqn:c} with  $\mathbf{u}^{k+1}$}.
	\STATE{Set $\mathbf{u}^{k}=\mathbf{u}^{k+1}$. Repeat Steps 1-2 until the convergence of $\mathbf{u}^{k+1}$.}
\end{algorithmic}
\end{algorithm}

%
%
%
%
%
%
\subsubsection{Discrete maximum bound principle}
 In this part, we characterize the stabilizer $S$ such that the discrete maximum bound principle holds for the proposed ETD scheme \eqref{ETD1} (or \eqref{ETD2}).
We assume the phase variable $u_i$ in the interval $[0,1]$,  and project $u_i \in [0,1]$ to $\tilde{u}_i \in [-1,1]$ by the operator $\mathcal{P}$
\begin{align}\label{eqn:proj}
\tilde{u}_i = \mathcal{P}({u}_i) = 2(u_i - \frac{1}{2}), \ {u}_i = \mathcal{P}^{-1}(\tilde{u}_i) =  \frac{1}{2}\tilde{u}_i + \frac{1}{2}, \ i = 1,2.
\end{align}
Therefore the original image $I$ and the average intensity $\mathbf{C}$ are also rescaled in $ [-1,1]$ and denoted by
$\tilde{I},\tilde{\mathbf{C}}$, correspondingly. Substituted $\tilde{u}_i, \tilde{I},\tilde{\mathbf{C}}$ into the Allen--Cahn equation \eqref{MAC_1} and \eqref{MAC_2},  the equations become
\begin{align}
\frac{\partial \tilde{u}_1}{\partial t} = 2\epsilon\Delta \tilde{u}_1 - \frac{2}{\epsilon}\tilde{w}(\tilde{u}_1) - {f}_1({  \frac{1}{2}\tilde{u}_1,\frac{1}{2}\tilde{u}_2}, \tilde{\mathbf{C}}),\label{MAC_r1}\\
\frac{\partial \tilde{u}_2}{\partial t} = 2\epsilon\Delta \tilde{u}_2 - \frac{2}{\epsilon}\tilde{w}(\tilde{u}_2) - {f}_2({  \frac{1}{2}\tilde{u}_1,\frac{1}{2}\tilde{u}_2}, \tilde{\mathbf{C}}),\label{MAC_r2}
\end{align}
where $\tilde{w}(u) = \frac{1}{2}u(u-1)(u+1)$. 
Then the ODE system \eqref{ODE} turns to
\begin{align}\label{t_ODE}
\frac{\mathrm{d}\tilde{U}_i}{\mathrm{dt}}  + \tilde{L}_h \tilde{U}_i = { \tilde{N}_i(\tilde{U}_i)},
\end{align}
with
{
\begin{align*}
&\tilde{L}_h =  -2\epsilon D_h  + \tilde{S}I_d, \\
&\tilde{N}_1(\tilde{U}_1)  = \tilde{S}\tilde{U}_1 - \frac{2}{\epsilon}\tilde{w}(\tilde{U}_1) - f_1({  \frac{1}{2}\tilde{U}_1,\frac{1}{2}\tilde{U}_2},\mathbf{\tilde{C}}),\\
&\tilde{N}_2(\tilde{U}_2)  = \tilde{S}\tilde{U}_2 - \frac{2}{\epsilon}\tilde{w}(\tilde{U}_2) - f_2({  \frac{1}{2}\tilde{U}_1,\frac{1}{2}\tilde{U}_2},\mathbf{\tilde{C}}).
\end{align*}
}

Following the idea in \cite{Zhang2021_NMTMA}, we turn to the proof of discrete maximum bound principle for the ETD numerical schemes of \eqref{t_ODE}.

{
\begin{lemma}\label{lemma:bound}
There holds the following bound for nonlinear term $\tilde{N}_i$, i = 1,2
\begin{align}
\norm{\tilde{N}_i(\zeta)}_{\infty} \leq \tilde{S}, \ \ -1\leq \zeta \leq 1.
\end{align}
 provided that
 \begin{align*}
 \tilde{S}\geq \gamma \triangleq \frac{2}{\epsilon} + \frac{2\omega\lambda\pi}{\epsilon_1^2}\ \text{and}\ \epsilon_1 = \frac{1}{2p},\ {p}\ \text{is a positive odd number},
\end{align*}
Here, $\omega$ is the number of channel of image $\tilde{I}$.
\end{lemma}
\begin{proof}
To estimate the boundedness of nonlinear term  $\tilde{N}_1$,
\begin{align*}
\tilde{N}_1(\zeta)  = \tilde{S}\zeta - \frac{2}{\epsilon}\tilde{w}(\zeta) - f_1({ \frac{1}{2}\zeta,\frac{1}{2}\tilde{U}_2},\mathbf{\tilde{C}}),
\end{align*}
 the first-order derivative is considered as
\begin{align}\label{eqn:derivative}
\tilde{N}_1'(\zeta) = \tilde{S} -\frac{2}{\epsilon}\tilde{w}'(\zeta) - \frac{1}{2}{f}_1'(\frac{1}{2}\zeta,\frac{1}{2}\tilde{U}_2,\mathbf{\tilde{C}}).
\end{align}
Simplifying the second term, we have
$\tilde{w}'(\zeta) = \frac{3}{2}\zeta^2 - \frac{1}{2}$, and the boundedness for the above second term in $[-1,1]$ is deduced as
\begin{align}\label{eqn:2}
 -\frac{1}{2}\leq \tilde{w}'(\zeta)\leq 1.
\end{align}
Then we rewrite the third term
\begin{align*}
{f}_1'(\frac{1}{2}\zeta,\frac{1}{2}\tilde{U}_2,\mathbf{\tilde{C}}) &= \lambda\sum\limits_{r=1}^\omega \left[(\tilde{I}(r)-\mathbf{\tilde{C}}_{11}(r))^2 -(\tilde{I}(r)-\mathbf{\tilde{C}}_{21}(r))^2 \right]H_{\epsilon_1}(\frac{1}{2}\tilde{U}_2)\delta'_{\epsilon_1}(\frac{1}{2}\zeta) \\
&+ \left[(\tilde{I}(r)-\mathbf{\tilde{C}}_{12}(r))^2 - (\tilde{I}(r)-\mathbf{\tilde{C}}_{22}(r))^2 \right]\left(1-H_{\epsilon_1}(\frac{1}{2}\tilde{U}_2)\right)\delta'_{\epsilon_1}(\frac{1}{2}\zeta),
\end{align*}
Since the original image $\tilde{I}$ and intensity averages $\mathbf{\tilde{C}}$ is in interval $[-1,1]$, it holds
$$\abs{(\tilde{I}(r)-\mathbf{\tilde{C}}_{11}(r))^2 -(\tilde{I}(r)-\mathbf{\tilde{C}}_{21}(r))^2 }\leq 4.$$
Considering the fact that $\displaystyle H_{\epsilon_1}(\frac{1}{2}\tilde{U}_2)$ lies in the interval $[0,1]$,  we have
\begin{align}\label{eqn:3}
\abs{{f}_1'(\frac{1}{2}\zeta,\frac{1}{2}\tilde{U}_2,\mathbf{\tilde{C}})} \leq \lambda\sum\limits_{r=1}^\omega 8\cdot\frac{\pi}{2\epsilon_1^2}
=\frac{4\omega\lambda\pi}{\epsilon_1^2}.
\end{align}
Substituting \eqref{eqn:2} and \eqref{eqn:3} into \eqref{eqn:derivative}, it follows that
\begin{align*}
\tilde{N}'_1(\zeta) \geq \tilde{S} - \frac{2}{\epsilon} - \frac{2\omega\lambda\pi}{\epsilon_1^2}, \ \zeta \in [-1,1].
\end{align*}
If  $\tilde{S} \geq \gamma \triangleq \frac{2}{\epsilon} + \frac{2\omega\lambda\pi}{\epsilon_1^2}$, we have
$\tilde{N}'_1(\zeta) \geq 0, \zeta \in [-1,1]$, which shows that $\tilde{N}_1$ is a monotonically increasing function.
Provided that
\begin{align}
\epsilon_1 = \frac{1}{2p},\  p\ \text{is an positive odd number},
\end{align}
 the following equalities
 $$\tilde{N}_1(-1) = -\tilde{S}, \ \tilde{N}_1(1)= \tilde{S},$$
can be easily derived.
Considering the monotonicity of $\tilde{N}_1$, we can derive $\norm{\tilde{N}_1(\zeta)}_{\infty} \leq \tilde{S}$. The proof for $\tilde{N}_2$ follows the same line, we just omit it.
\end{proof}
}

With the estimate in Lemma \ref{lemma:bound}, we can deduce the discrete maximum bound
principle for ETD1 and ETDRK2 schemes of \eqref{t_ODE} under $\tilde{S}\geq\gamma$, referring to Theorems 3.4 and 3.5 in \cite{Li2019siam}.
That is, the numerical solution $\tilde{U_i}^k$ satisfies $\norm{\tilde{U}_i^k}_{\infty} \leq 1 $ if the initial
value have $\norm{\tilde{U}_i^0}_{\infty} \leq 1$. $\tilde{U}_i$ and $U_i$ also satisfy the projection relations in \eqref{eqn:proj}.
Thus, the numerical solution of the ETD scheme \eqref{ETD1} (or \eqref{ETD2}) for equations \eqref{ODE} preserves the discrete maximum bound principle,
as shown in the following theorem.

 \begin{theorem}\label{thm1}
 If the initial value satisfies $0\leq \mathbf{u}^0\leq 1$, then the numerical solution $\mathbf{U}^k$ of the ETD scheme \eqref{ETD1} (or
 \eqref{ETD2}) for \eqref{ODE} is also in the interval $[0,1]$, for all $k>0$, provided that the stabilizer $S$ satisfies
\begin{align}\label{bound_L}
S\geq\gamma.
\end{align}
\end{theorem}

Theorem \ref{thm1} demonstrates that for fixed $\mathbf{C}$, each $U_i^k$ can be kept in $[0,1]$, which causes
the interval which $\mathbf{C}$ belong to is invariant, after updating $\mathbf{C}$ in Step 2 of Algorithm \ref{ACCV}.
It guarantees that the discrete maximum bound principle always holds for the whole iteration.

\subsubsection{Discrete energy stability}
For given $\mathbf{C}$, the Allen--Cahn equations \eqref{MAC_1} and \eqref{MAC_2} are energy stable based on \eqref{eqn:energy_4_phase}, {viz.,}
\begin{align}
\frac{\mathrm{d}E}{\mathrm{d}t} \leq 0.
\end{align}
Subsequently, we will prove that the ETD scheme \eqref{ETD1} or
 \eqref{ETD2} for \eqref{ODE} has discrete energy stability in the sense that the following discrete energy is non-increase in time:
\begin{align}\label{eqn:discrete_E}
E_h(\mathbf{U},\mathbf{C})=\sum_{\ell=1}^{M_1M_2} \left(\sum_{i=1}^{2} (\frac{1}{\epsilon}W(U_{i,\ell}))+F_{\epsilon_1}(U_{1,\ell},
U_{2,\ell},I)\right)- \sum_{i=1}^{2} {\epsilon}U_{i}^TD_hU_{i}. 
\end{align}

\begin{lemma}\label{lemma_1}
For the given intensity average $\mathbf{C}$, the unconditional energy stability of the ETD scheme  \eqref{ETD1} (or  \eqref{ETD2}) for \eqref{ODE}
holds in the following form provided that ${S}\geq \frac{\gamma}{2}$
{
\begin{align*}
 E_h(\mathbf{U}^{k+1},\mathbf{C})\leq E_h(\mathbf{U}^k,\mathbf{C}).
\end{align*}
}
\end{lemma}

The proof for the ETD1 scheme \eqref{ETD1} can be obtained with a similar argument in \cite{Li2019siam}. For the ETDRK2 scheme \eqref{ETD2},
the proof in \cite{Li2019siam} is not optimal, and we can adopt a similar proof in \cite{Fu2022} to get the unconditional energy stability
in Lemma \ref{lemma_1}. In both proofs, the vital step lies in the boundedness of $\norm{F''_{\epsilon_1}}$ in \eqref{eqn:discrete_E}, which
can be proved by the boundedness of numerical solutions guaranteed by Theorem \ref{thm1} and estimates on the regularized
Heaviside function $H_{\epsilon_1}(u)$ and the regularized Dirac function $\delta_{\epsilon_1}(u)$. The proofs are quite similar to those in
\cite{Li2019siam} and \cite{Fu2022}. So we omit the proof of Lemma \ref{lemma_1}.

\begin{theorem}\label{thm2}
The iterative process \eqref{eqn:iter1}-\eqref{eqn:iter3} holds the unconditional energy stability under the condition ${S}\geq \frac{\gamma}{2}$,
{viz.,}
\begin{align}\label{energy_stable}
 E_h(\mathbf{U}^{k+1},\mathbf{C}^{k+1}) \leq E_h(\mathbf{U}^{k},\mathbf{C}^{k}).
\end{align}
\end{theorem}
{
\begin{proof}
For the first stage  \eqref{eqn:iter1}, Lemma \ref{lemma_1} derives
\begin{align*}
E_h(\mathbf{U}^{k+1},\mathbf{C}^{k})\leq E_h(\mathbf{U}^k,\mathbf{C}^{k}).
\end{align*}
And then the second stage \eqref{eqn:iter3} implies
\begin{align*}
E_h(\mathbf{U}^{k+1},\mathbf{C}^{k+1})\leq E_h(\mathbf{U}^{k+1},\mathbf{C}^{k}).
\end{align*}
Combining above two inequalities, we can get \eqref{energy_stable}.
\end{proof}

\begin{remark}
The maximum bound principle in Theorem \ref{thm1} and unconditional energy stability in Theorem \ref{thm2} are both based on two phase variables, four-phase segmentation. For the case of more than two phase variables, Theorems \ref{thm1} and \ref{thm2} can be derived without any special treatments; only $\gamma$ in Lemma \ref{lemma:bound} needs to be re-evaluated.
\end{remark}
}

\section{Numerical experiments}\label{sec:numerical}
Through a variety of experiments, this section exhibits the capability
and efficiency of the proposed ACCV model and the developed ADMM-ETD solver in the multi-phase image segmentation.
The discrete maximum bound of the solution and energy stability of ETD1 and ETDRK2 schemes are presented,
which are consistent with the theoretical results in
Theorems \ref {thm1} and \ref{thm2}. The comparisons between the { state-of-the-art} segmentation methods and our proposed
method are stated to display the effectiveness of our ACCV model for the multi-phase segmentation.
All experiments in this part were performed on a Windows laptop system with an Intel Core i5 processor,
2.11-GHz CPU and 16GB RAM, and were programmed in MATLAB R2019a.


\subsection{Comparison between the ETD1 and ETDRK2 schemes }
This subsection is devoted to testing a synthetic grayscale image whose size is $[1024,1024]$. Without affecting the final segmentation results, we rescale this image into $[300,300]$ to speed up the process of segmentation. The edge detection by the Multi-IGLIM gives the initial contours of $U^0_1,U^0_2$ in the first row of Figure \ref{fig:circle1024}. The second row shows the final contour results at interfaces and the color segmentation results, respectively. Wherein the contour result uses the original image as the background, and then we plot the interface of $U_1$ and $U_2$ by green and red to distinguish. The color segmentation result is obtained by marking the disjoint regions $\{ U_1 > \frac{1}{2},U_2 > \frac{1}{2} \},  \{ U_1 \leq \frac{1}{2},U_2 > \frac{1}{2} \},\{ U_1 > \frac{1}{2},U_2 \leq \frac{1}{2} \},\{ U_1 \leq \frac{1}{2},U_2 \leq \frac{1}{2} \}$ by different colors. More specifically, the green curve recognizes the ``circle" and ``heart", and the red one identifies the ``heart" and ``arch" in contour results (c).  It assumed that $U_1>\frac{1}{2}$ holds inside the green curve, otherwise $U_1 \leq \frac{1}{2}$, similar case for $U_2$ and the red one.

Table \ref{tab:CPU_time_circle1024} contains the comparison for ETD1 and ETDRK2 schemes, from which we verify that the second-order scheme is more efficient. 
As a result, we will adopt the ETDRK2 to do the multi-phase segmentation in what follows. In view of stability, the discrete maximum bound principle holds for variables $U_1$ and $U_2$. The energy evolution of the two temporal discretizations sketched in Figure \ref{fig:circle1024_energy} reveals the energy stability. The relevant parameters are given by $\epsilon = 4.0, \lambda = 40, h = 1$. 

\begin{figure}[!t]
\setlength{\fboxrule}{0.3pt}
  \setlength{\fboxsep}{0.3pt}
\centering
		\subfigure[The initial value $U^0_1$] {
			\fbox{\includegraphics[width=0.3\columnwidth,width=0.3\columnwidth]{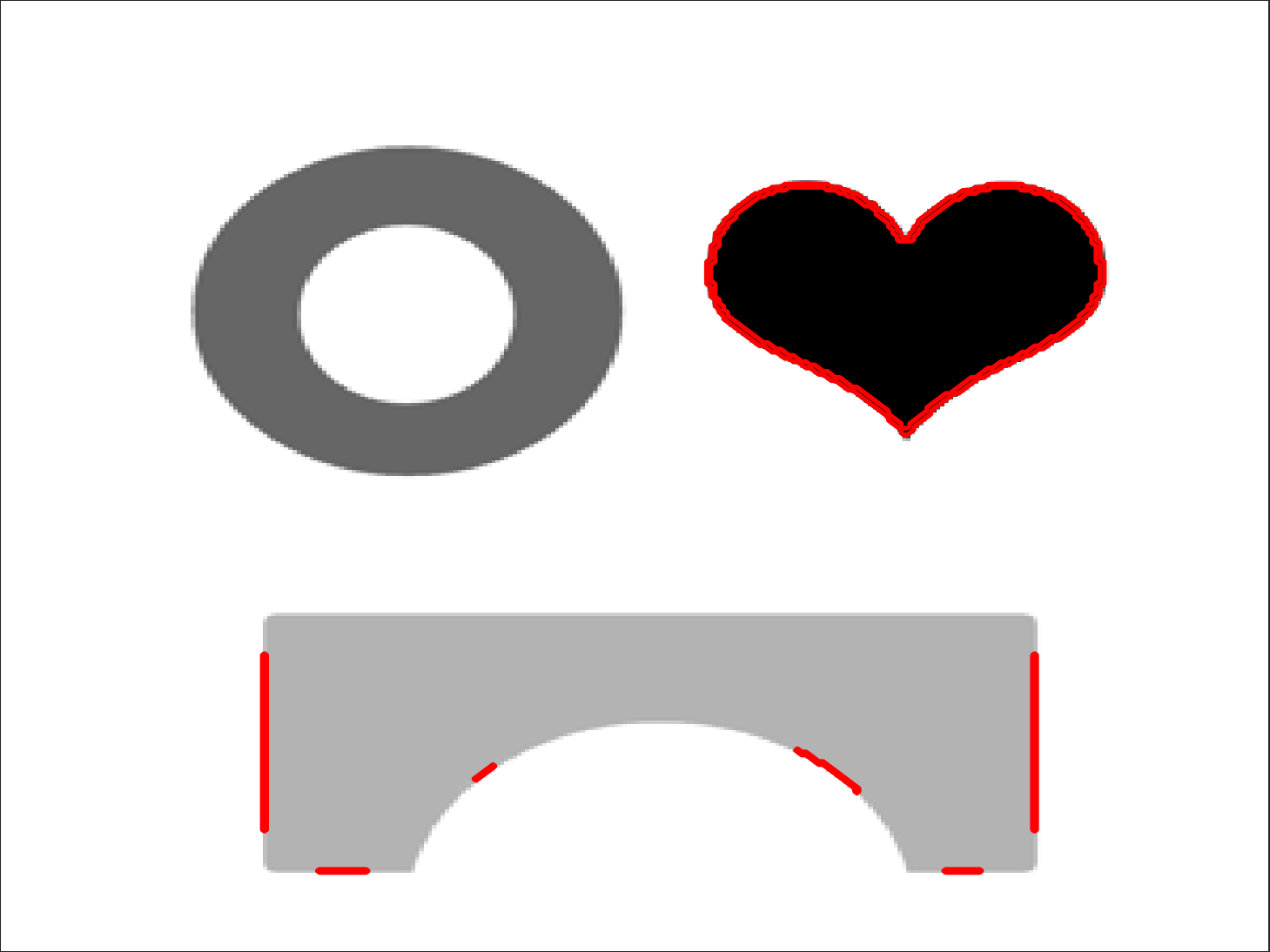}}
		}\hspace{5mm}
		\subfigure[The initial value $U^0_2$] {
			\fbox{\includegraphics[width=0.3\columnwidth,width=0.3\columnwidth]{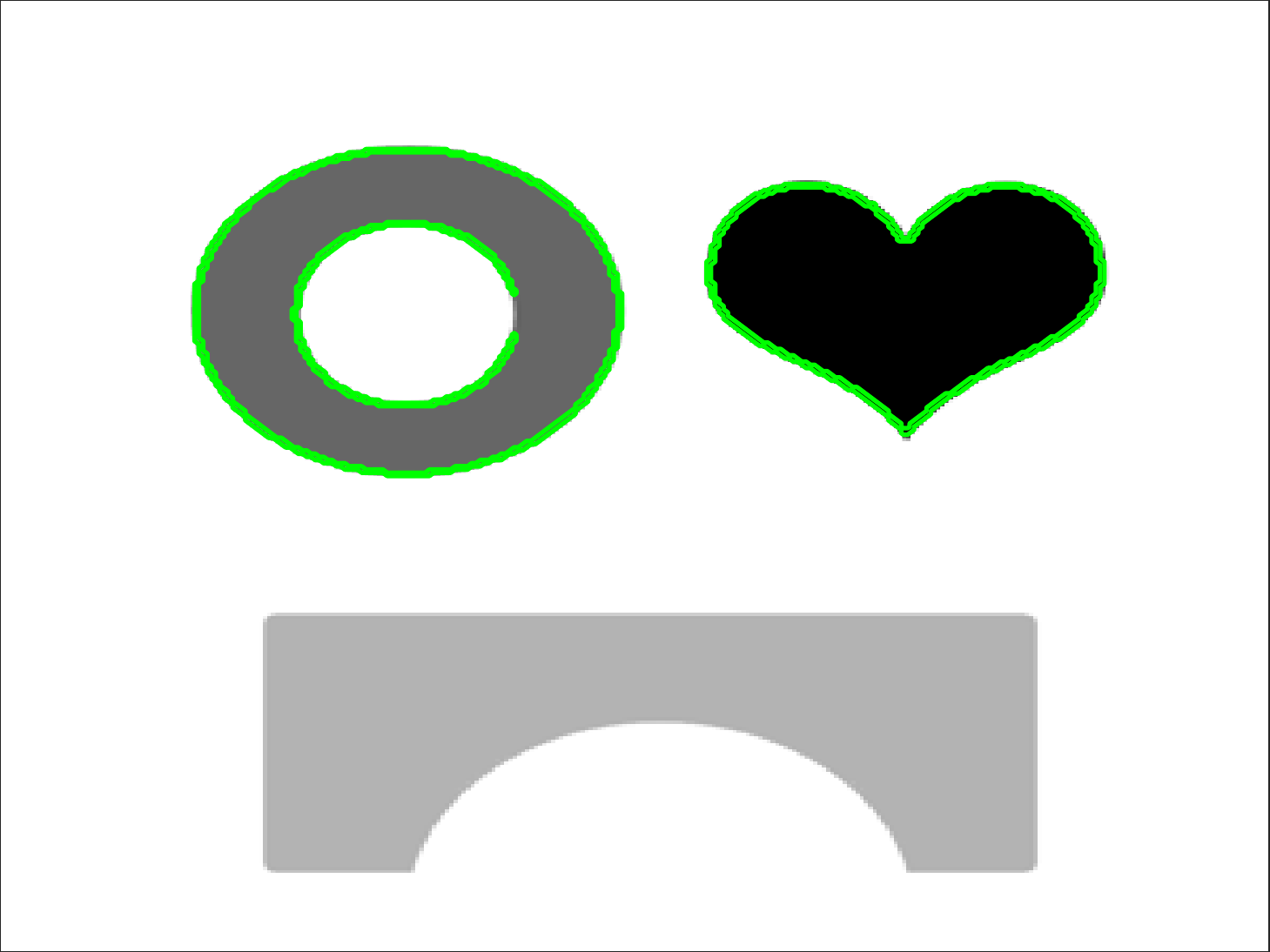}}
		}
		\subfigure[The contour results] {
			\fbox{\includegraphics[width=0.3\columnwidth,width=0.3\columnwidth]{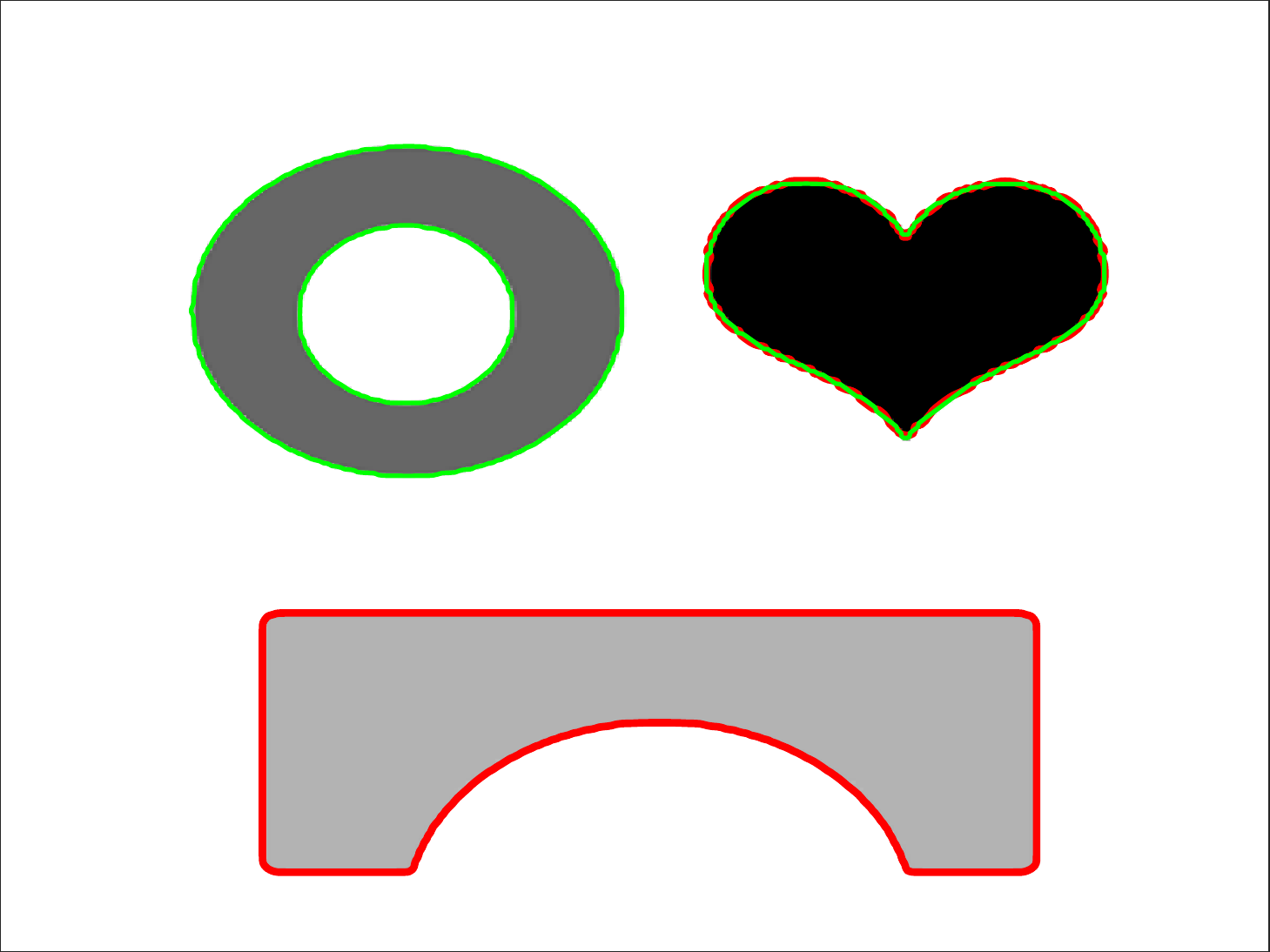}}
		}\hspace{5mm}
		\subfigure[The segmentation result] {
			\fbox{\includegraphics[width=0.3\columnwidth,width=0.3\columnwidth]{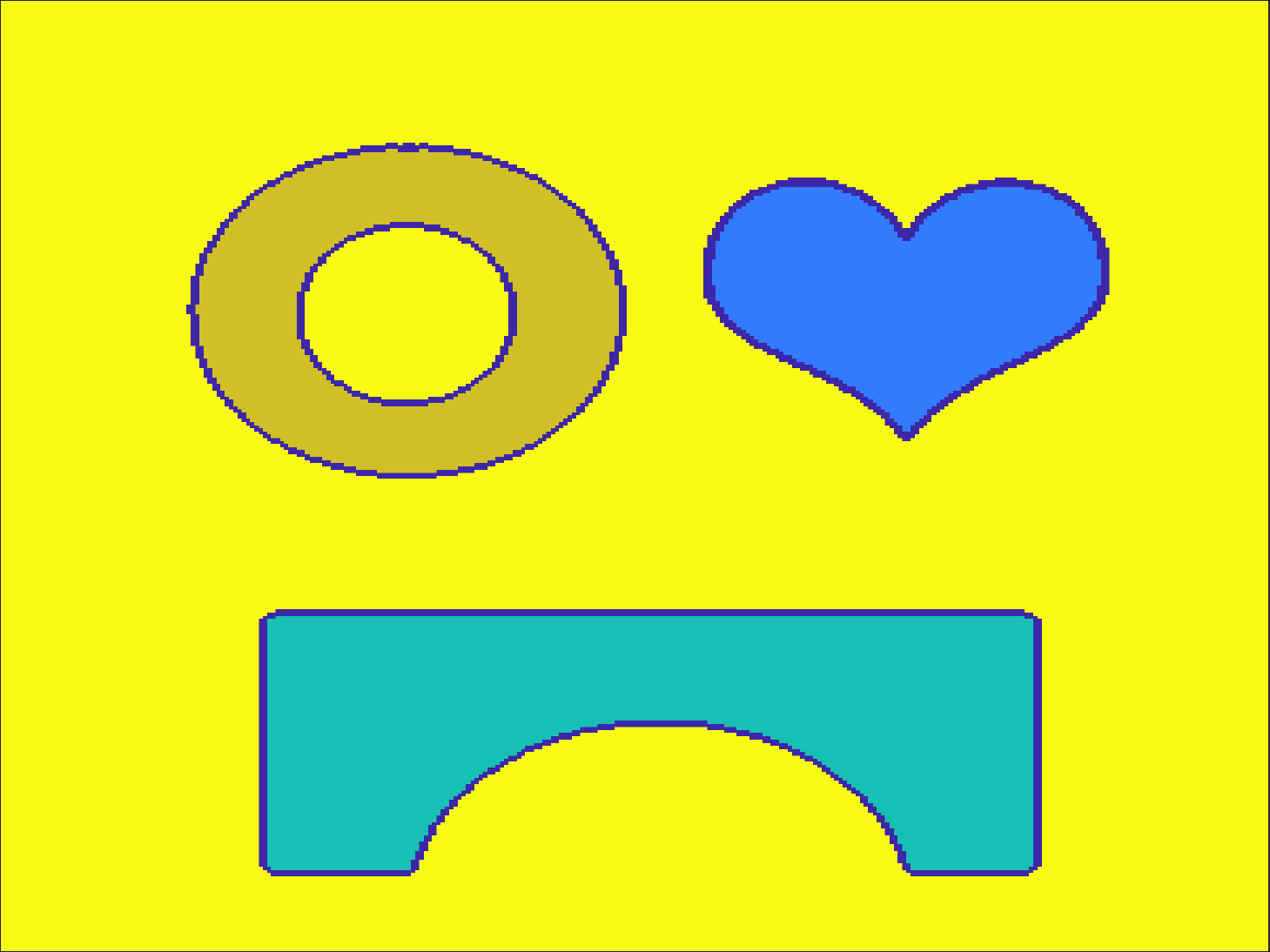}}
		}
			 \caption{ \label{fig:circle1024} Comparison between ETD1 and ETDRK2 schemes. (a) and (b): Initial values $U^0_1$ and $U^0_2$; (c) The final contours; (d) The segmentation results.   }
	\end{figure}

\begin{table}[!t]
\begin{tabular}{cccccccc}
\hline
            Scheme       &  Time step   & Iteration    &  CPU time(s)  & \multicolumn{2}{c}{1-Max} &   \multicolumn{2}{c}{Min} \\
\cmidrule(lr){5-6}\cmidrule(lr){7-8}
                                 &                      &        &                         &$U_1$                         &$U_2$  &$U_1$&$U_2$ \\
           ETD1          & $\Delta t = 0.3$   &  65     &   { 6.35}  & { 1.88e-2}&{ 3.64e-3}  &{ 1.47e-2} &{ 4.51e-3}  \\
           ETDRK2        & $\Delta t = 0.3$   &  36    &    { 4.76}    & { 1.54e-2}&{ 3.20e-3}  &{ 1.29e-2} &{ 4.14e-3 }\\\hline  
\end{tabular}
\caption{\label{tab:CPU_time_circle1024}  Comparison between ETD1 and ETDRK2 schemes for the synthetic grayscale image.} 
\end{table}

%
%
%
%

\begin{figure}[!htp] \centering
		\subfigure[ETD1] {
			\includegraphics[width=0.4\columnwidth]{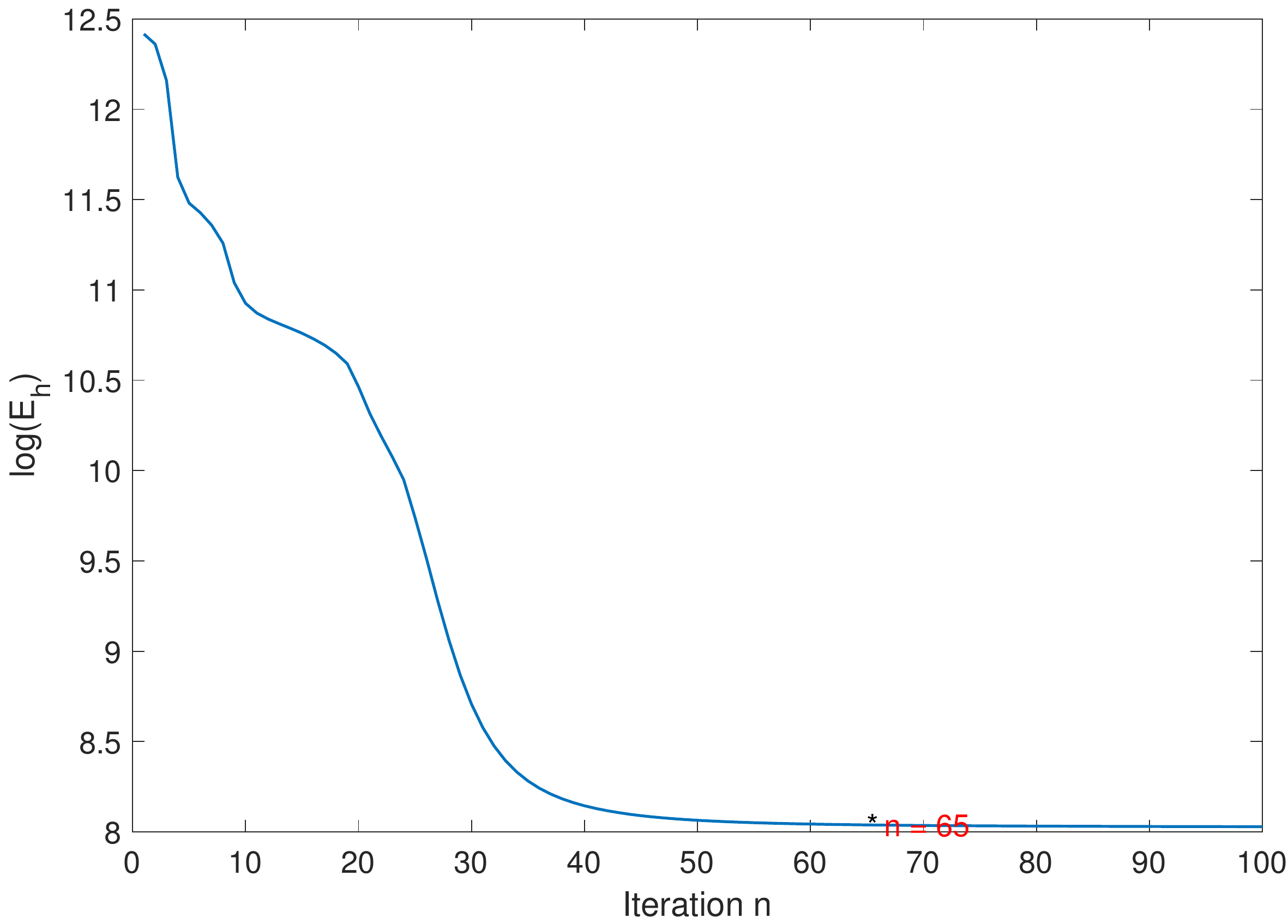}
		}
		\subfigure[ETDRK2] {
			\includegraphics[width=0.4\columnwidth]{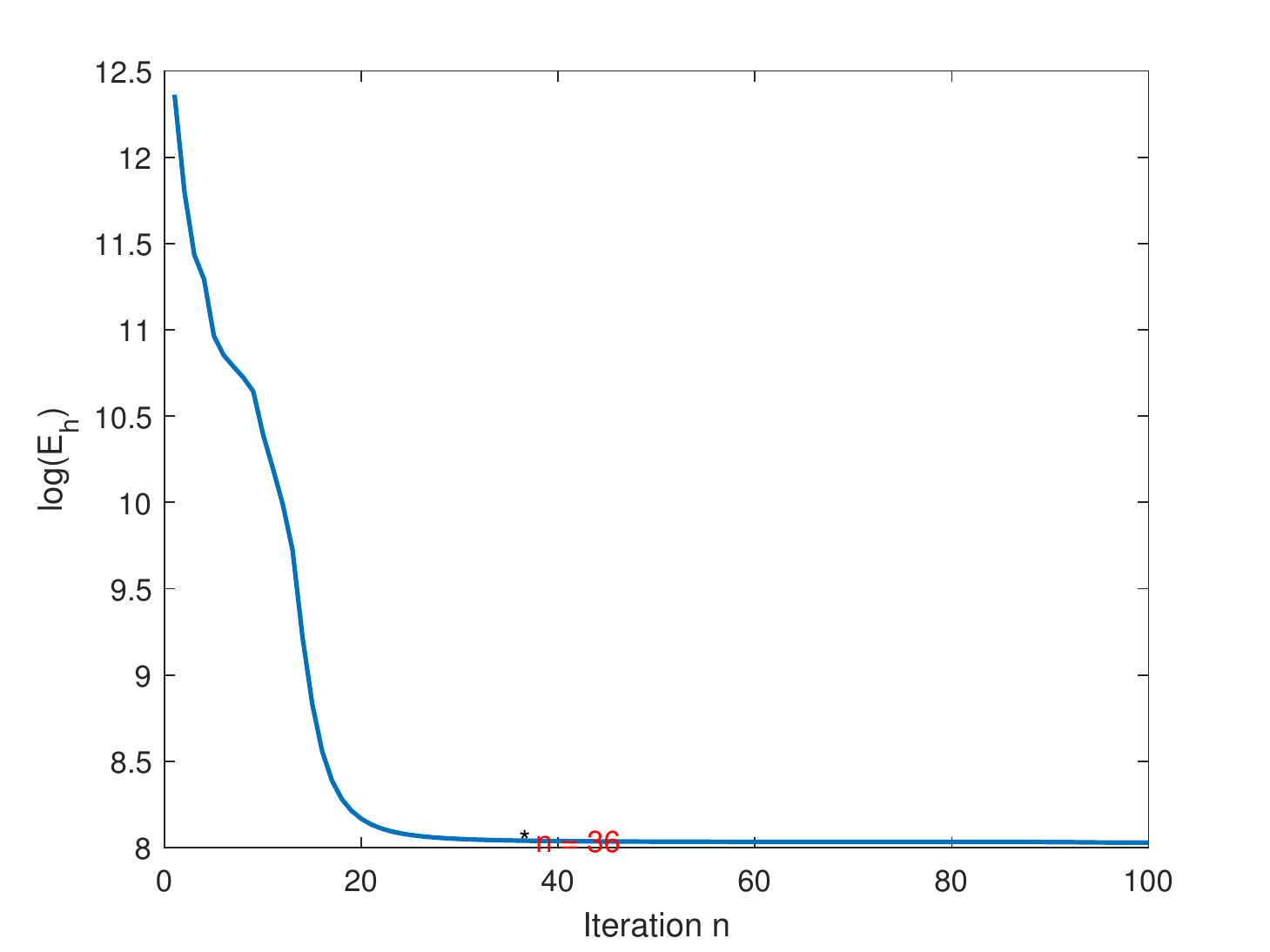}
		}
		\caption{\label{fig:circle1024_energy} The energy evolution of the two time discretizations. }
\end{figure}

\subsection{Experiments on multi-phase color images}
In this part, we carry out multi-phase image segmentations for some color images, including 4-phase and 6-phase segmentations. The first four images in Figure \ref{fig:subsec2} are 4-phase segmentation. The first is a color version of the synthetic grayscale image in the last example. The second one is the image ``flowers" with two flowers in red and yellow, stalks in green, sky in blue. The remaining two are extracted from the image dataset ``white blood cell", for which we can segment the white blood cell and its nucleus from the around red blood cell. For the 4-phase image, Figure \ref{fig:subsec2} states the original image, the initial contour $U^0_1$ and $U^0_2$, the contour results, and the color segmentation result from left to right column, respectively. For the 6-phase image, we give the original image, the initial contour $U^0_1$, $U^0_2$ and $U^0_3$, and the color segmentation result, respectively. Notably, 6-phase segmentation is beyond the capability of two phase variables. Hence we need three phase variables, that is, three Allen--Cahn equations. There will appear some vacuum segments which does not influence the segmented results.

	The relevant parameters change along with the images, which are listed in Table \ref{tab:subsec2}. First, $\kappa$ in multi-phase inhomogeneous graph Laplacian operator \eqref{eqn:cl} is fixed as 50, an independent parameter which can not change along with the image. The threshold value $\sigma$ and the denoising number $M$ need to be adjusted through trial and experience. We conclude that $0 \leq \sigma \leq 0.05, M \in \{1,5,10\}$ in most cases. {Actually, the choices for the relevant parameters are selected balancing the segmentation result and cost time. For instance, as long as $\sigma \in [0.02,0.05]$, an ideal segmentation will be obtained for the image``flowers". Since the case of $\sigma = 0.05$ costs less time, we call the choice of parameter $\sigma$ is optimal and recorded in Table \ref{tab:subsec2}}. On the other hand, the parameters for the Allen--Cahn equations contain the diffusion parameter $\epsilon$, the weighted coefficient $\lambda$ for the fitting term, the spatial discretization step $h$, and the stabilizer $S$. We notice that $\epsilon, \lambda, h$ are similiar for 4-phase image segmentation. It does not mean that this choice is suitable and optimal for all images. The optimal choice for the relevant parameters varies according to the image. The stabilizer $S$ has been proved an upper bound in Lemma \ref{lemma:bound}, whereas we can narrow down the range of $S$ in practical, as we have done in Table \ref{tab:subsec2}.

\begin{figure}[!t]
\setlength{\fboxrule}{0.3pt}
  \setlength{\fboxsep}{0.3pt}
\resizebox{\textwidth}{!}{\hspace{11.8mm}
	\subfigure{
	\begin{minipage}[c]{0.2\textwidth}
   \fbox{\includegraphics[width=2.2cm,height = 1.8cm]{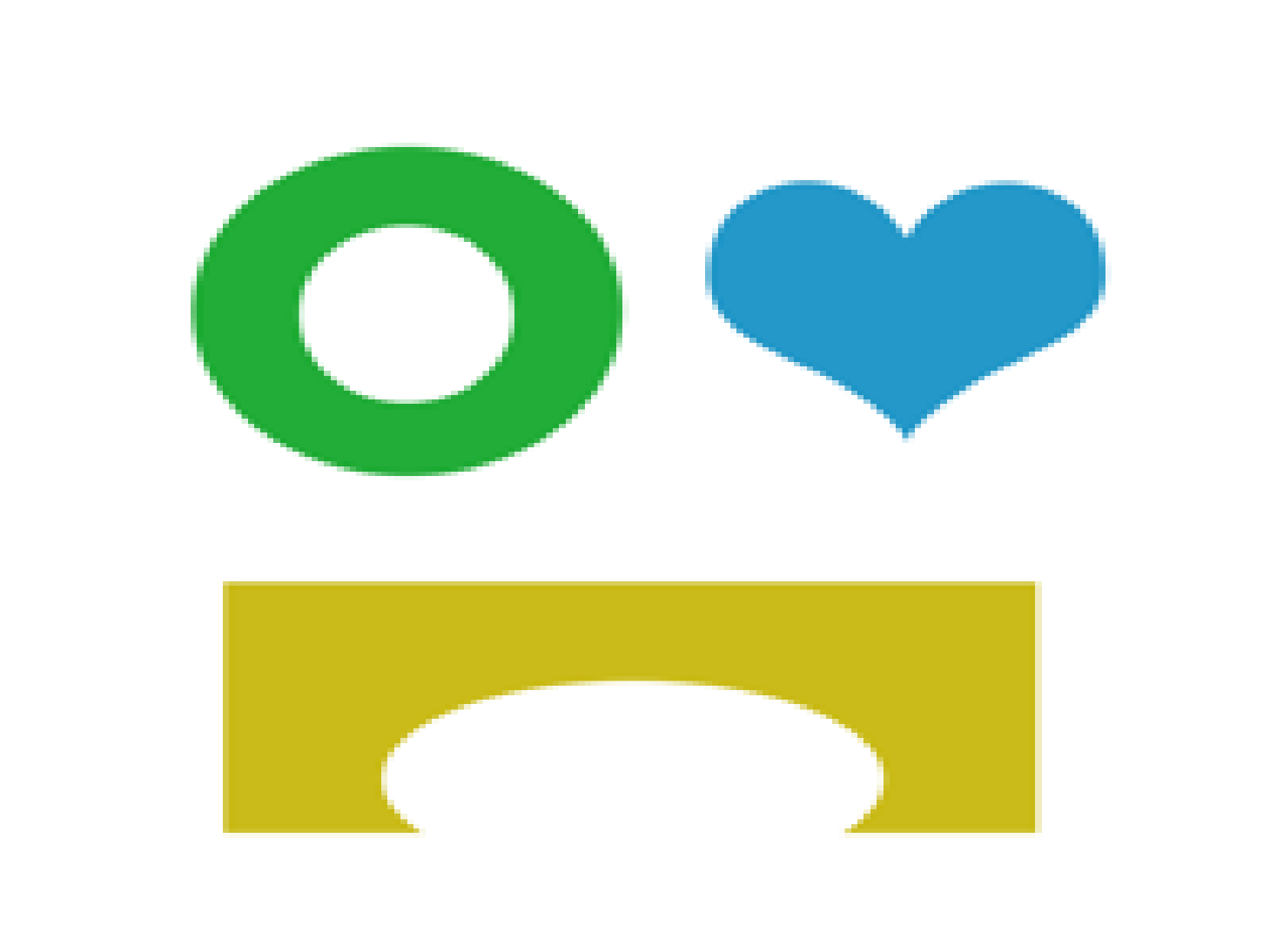}}\vspace{2mm}
    \fbox{\includegraphics[width=2.2cm,height = 1.8cm]{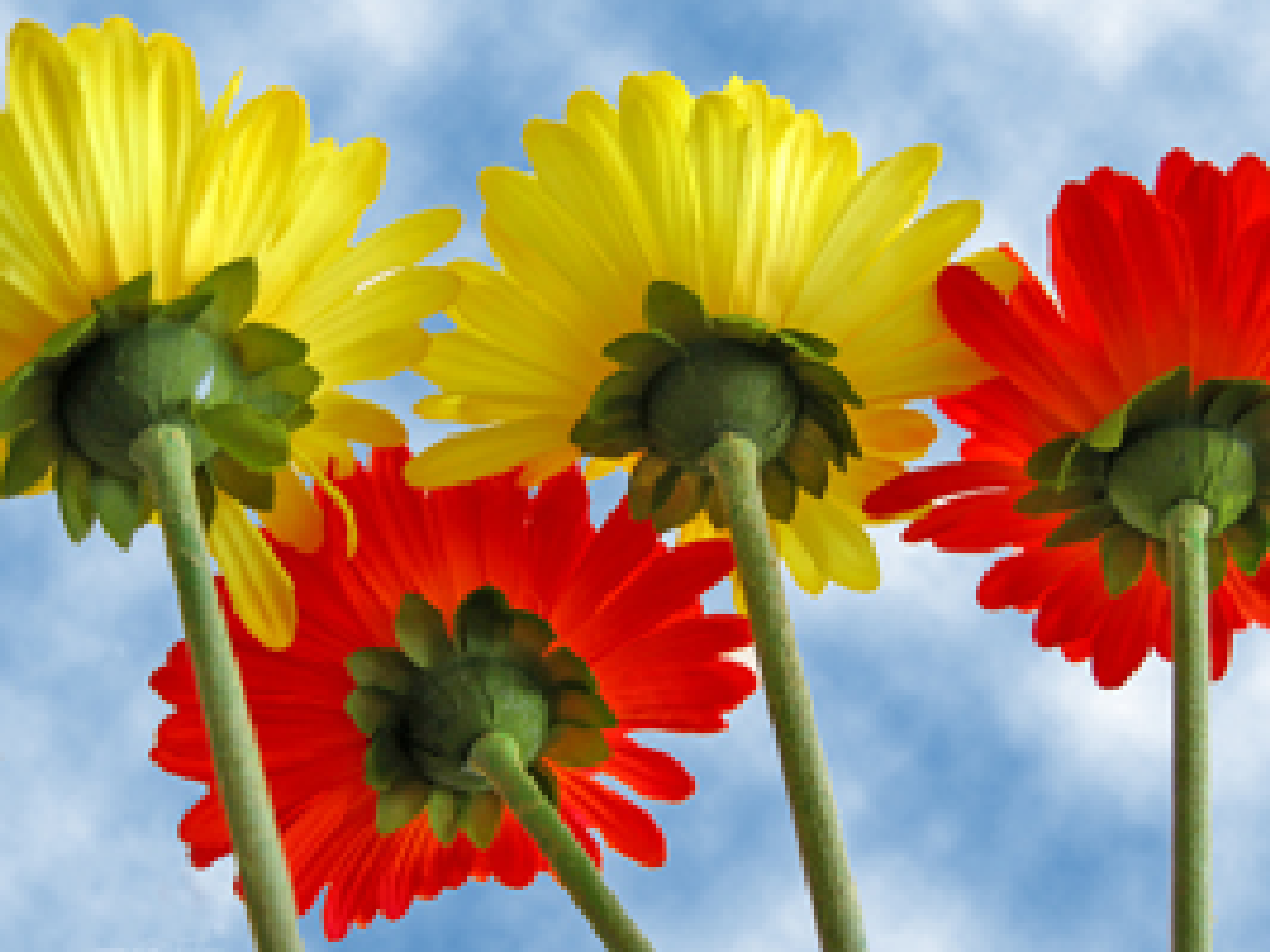}}\vspace{2mm}
    \fbox{\includegraphics[width=2.2cm,height = 1.8cm]{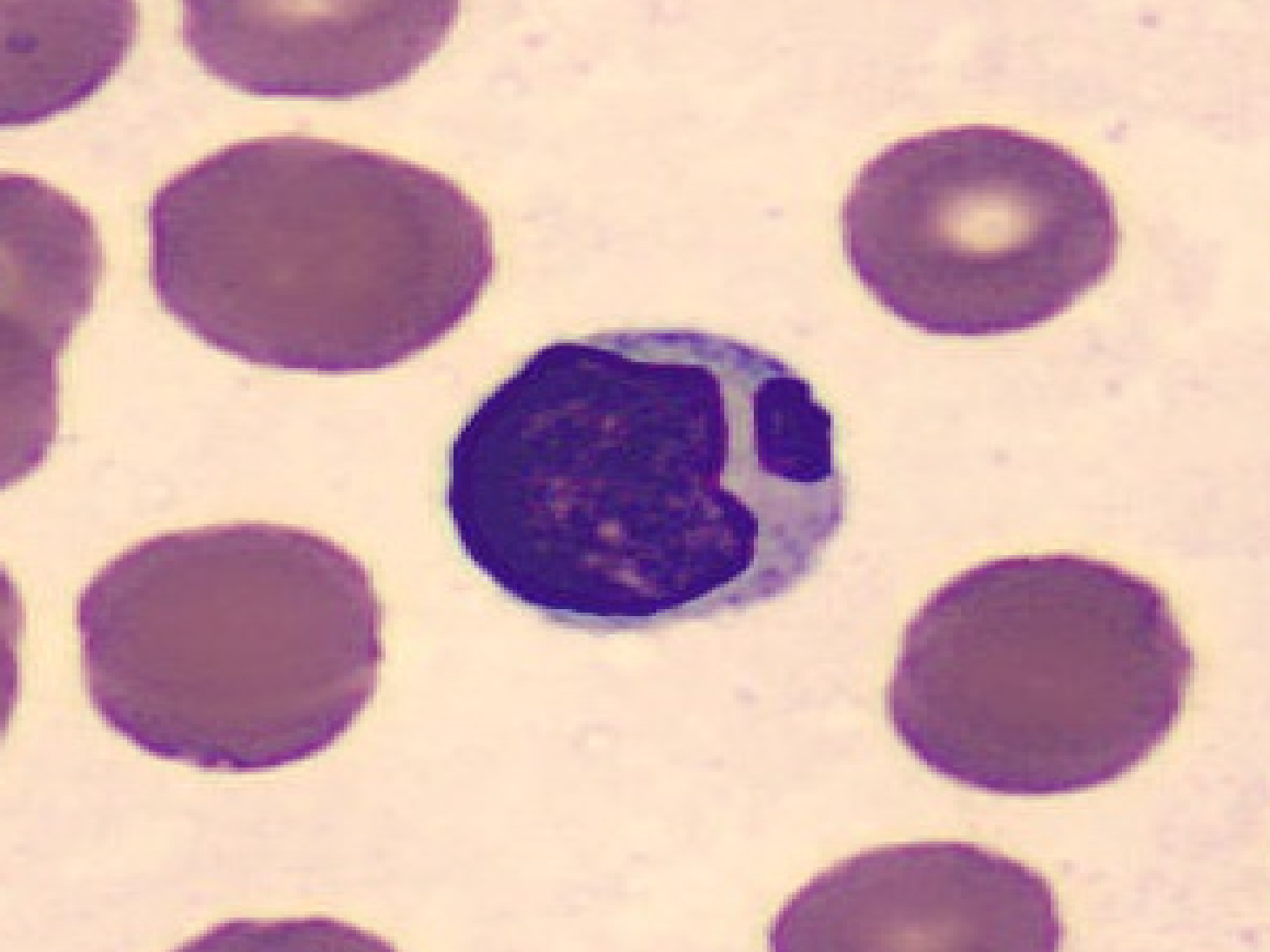}}\vspace{2mm}
	\fbox{\includegraphics[width=2.2cm,height = 1.8cm]{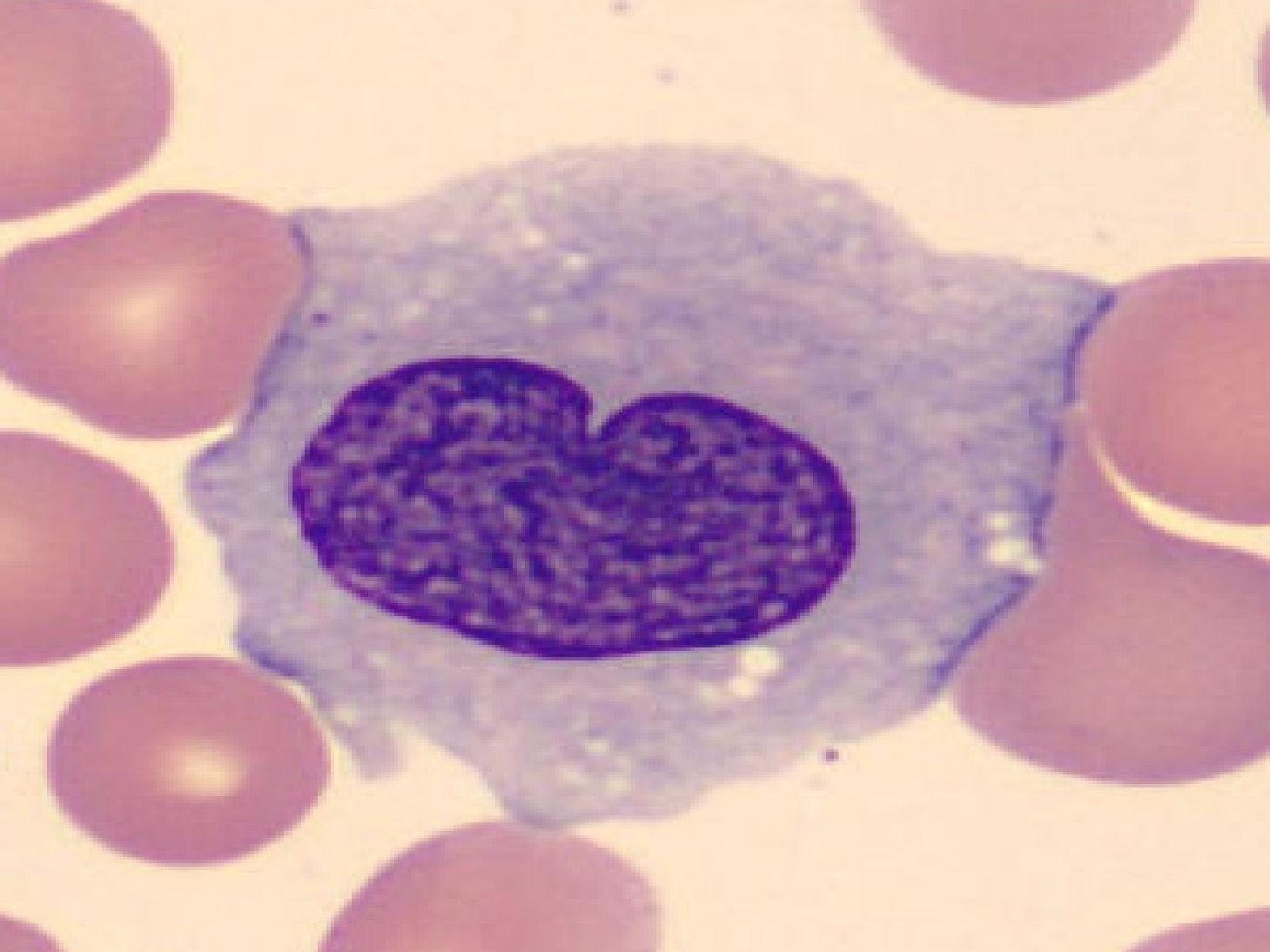}}\vspace{2mm}
   \fbox{\includegraphics[width=2.2cm,height = 1.8cm]{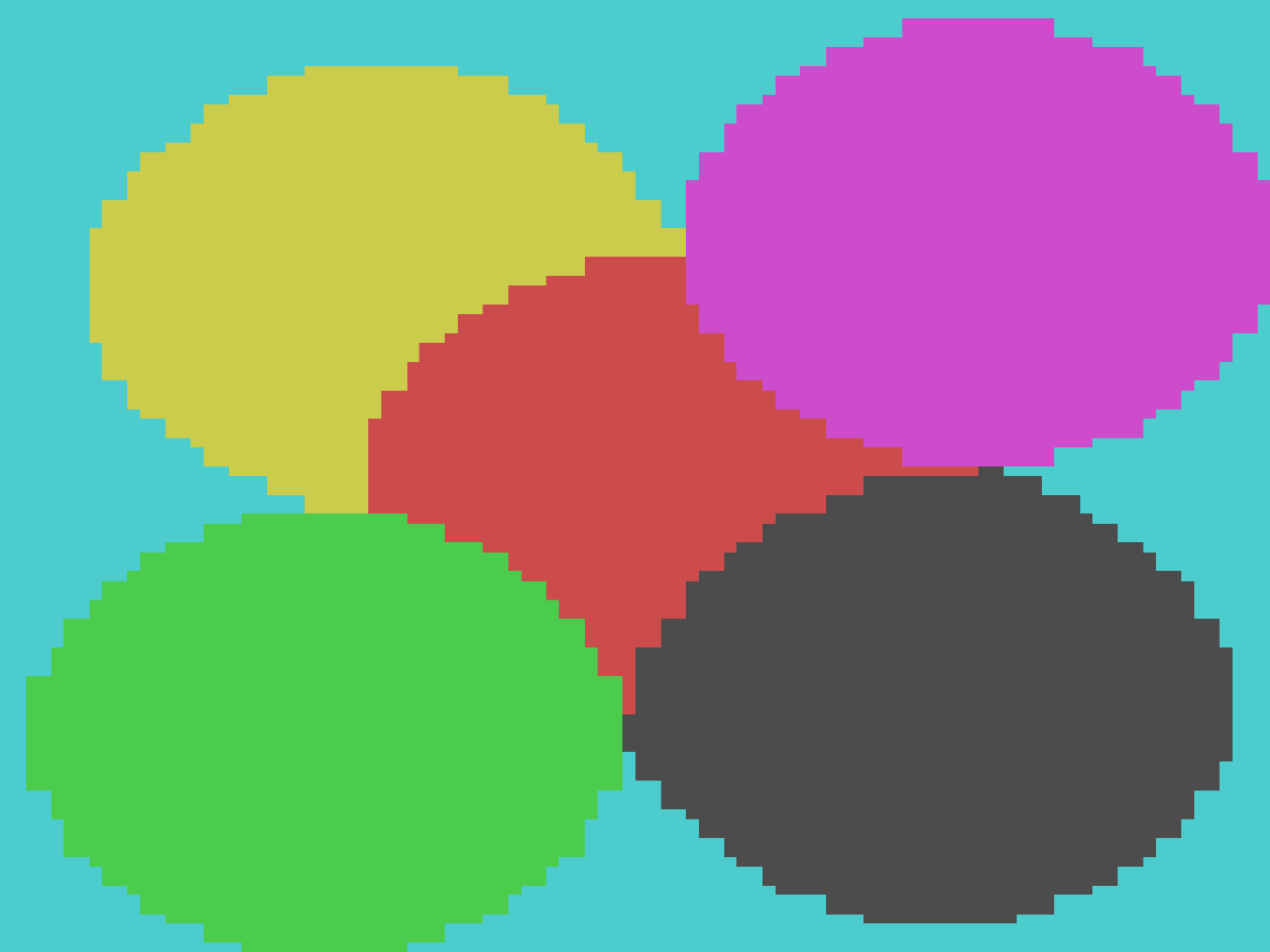}}\vspace{2mm}
	\end{minipage}
	}
  \hspace{-11.8mm}
	\subfigure{
	\begin{minipage}[c]{0.2\textwidth}
    \fbox{\includegraphics[width=2.2cm,height = 1.8cm]{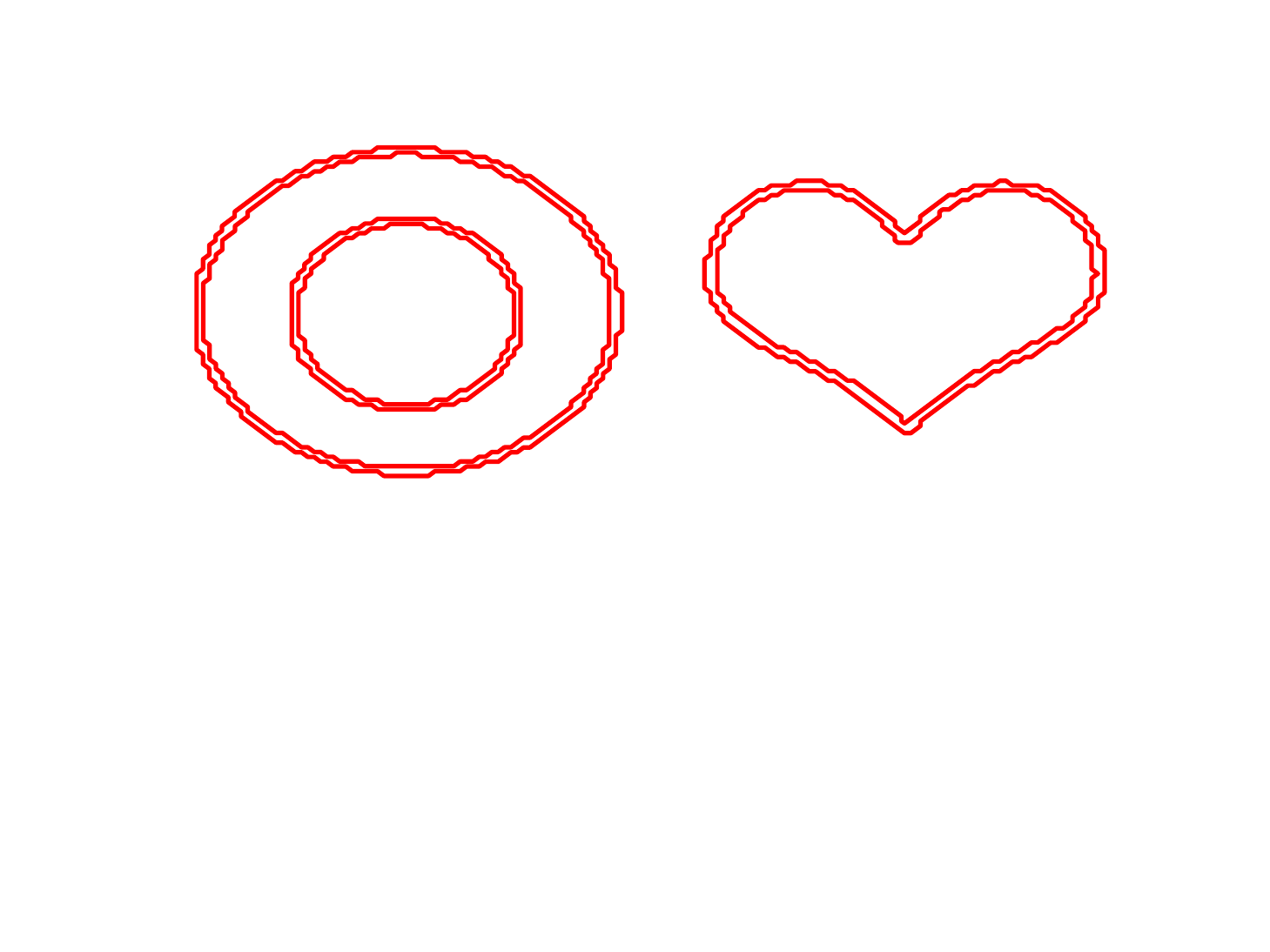}}\vspace{2mm}
    \fbox{\includegraphics[width=2.2cm,height = 1.8cm]{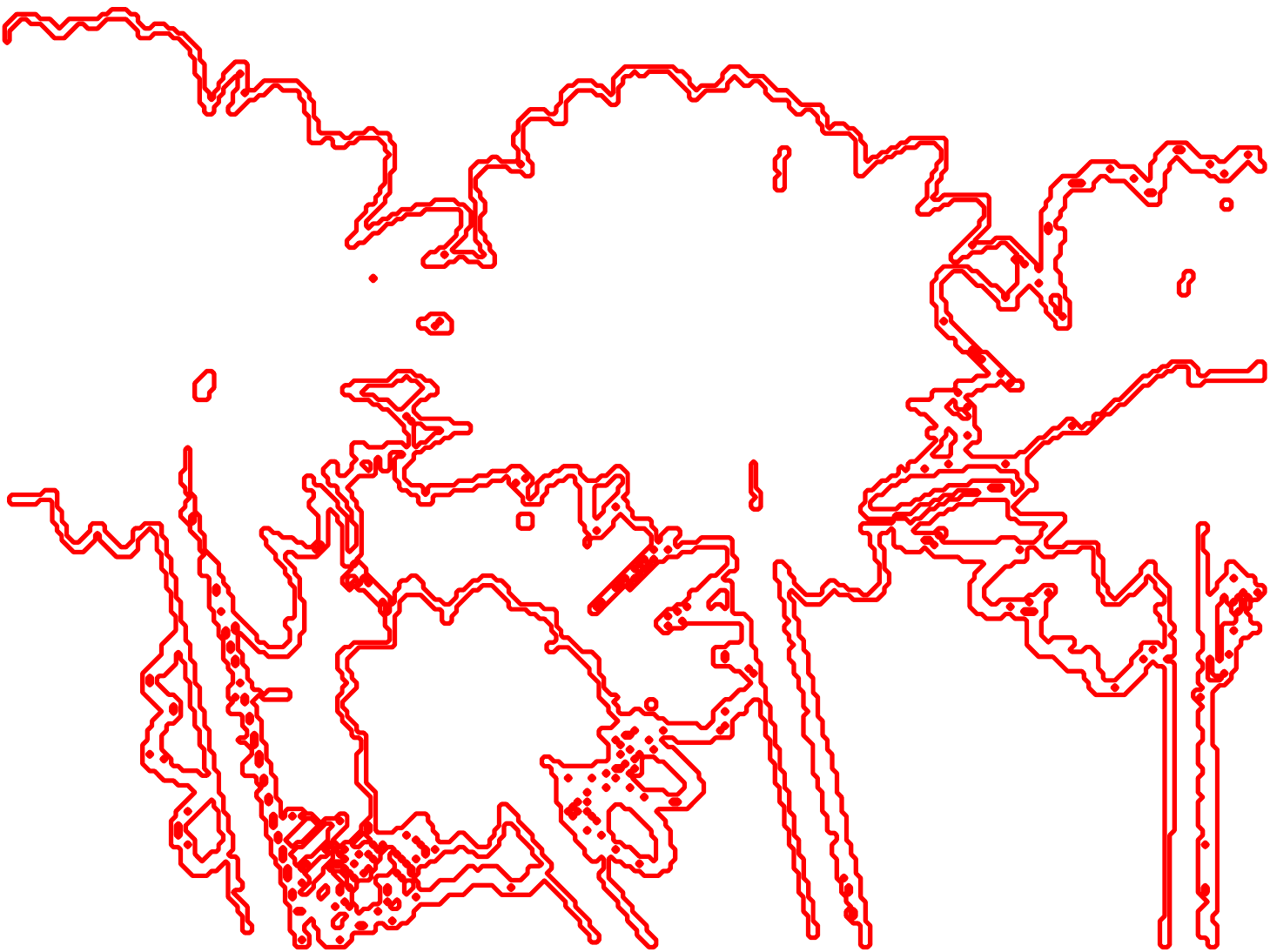}}\vspace{2mm}
   \fbox{\includegraphics[width=2.2cm,height = 1.8cm]{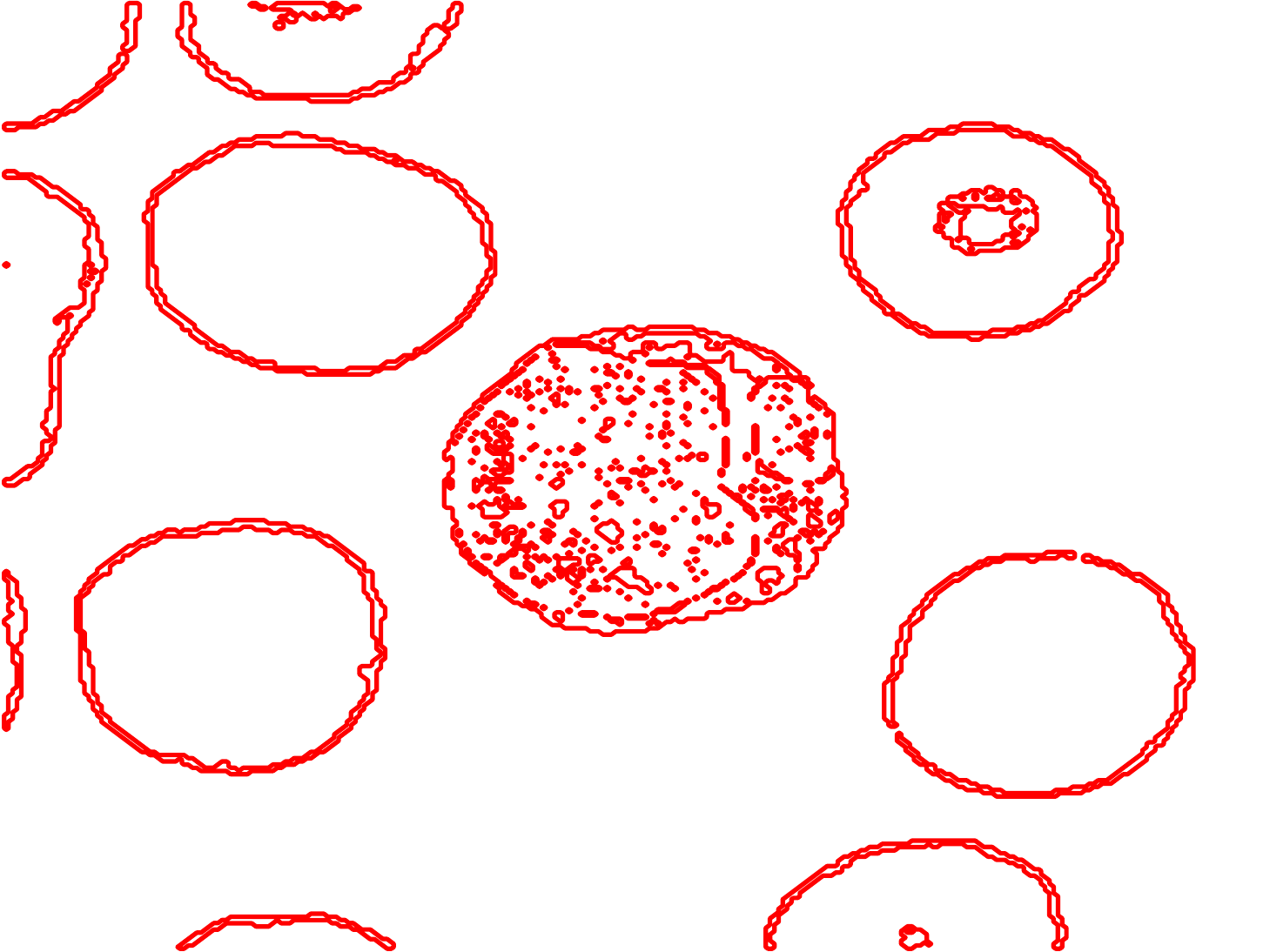}}\vspace{2mm}
    \fbox{\includegraphics[width=2.2cm,height = 1.8cm]{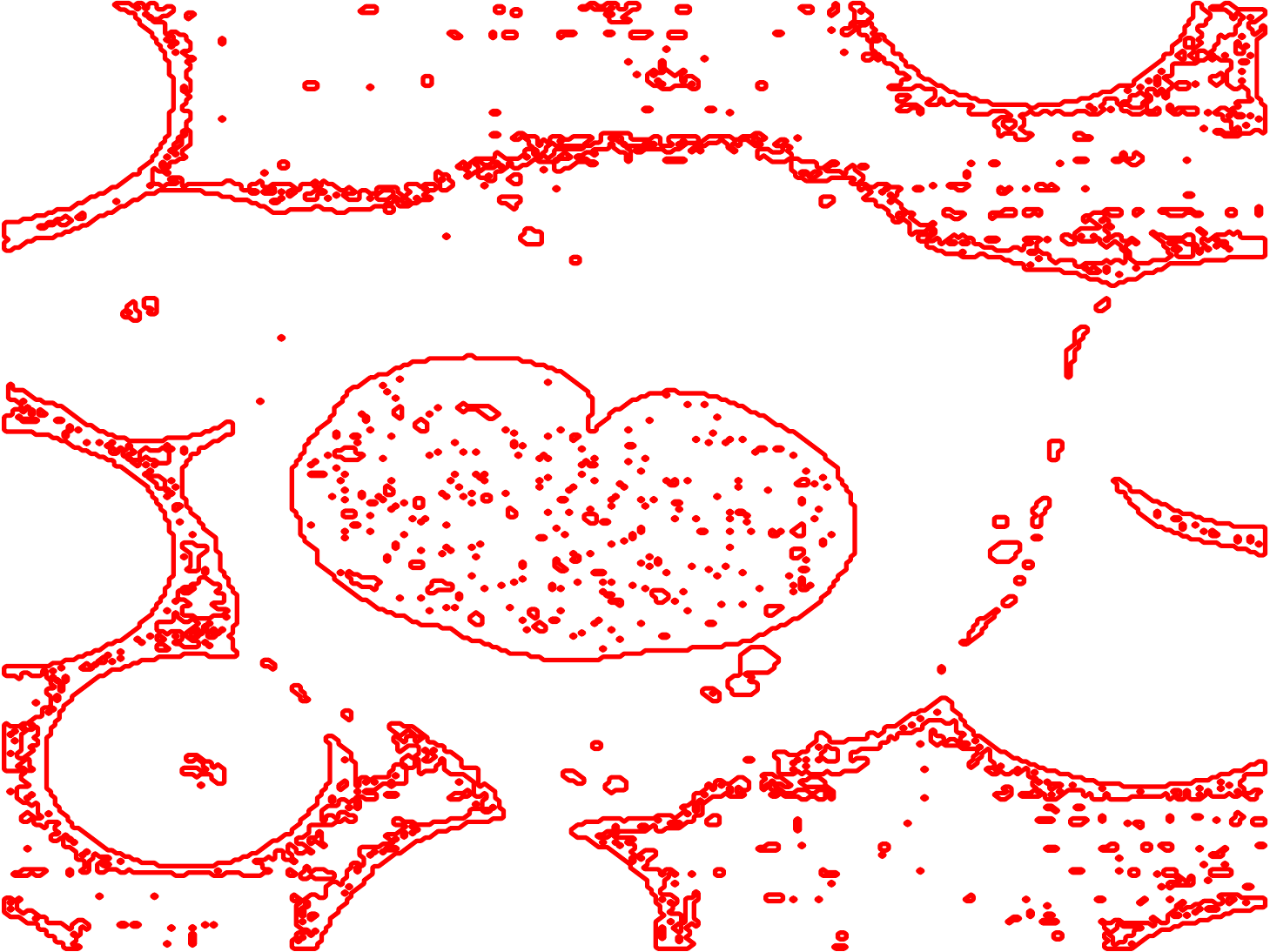}}\vspace{2mm}
    \fbox{\includegraphics[width=2.2cm,height = 1.8cm]{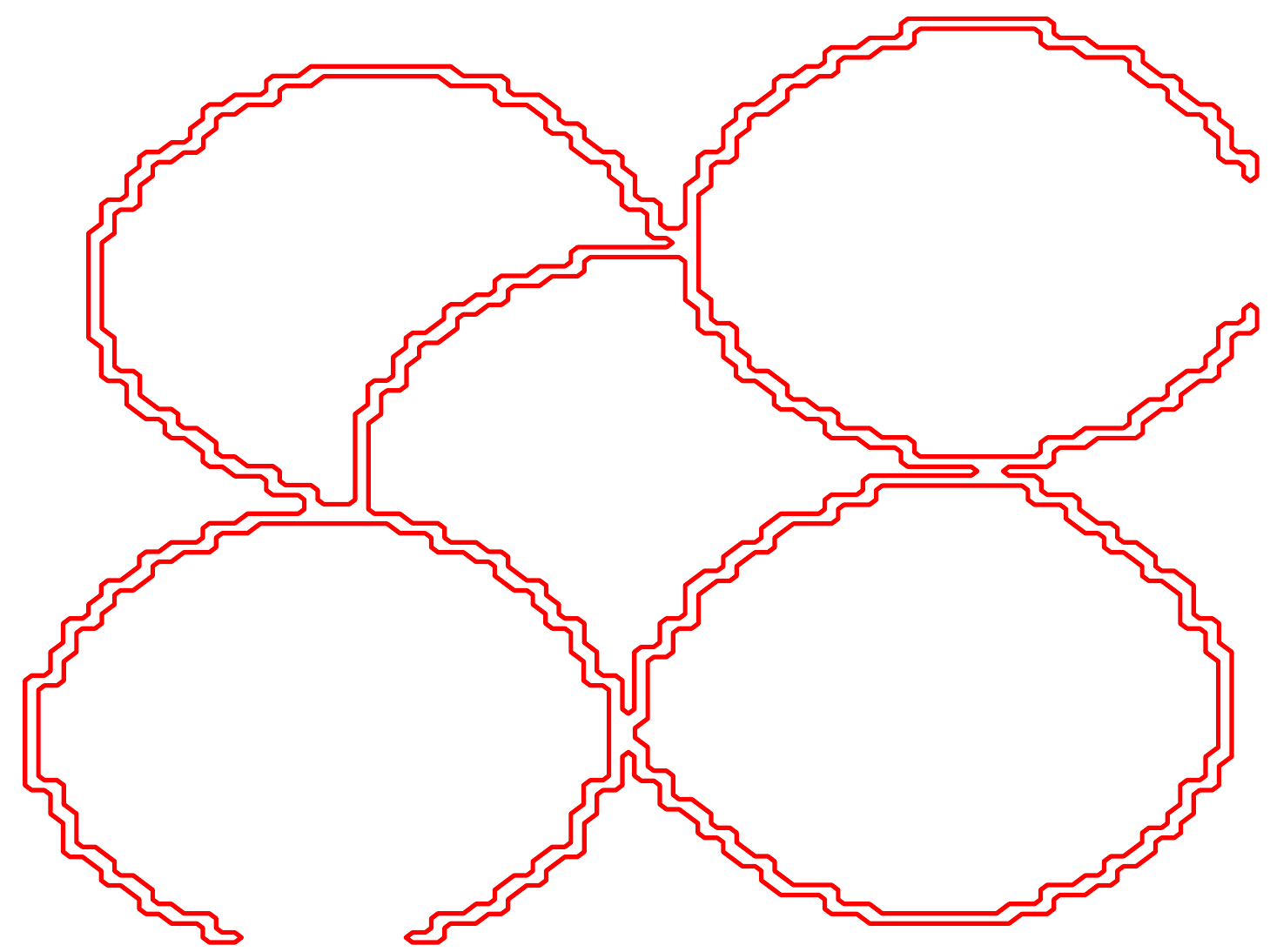}}\vspace{2mm}
  \end{minipage}
	}
  \hspace{-11.8mm}
	\subfigure{
	\begin{minipage}[c]{0.2\textwidth}
      \fbox{\includegraphics[width=2.2cm,height = 1.8cm]{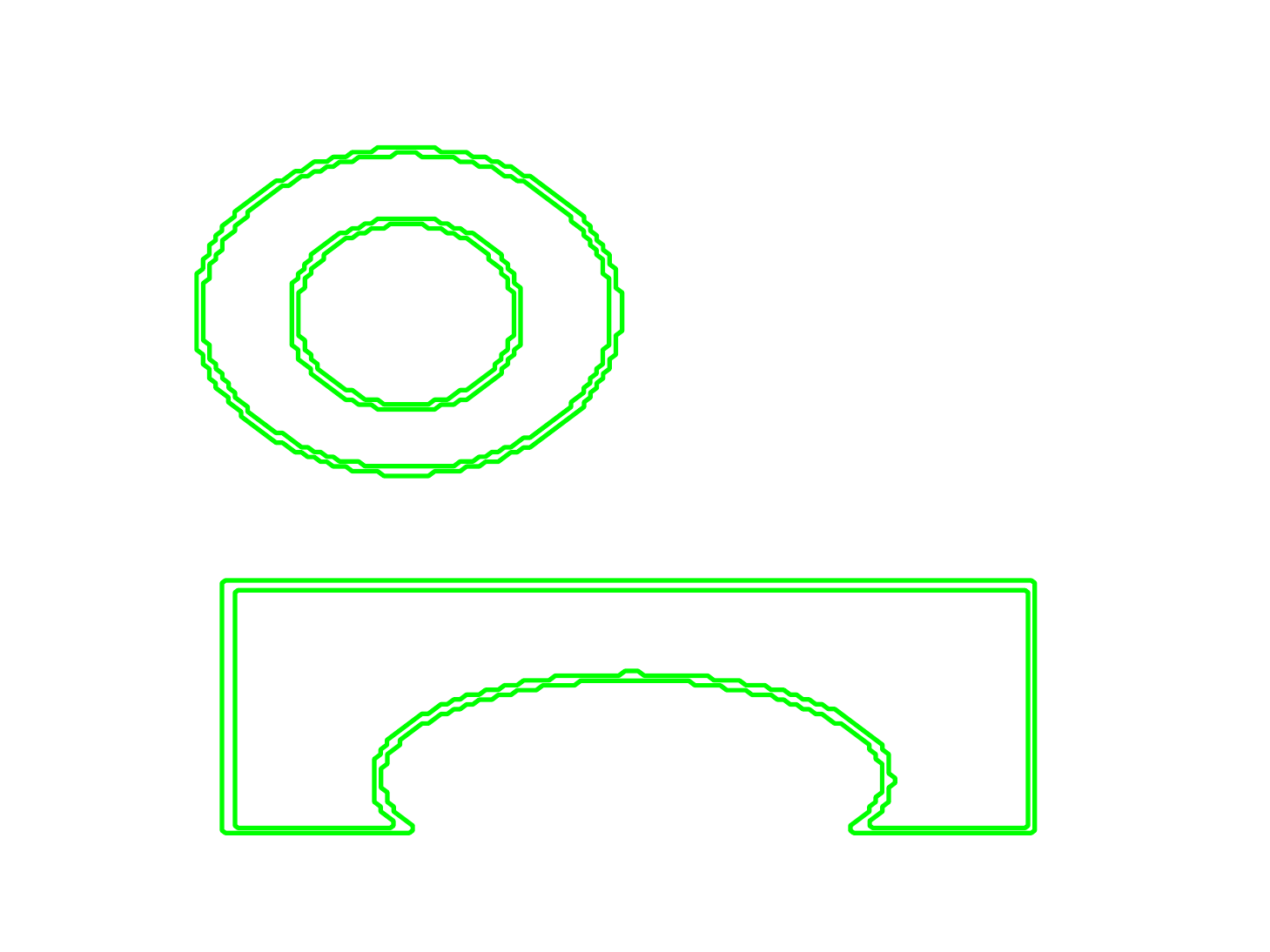}}\vspace{2mm}
    \fbox{\includegraphics[width=2.2cm,height = 1.8cm]{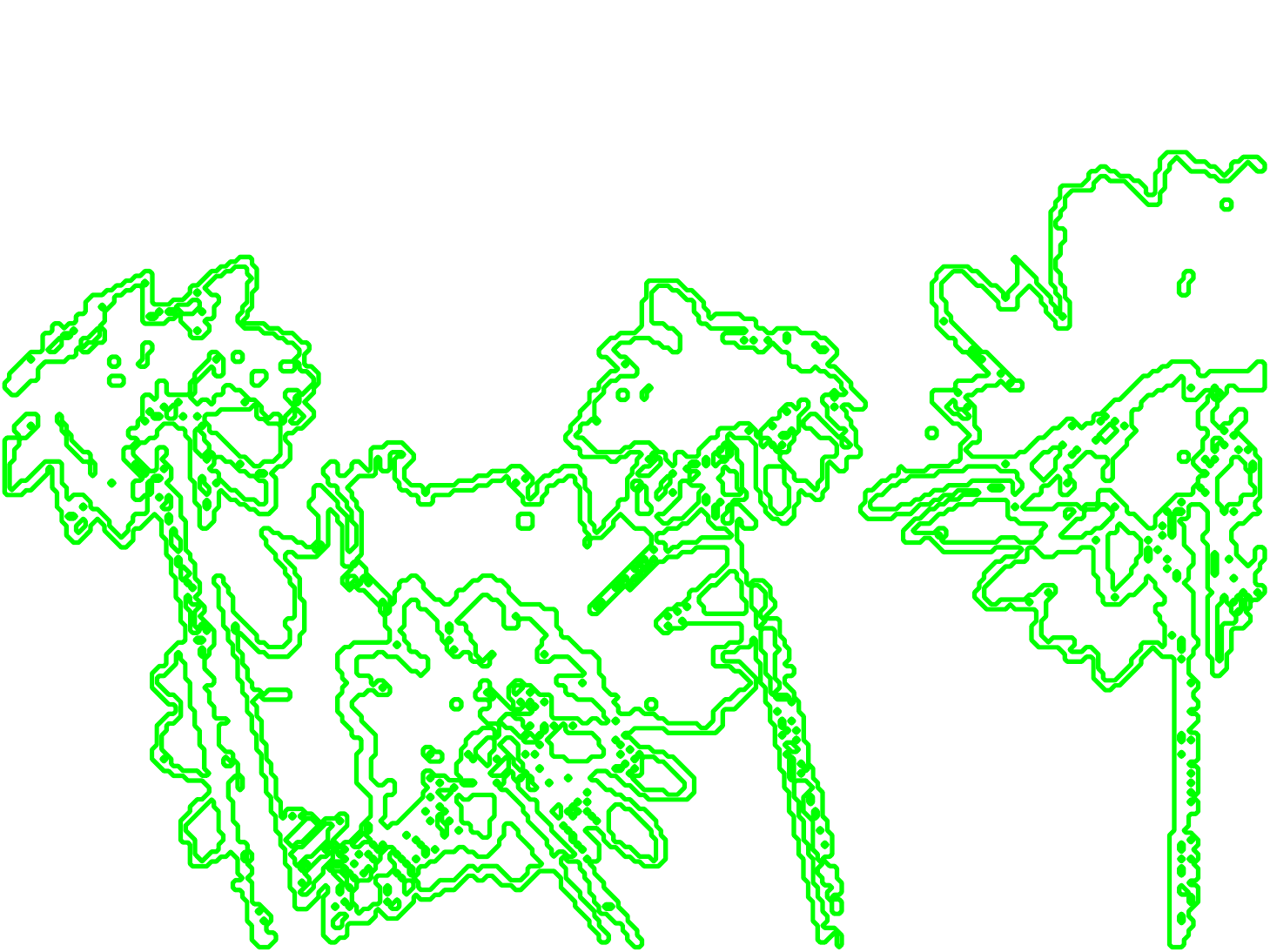}}\vspace{2mm}
    \fbox{\includegraphics[width=2.2cm,height = 1.8cm]{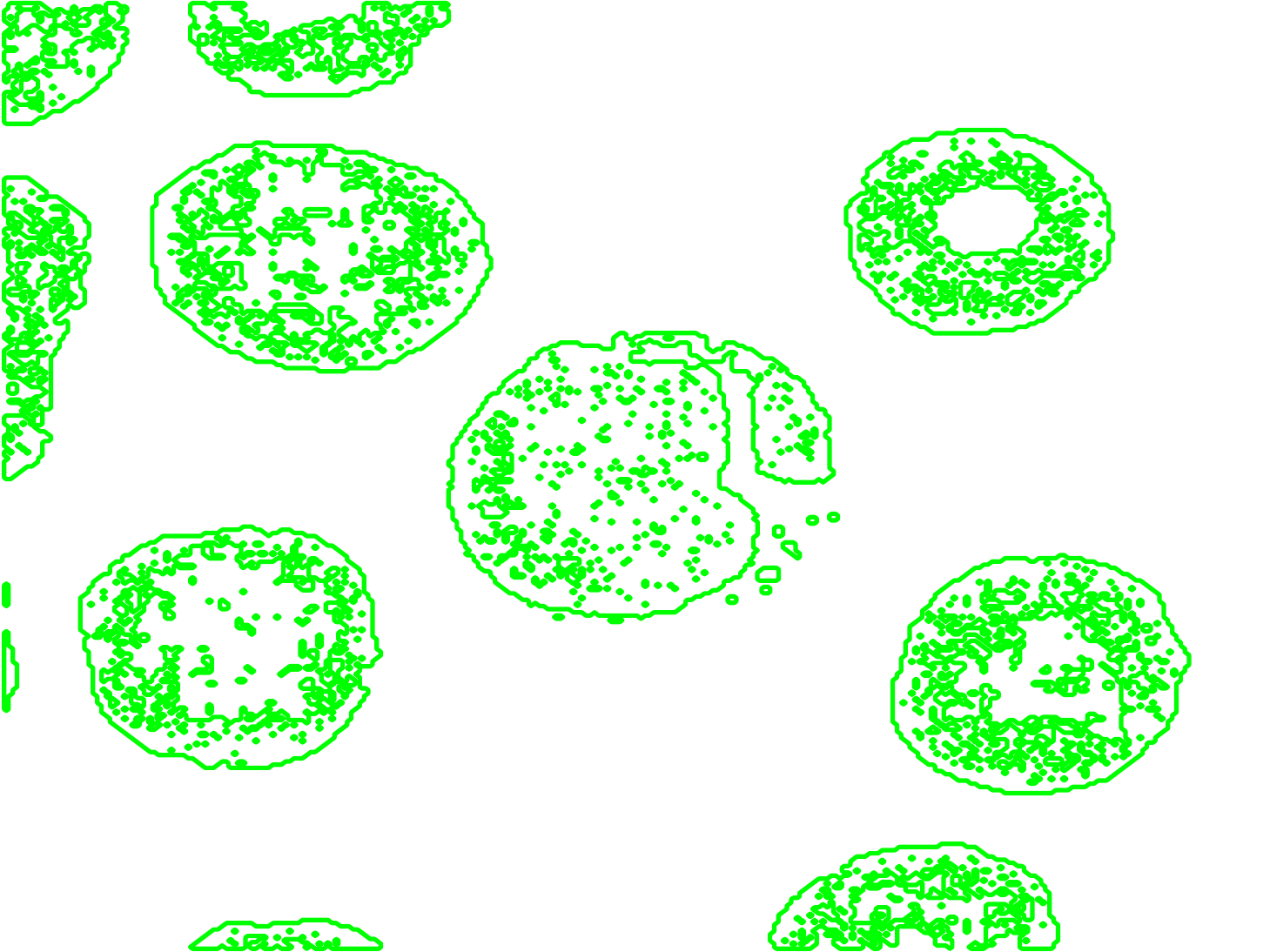}}\vspace{2mm}
    \fbox{\includegraphics[width=2.2cm,height = 1.8cm]{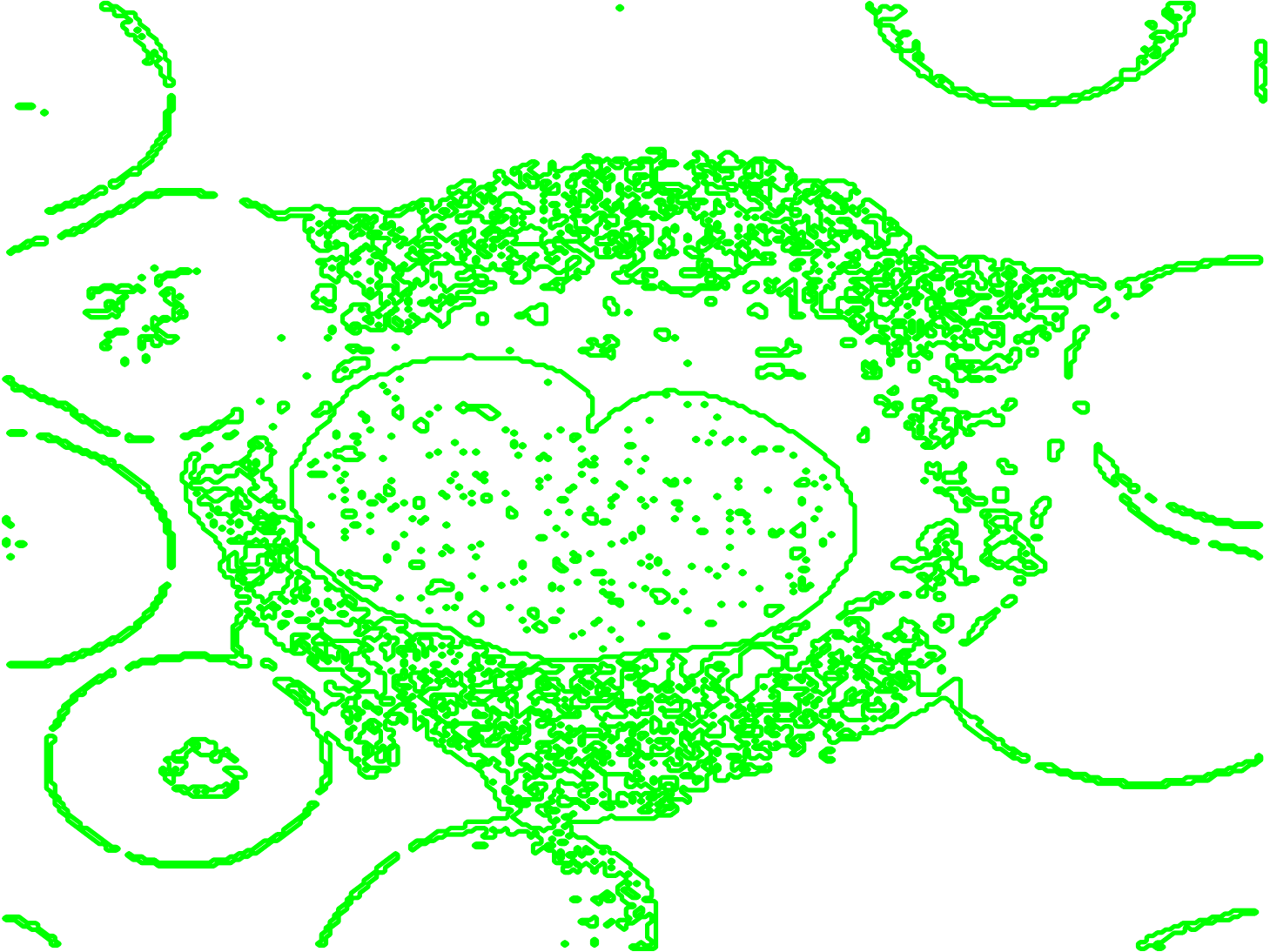}}\vspace{2mm}
    \fbox{\includegraphics[width=2.2cm,height = 1.8cm]{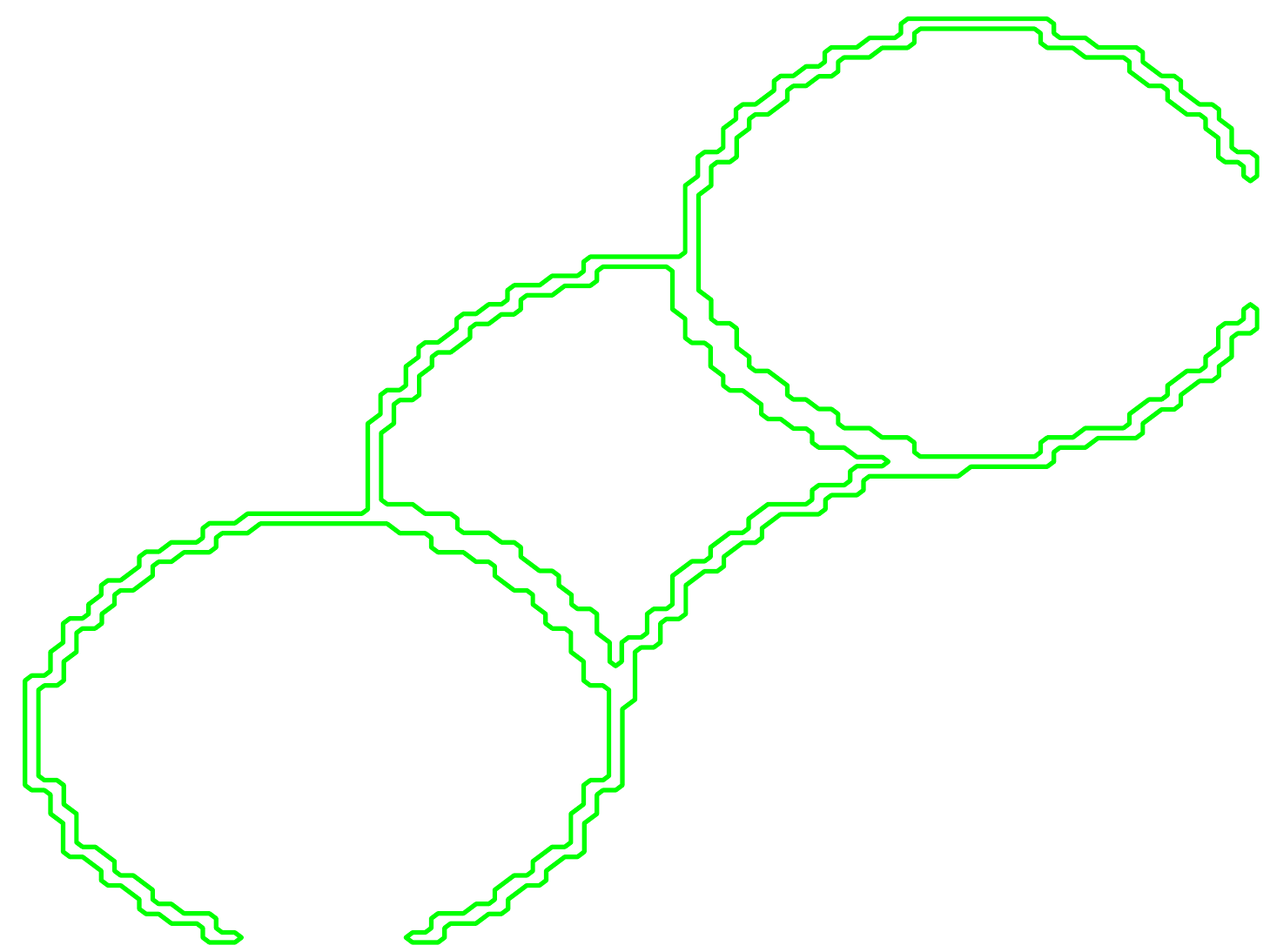}}\vspace{2mm}
	\end{minipage}
	}
  \hspace{-11.8mm}
	\subfigure{
	\begin{minipage}[c]{0.2\textwidth}
    \fbox{\includegraphics[width=2.2cm,height = 1.8cm]{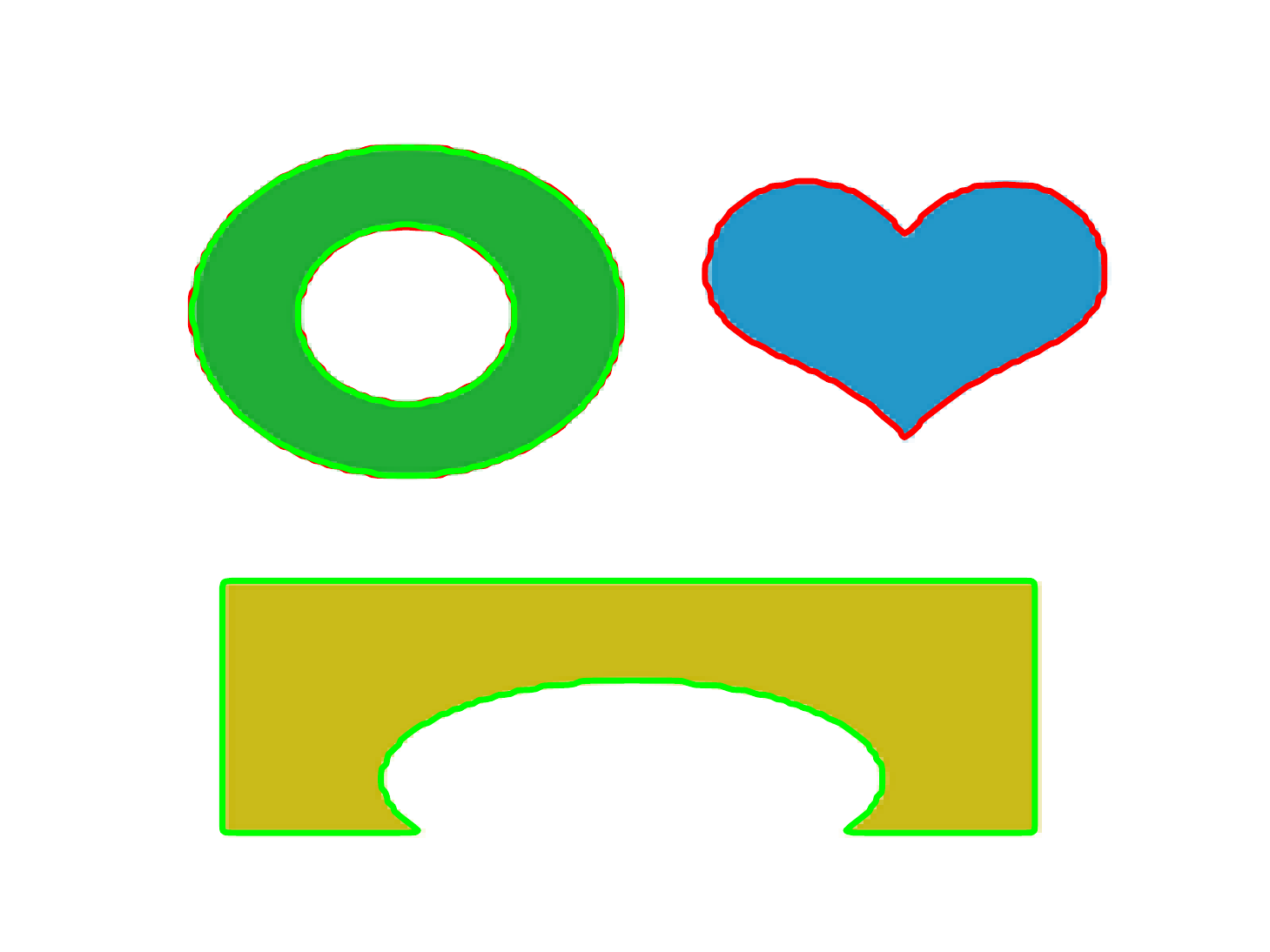}}\vspace{2mm}
   \fbox{\includegraphics[width=2.2cm,height = 1.8cm]{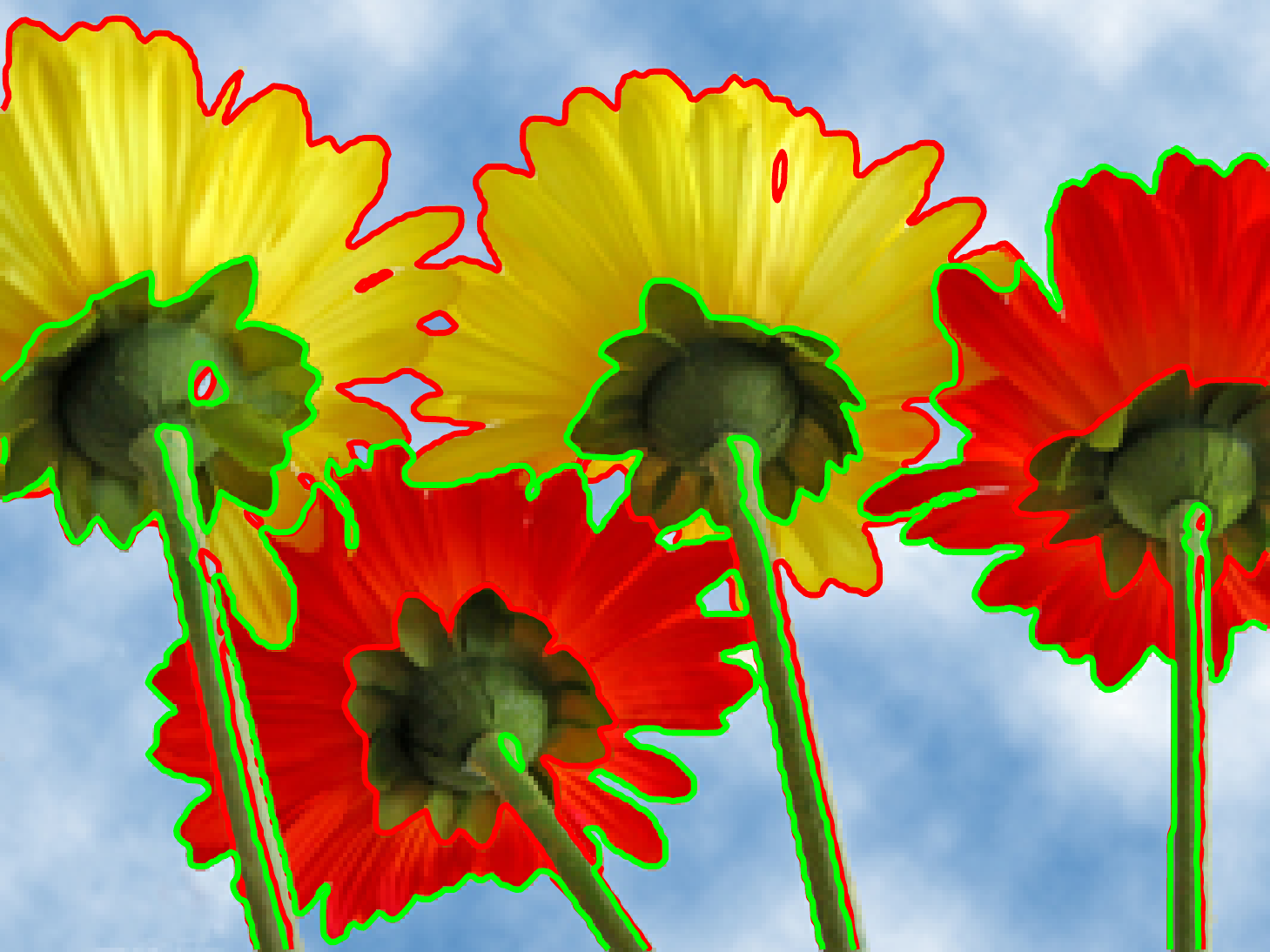}}\vspace{2mm}
   \fbox{\includegraphics[width=2.2cm,height = 1.8cm]{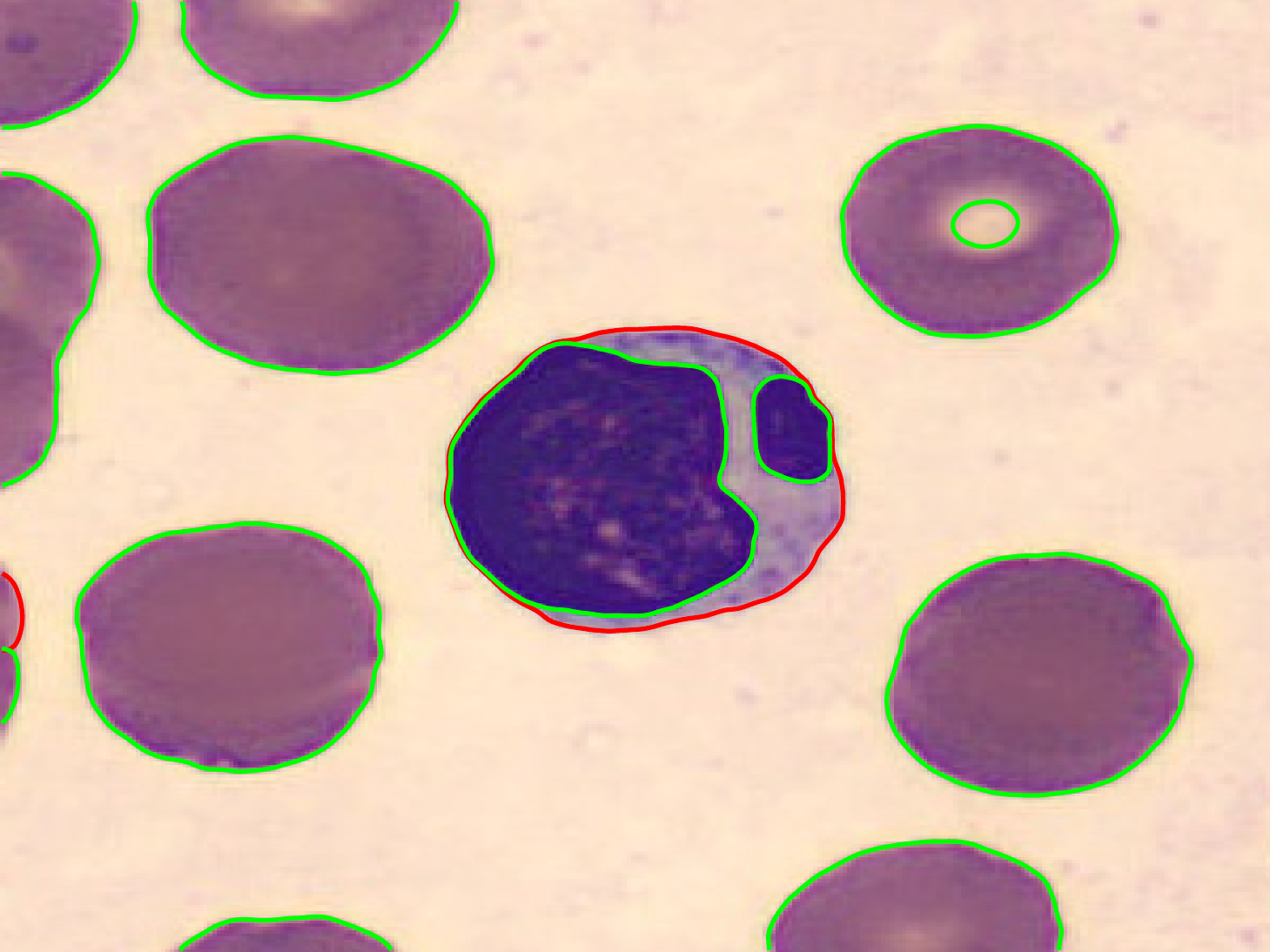}}\vspace{2mm}
   \fbox{\includegraphics[width=2.2cm,height = 1.8cm]{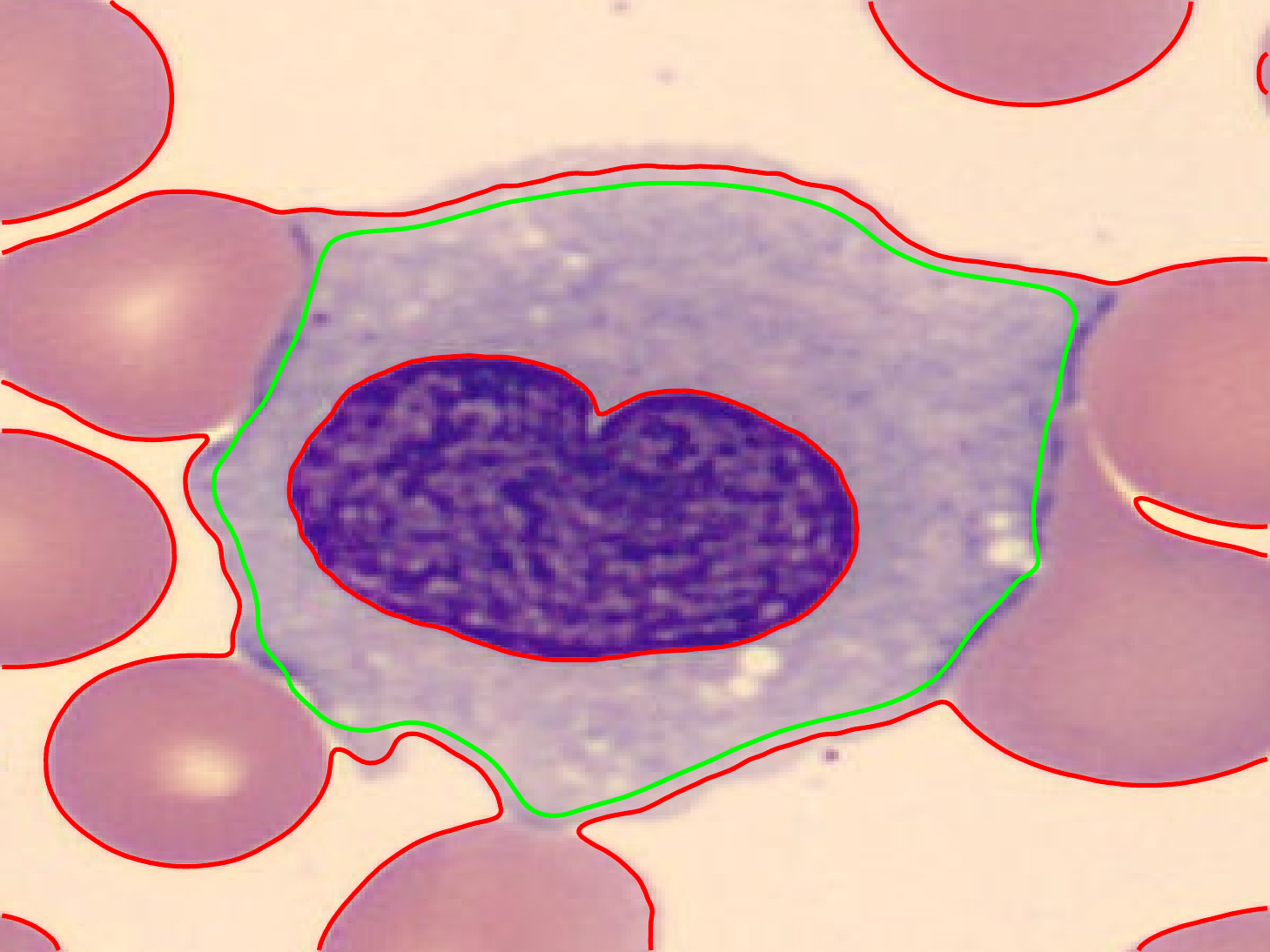}}\vspace{2mm}
    \fbox{\includegraphics[width=2.2cm,height = 1.8cm]{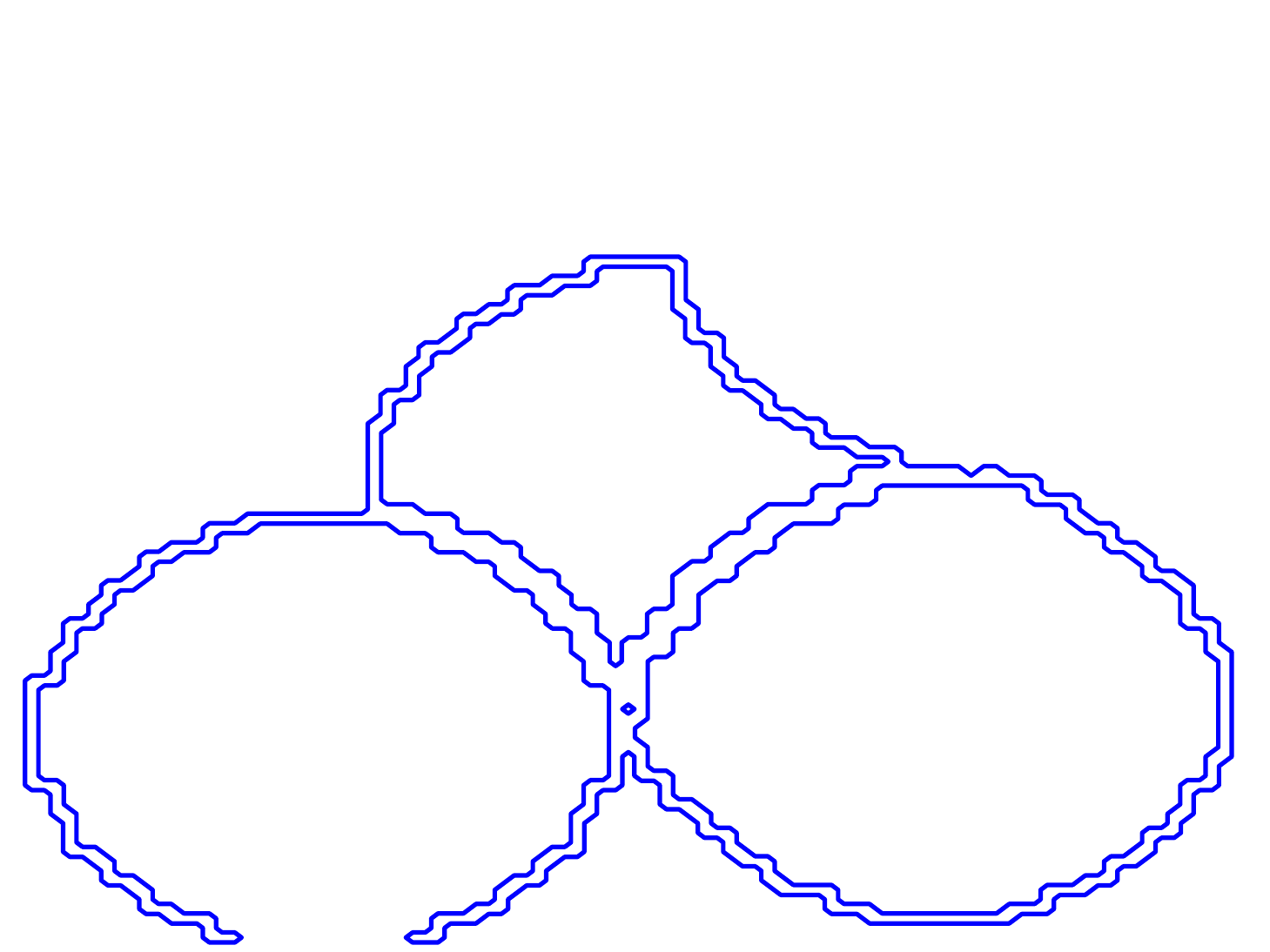}}\vspace{2mm}
	\end{minipage}
	}
\hspace{-11.8mm}
	\subfigure{
	\begin{minipage}[c]{0.2\textwidth}
    \fbox{\includegraphics[width=2.2cm,height = 1.8cm]{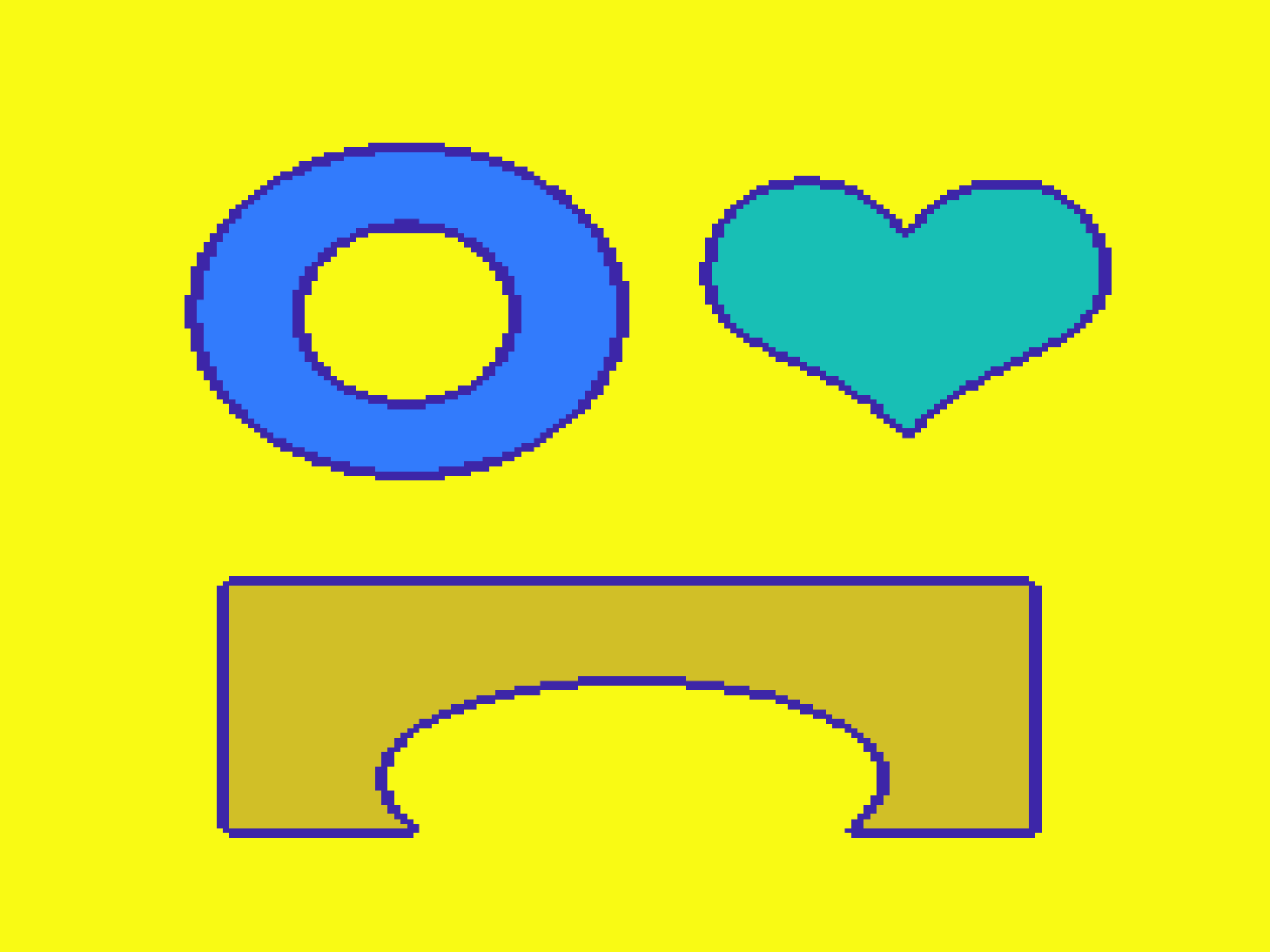}}\vspace{2mm}
   \fbox{\includegraphics[width=2.2cm,height = 1.8cm]{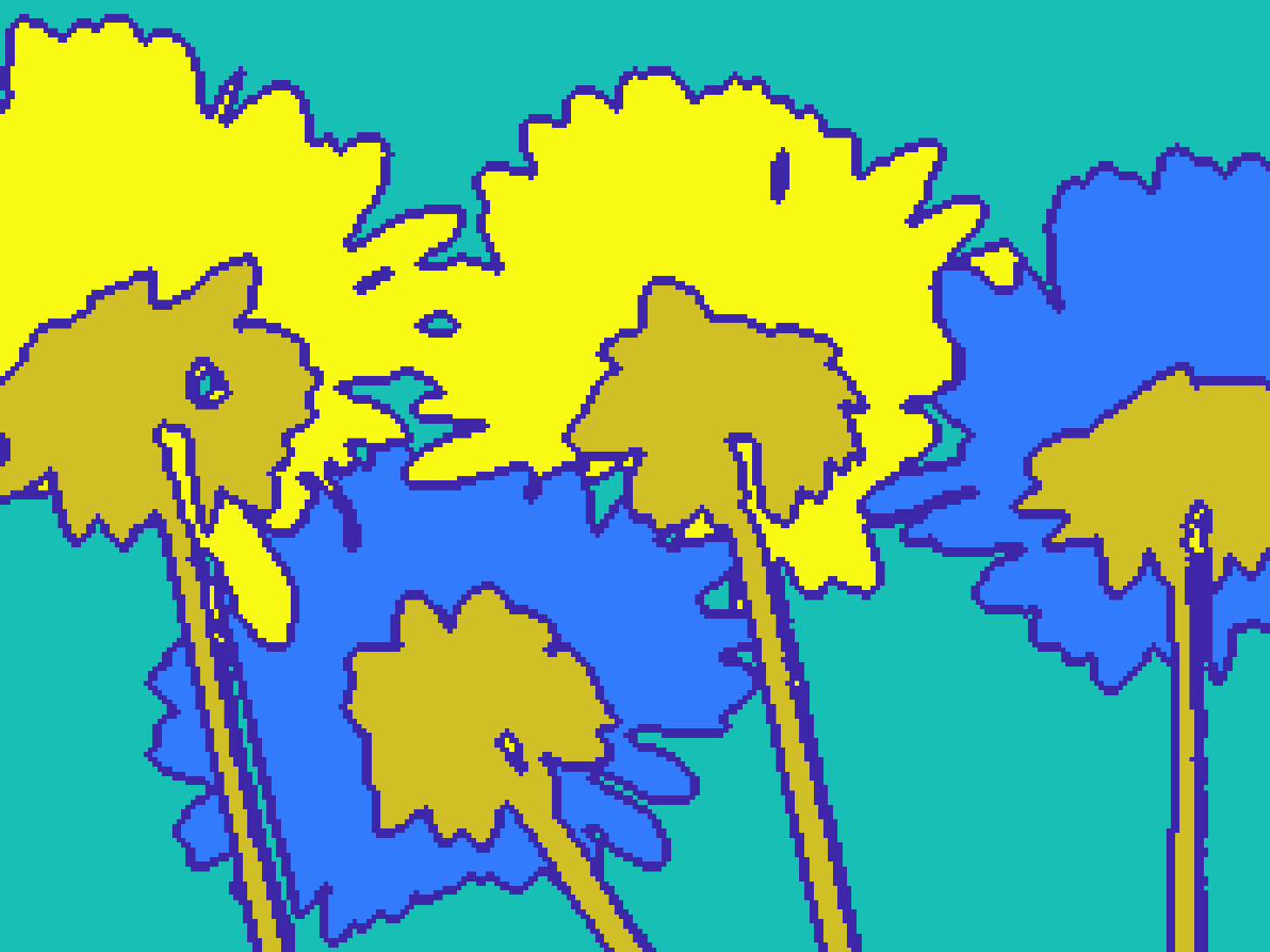}}\vspace{2mm}
   \fbox{\includegraphics[width=2.2cm,height = 1.8cm]{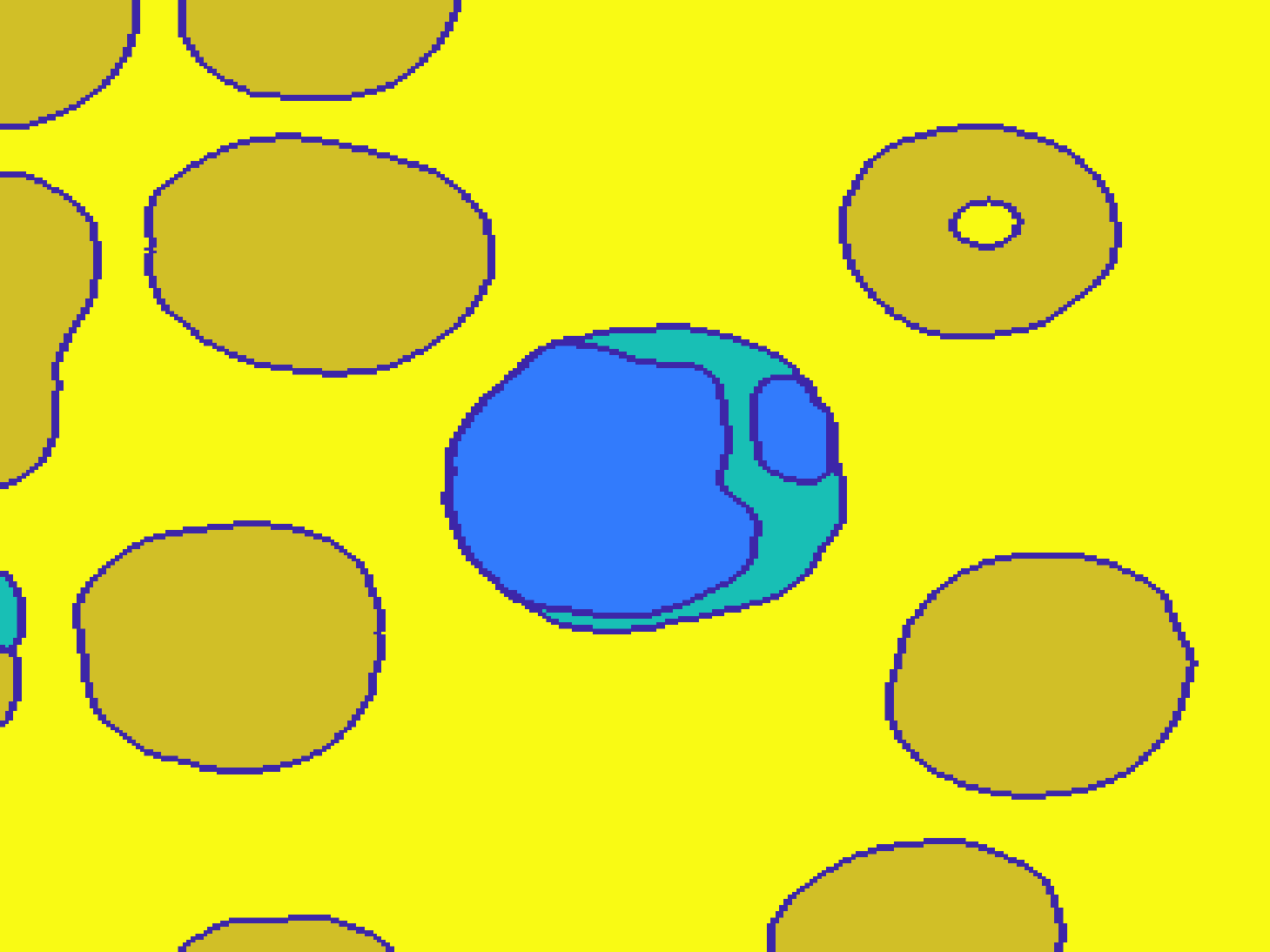}}\vspace{2mm}
   \fbox{\includegraphics[width=2.2cm,height = 1.8cm]{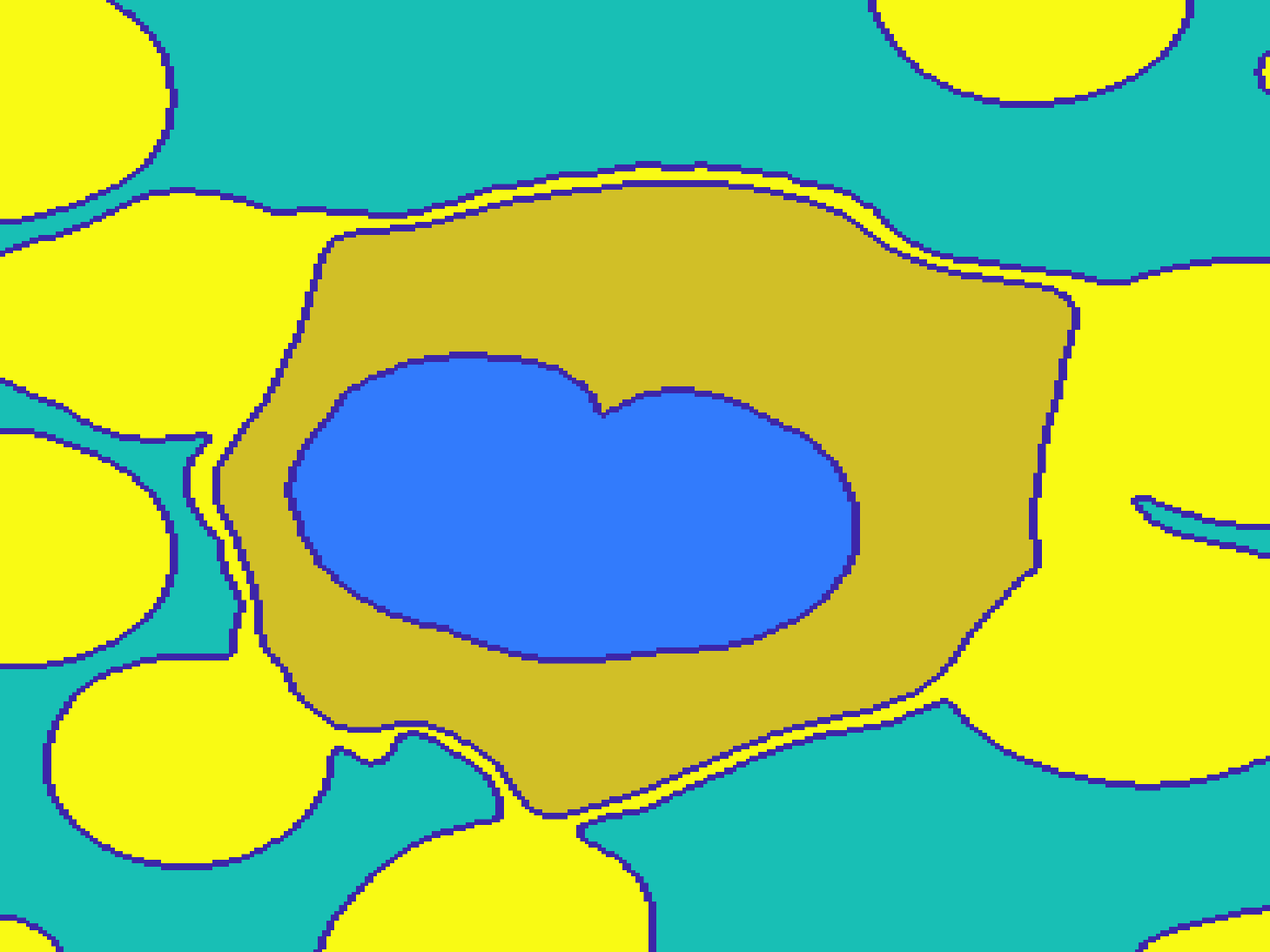}}\vspace{2mm}
   \fbox{\includegraphics[width=2.2cm,height = 1.8cm]{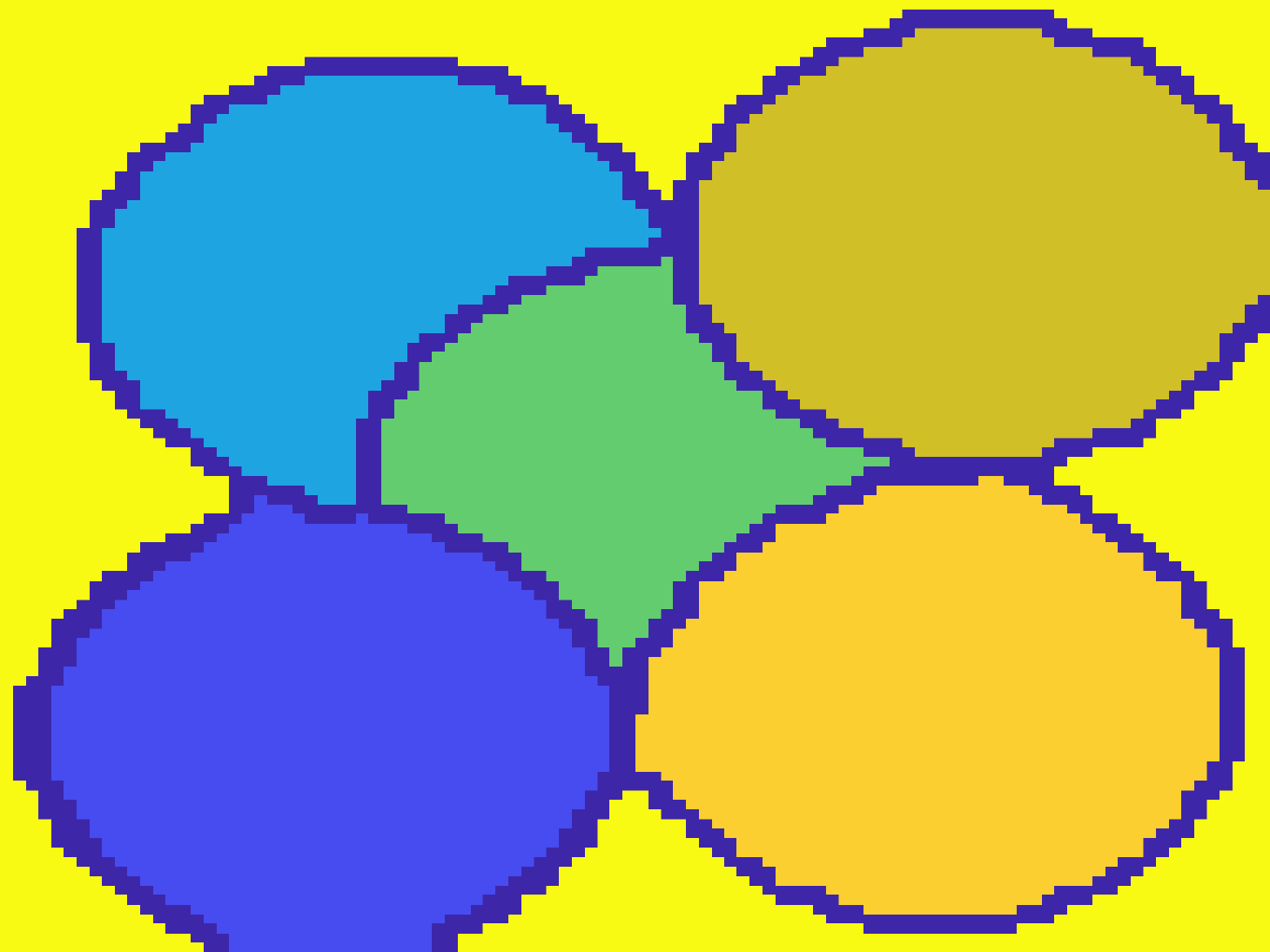}}\vspace{2mm}
	\end{minipage}
	}
}
\caption{ \label{fig:subsec2} The segmentation results for several images. From left to right: Original image, the initial contour of $U_1$, the initial contour of $U_2$, the contour results (the initial contour of $U_3$) and the segmentation results. }
\end{figure}

\begin{table}[!b]
	\begin{tabular}{lllp{1.2cm}lllll}
		\hline
		\multirow{2}{*}{Image}&
		\multicolumn{3}{c}{Multi-IGLIM} &\multicolumn{4}{c}{ACCV} \cr
		&$\kappa$ &$\sigma$ &$M$ &$\epsilon$ &$\lambda$ &$h$ &$S$\cr
		\hline
        { Synthetic} &50  &$ 0.05$  &5 &6 &40 & 0.3 &120 \cr%
		{ Flowers} &50  &$0.05$  &10  &6 &40 & 0.3 &120\cr
		{ WBC1} &50  &$0.01$  &1  &{8} &40 & 0.3 &70\cr
		{ WBC2} &50  &$0.01$  &1  &{8} &40 & 0.3 &70 \cr
		{ 6-phases} &50  &$0.05$  &1  &8 &40 & 0.5 &80\cr\hline

	\end{tabular}
	\label{tab:subsec2}\caption{Parameters setting for images in Figure \ref{fig:subsec2}.}
\end{table}

\subsection{Experiments on segmentation comparison } 
For illustrating the capability of our ACCV model, we will make comparisons with several prominent image segmentation models, including the Cahn-Hilliard model (CH) \cite{Yang2019_JSC}; Smoothing, Lifting and Thresholding model (SLaT) \cite{Cai2015JSC}; Iterative convolution-thresholding method (ICTM)  \cite{Wang_2017}, and convex K-means approach (CKA) \cite{Wu2021IET}. 4-phase and 6-phase image segmentations will be tested in Figure \ref{fig:WBC}. From top to bottom, the images ``flowers", ``4colors", ``WBC1", ``WBC2", ``6phases" are segmented respectively.

Before the illustration of the segmentation, we need to do some explanations as follows. The segmentation results of SLaT and CKA keep the same color as the original image so that some segmentation results look like the original images, such as the image ``4colors" and ``6phases". For the remaining models, we have filled the segmentation results with other colors after the segmentation to distinguish. The recorded CPU time in Table \ref{tab:CPU_time_comparison} is obtained by the average of 10 times segmentations. Since the chosen models also have some crucial parameters which will directly determine the behavior of the segmentation method, we have tried our best to find the optimal parameters to balance the segmentation result and the CPU time, as we did for our ACCV model.

Considering the segmentation result and the CPU time, it can be seen that the CH model takes much longer than the other models. Since the Cahn-Hilliard equation is a fourth-order PDE, the simulation will be more time-consuming. And for segmentation of more than four phases, it is not satisfactory. Except for the image ``flowers", the SLaT model finishes the segmentation in the least time among all models. For the image of white blood cells, the segmentation shows double edges around the cell membrane of the red blood cell due to the low resolution. Since the image `` flowers " has texture in each partition rather than a homogeneous color block, it is the reason that the SLaT model needs more time for ``flowers '' than the other images. ICTM achieves satisfactory segmentations for all images, although it costs a little longer in the images of the white blood cell. The CKA fails to separate the images ``WBC1" and ``WBC2", compared with our ACCV model. Through experiments, we summarize that the ACCV model is generally applicable for various images except for the image with many details. And the efficiency of the ACCV model is comparable to the above prominent models.


\begin{figure}[!t]
\resizebox{\textwidth}{!}{\hspace{11.8mm}
	\centering
	\subfigure{
	\begin{minipage}[c]{0.2\textwidth}
	\includegraphics[width=2.0cm]{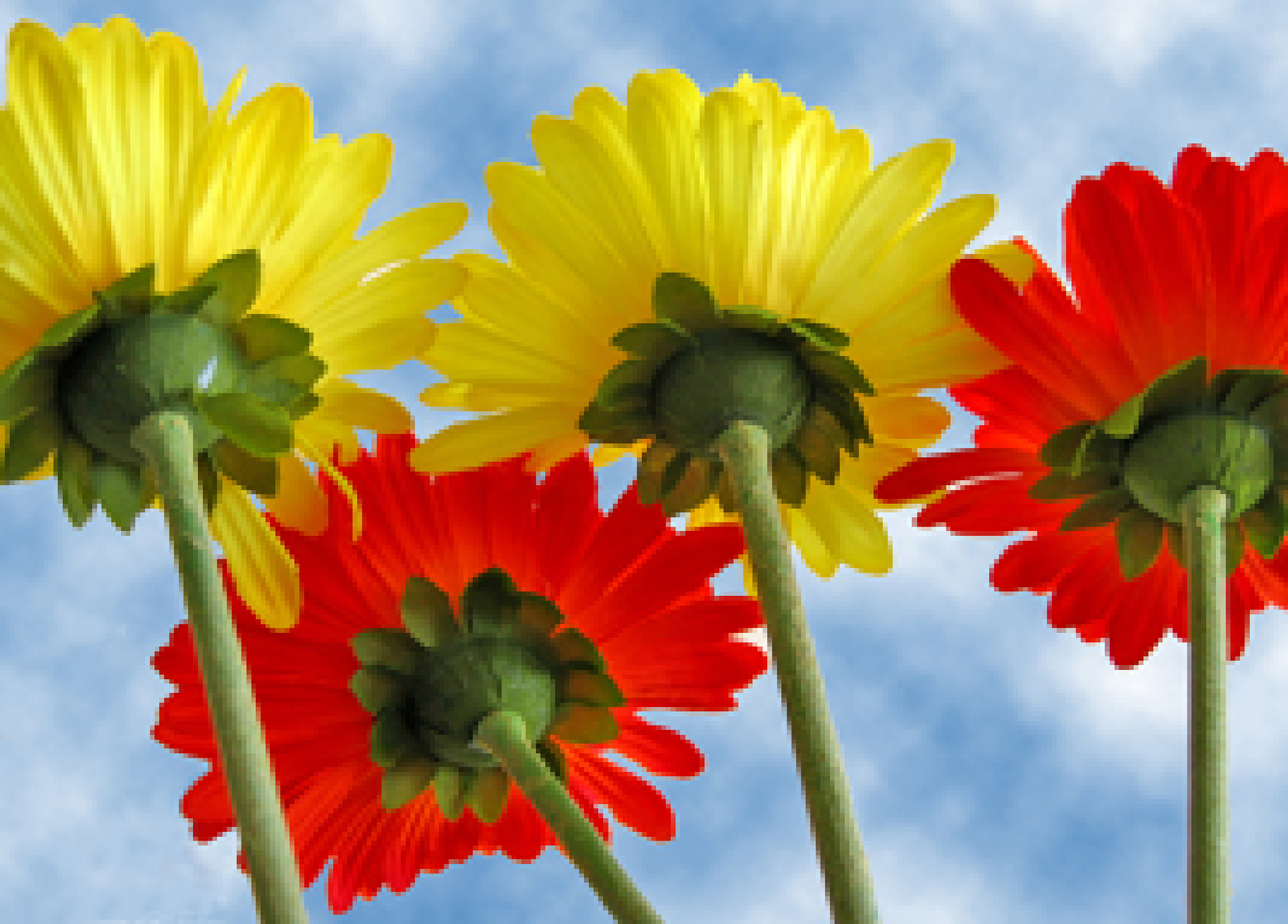}\vspace{1mm} \vspace{1mm} 	
	\includegraphics[width=2.0cm]{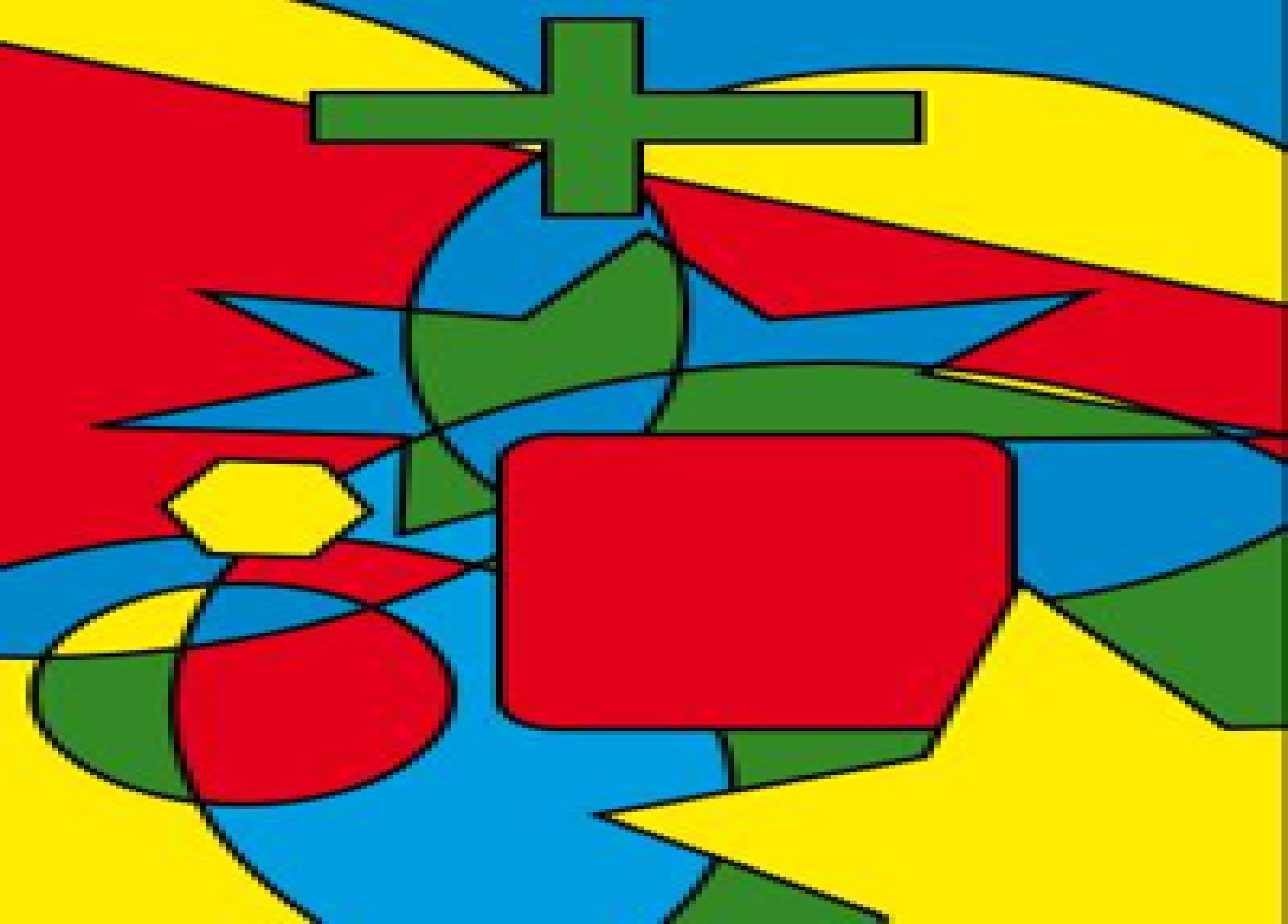} \vspace{1mm}  
    \includegraphics[width=2.0cm]{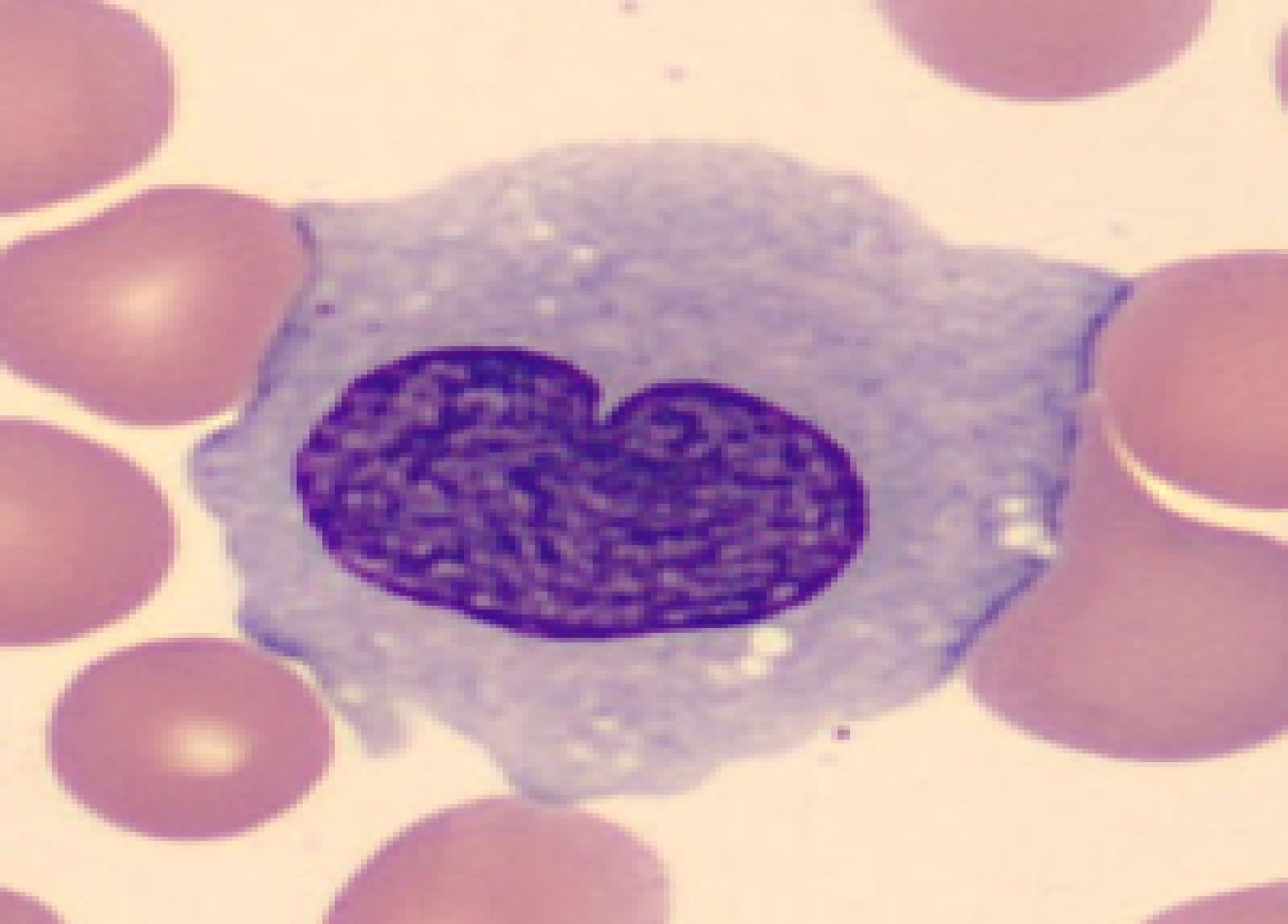}\vspace{0.5mm}	
	\includegraphics[width=2.0cm]{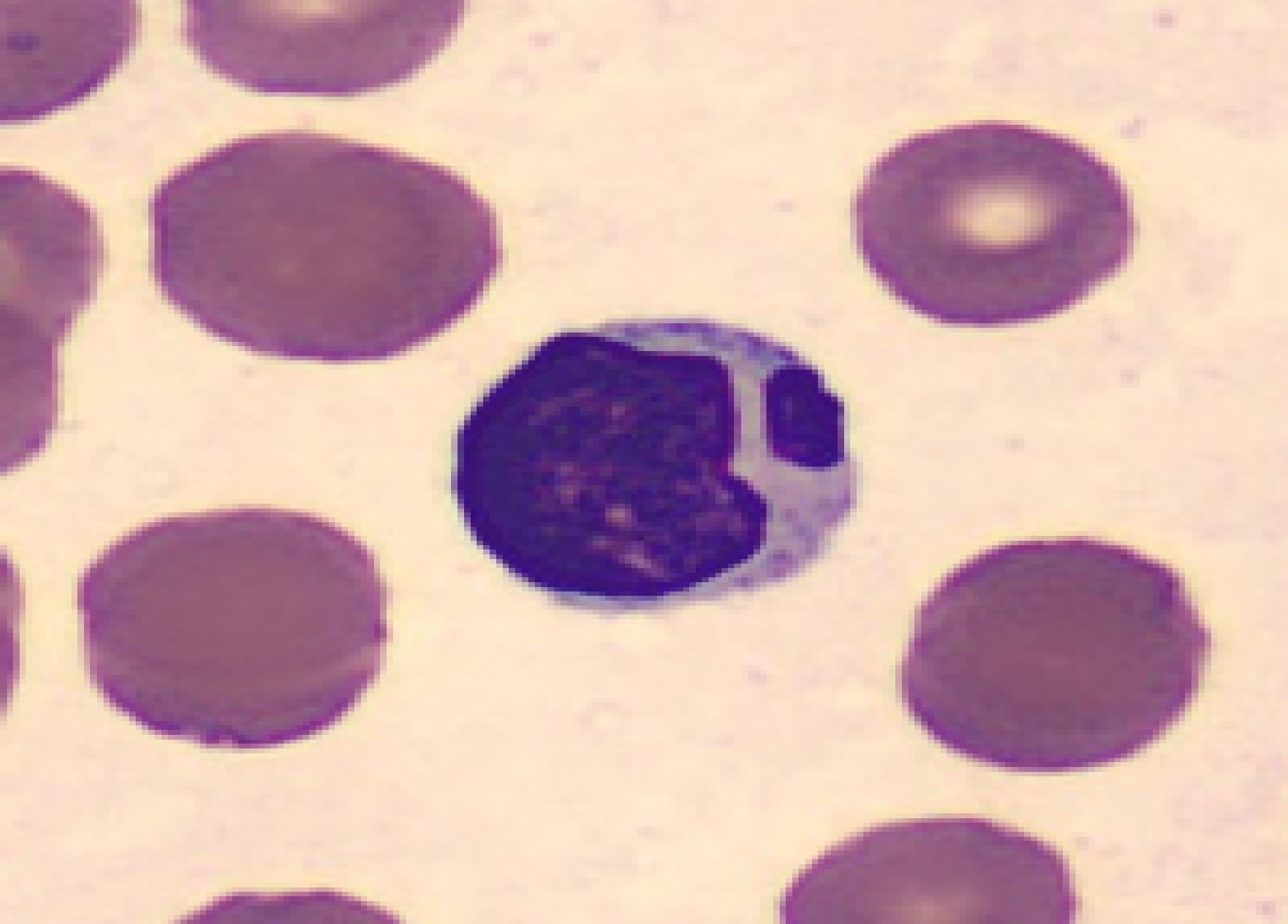} \vspace{1mm}\vspace{1mm}
\includegraphics[width=2.0cm]{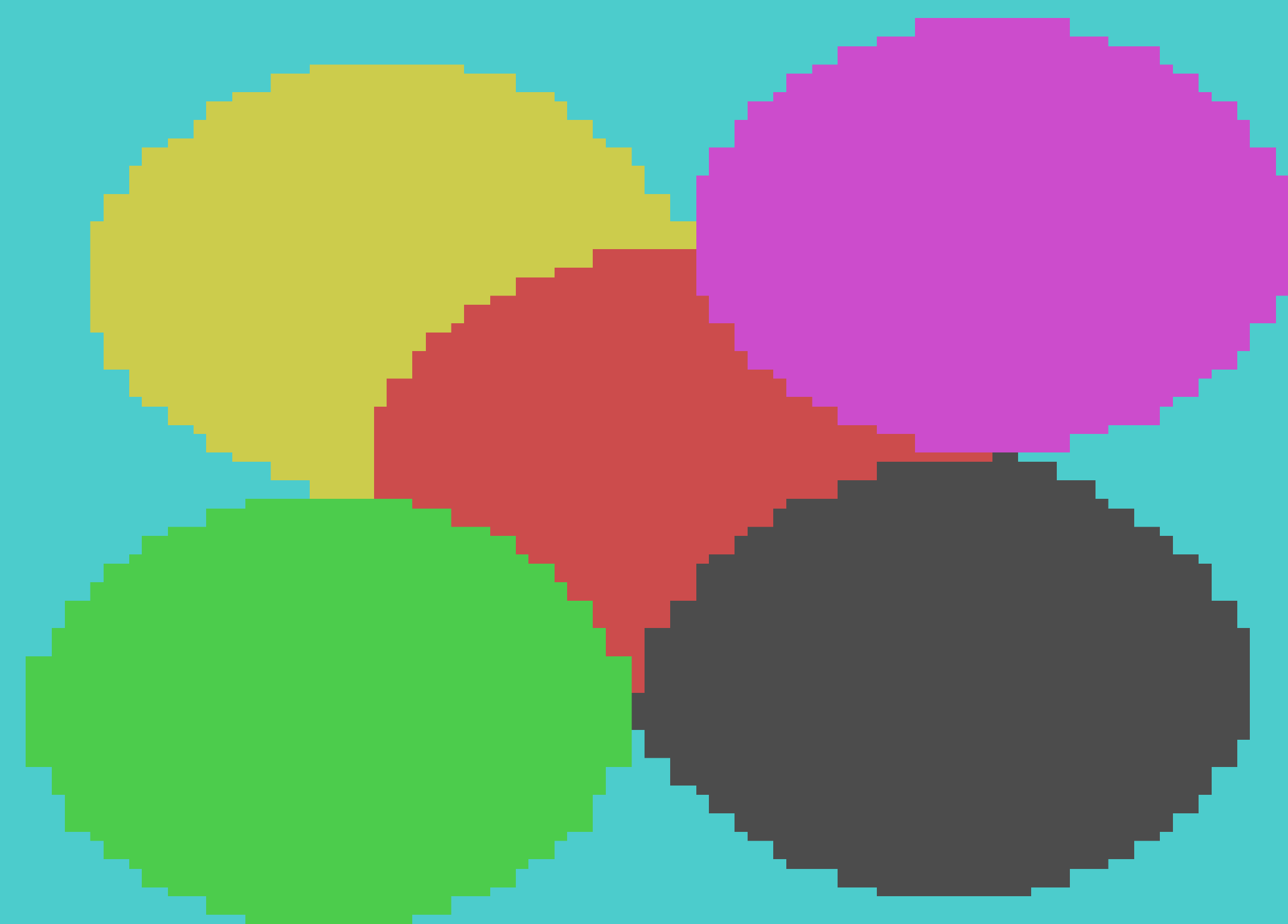} \vspace{1mm}
	\end{minipage}
	}
  \hspace{-11.8mm}
	\subfigure{
	\begin{minipage}[c]{0.2\textwidth}\vspace{1mm}
    \includegraphics[width=2.08cm]{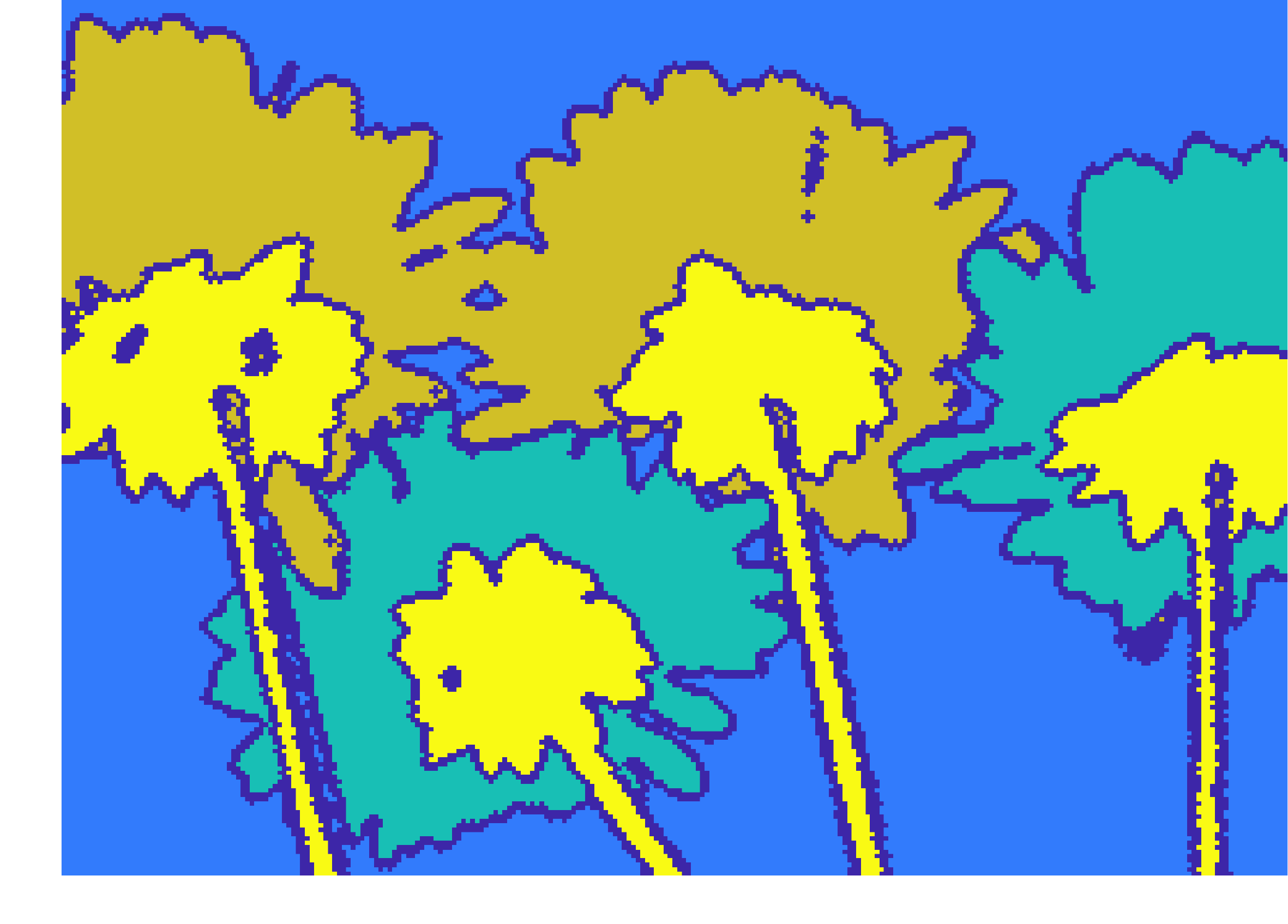}\vspace{0.5mm}
	\includegraphics[width=2.08cm]{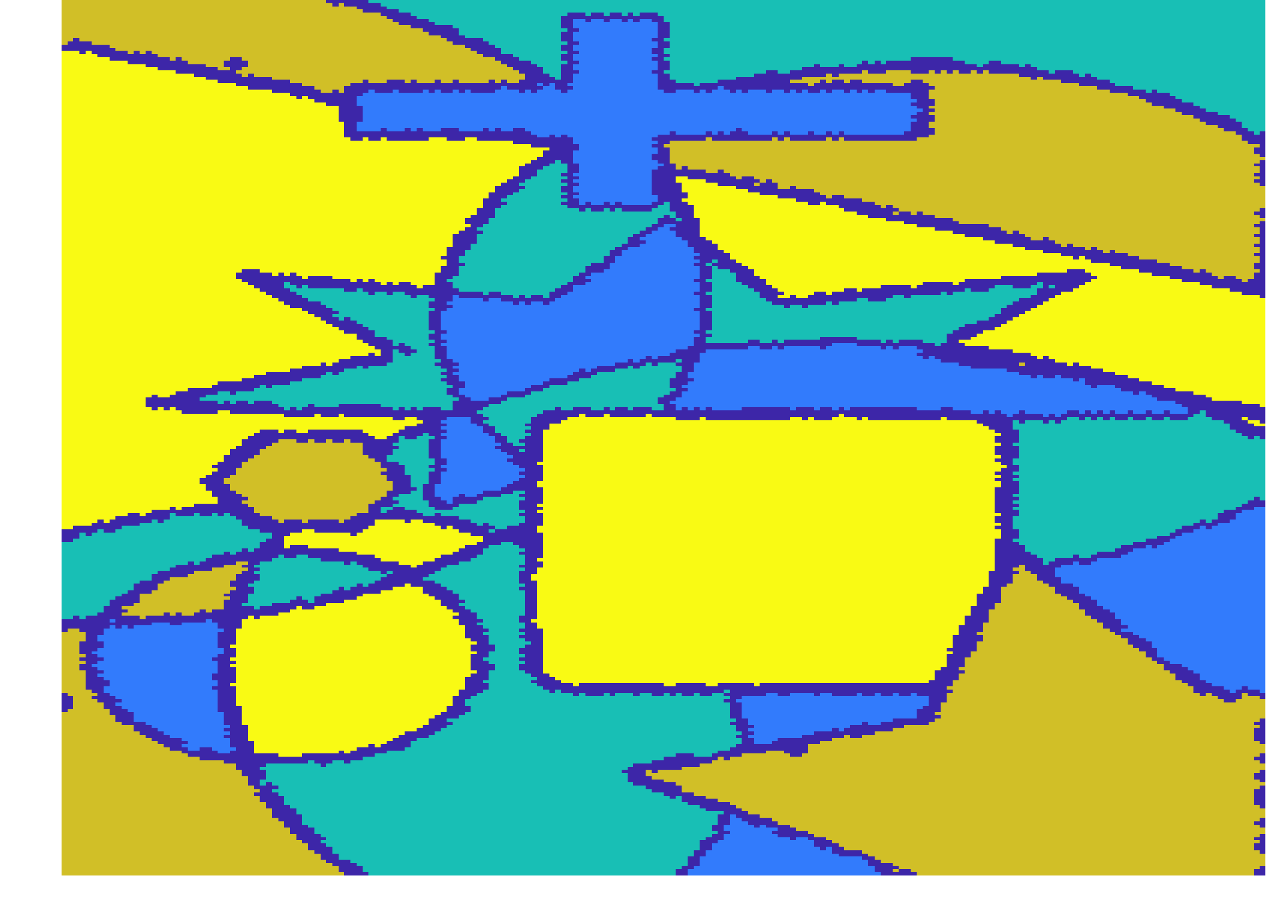}\vspace{0.5mm}
\includegraphics[width=2.08cm]{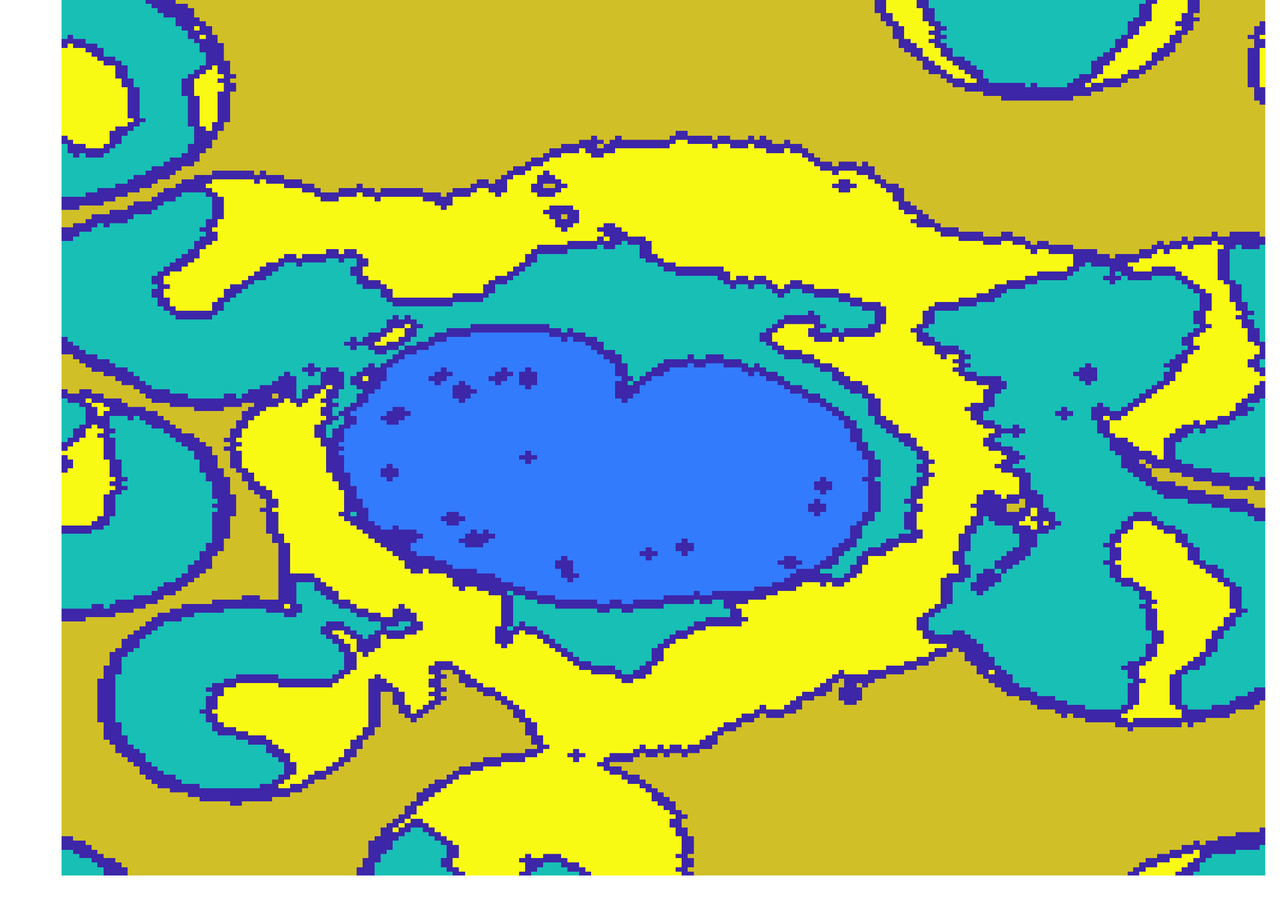}\vspace{1mm}
	\includegraphics[width=2.08cm]{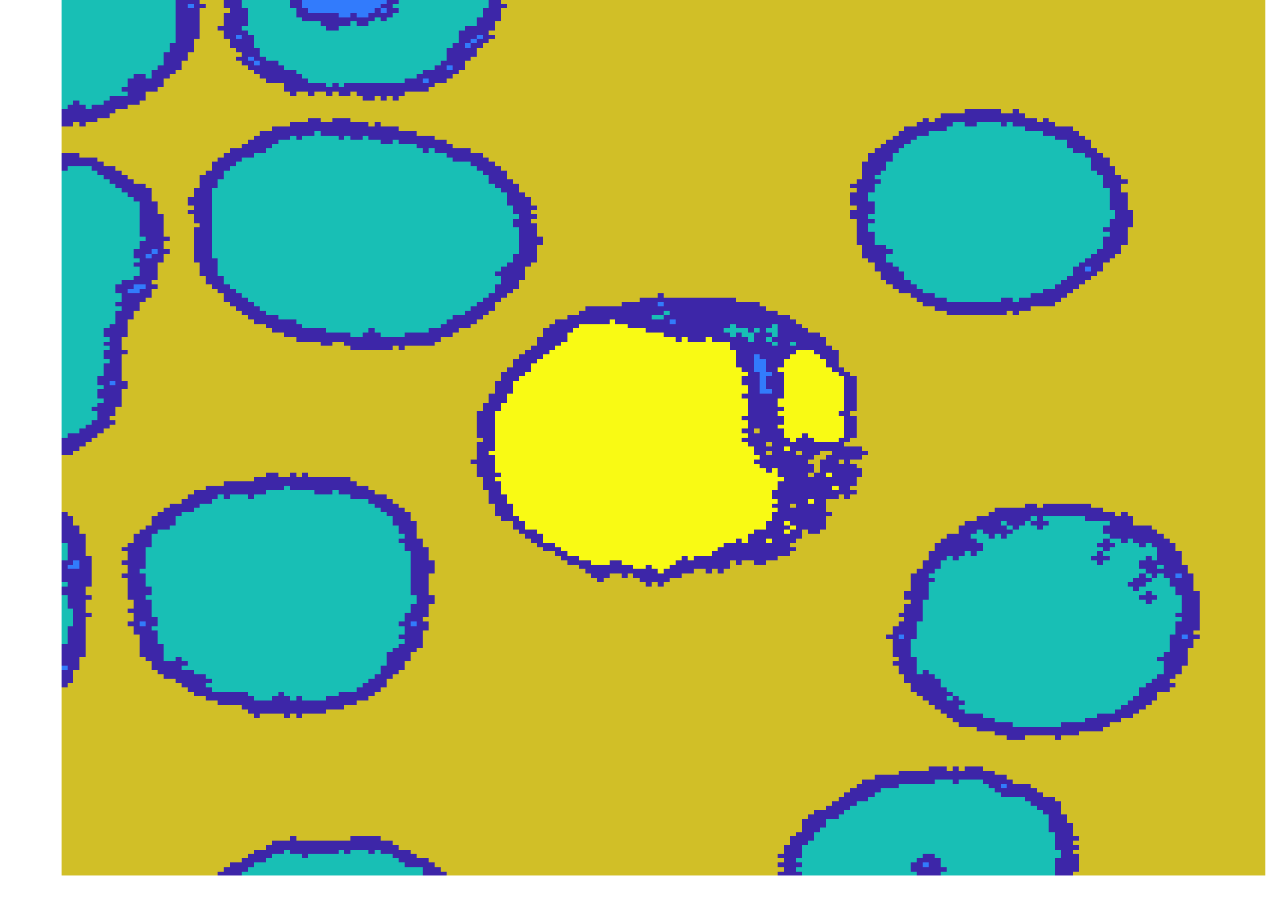}\vspace{1mm}
\includegraphics[width=2.08cm]{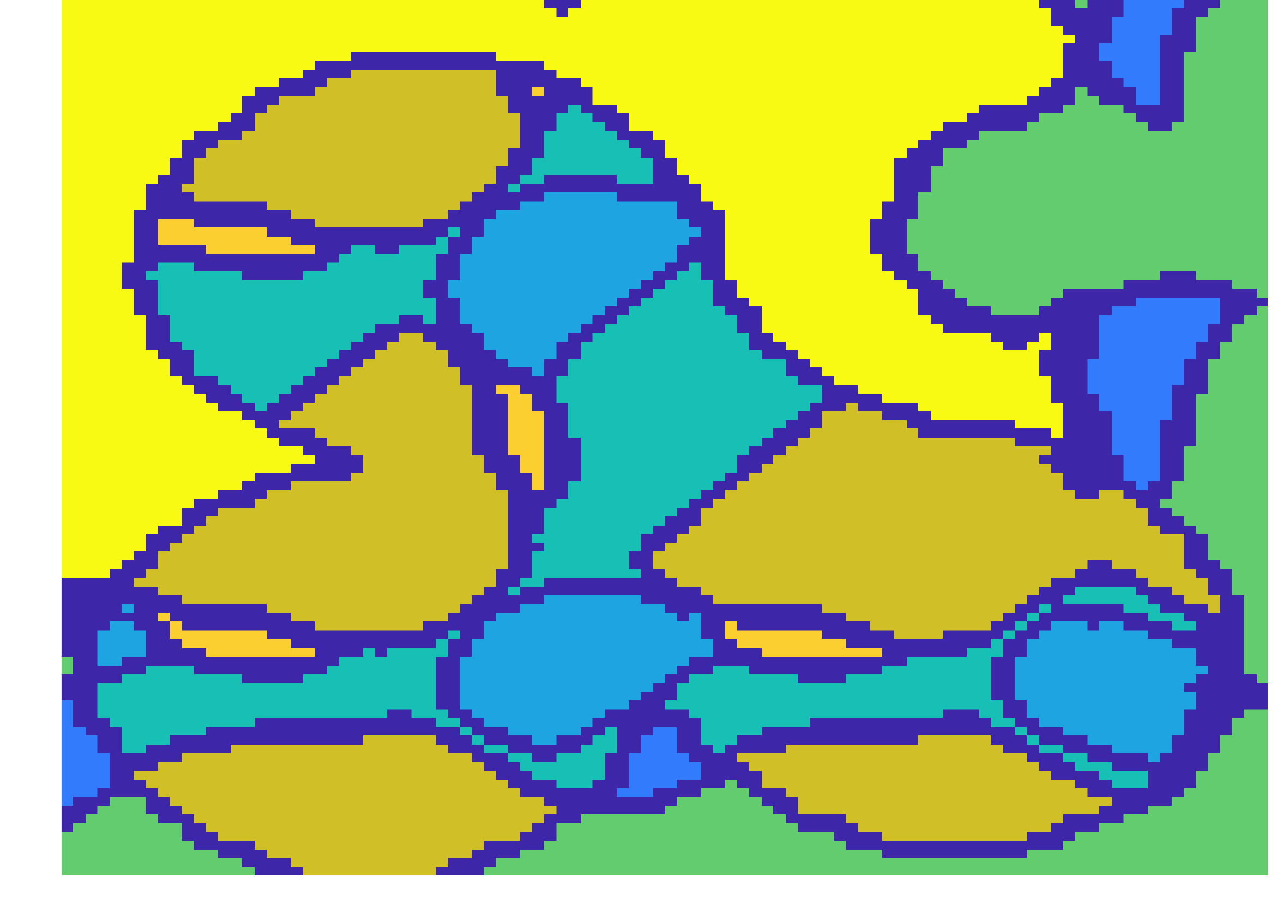}\vspace{1mm}
	\end{minipage}
	}
\hspace{-11.8mm}
	\subfigure{
	\begin{minipage}[c]{0.2\textwidth}
    \includegraphics[width=2.0cm]{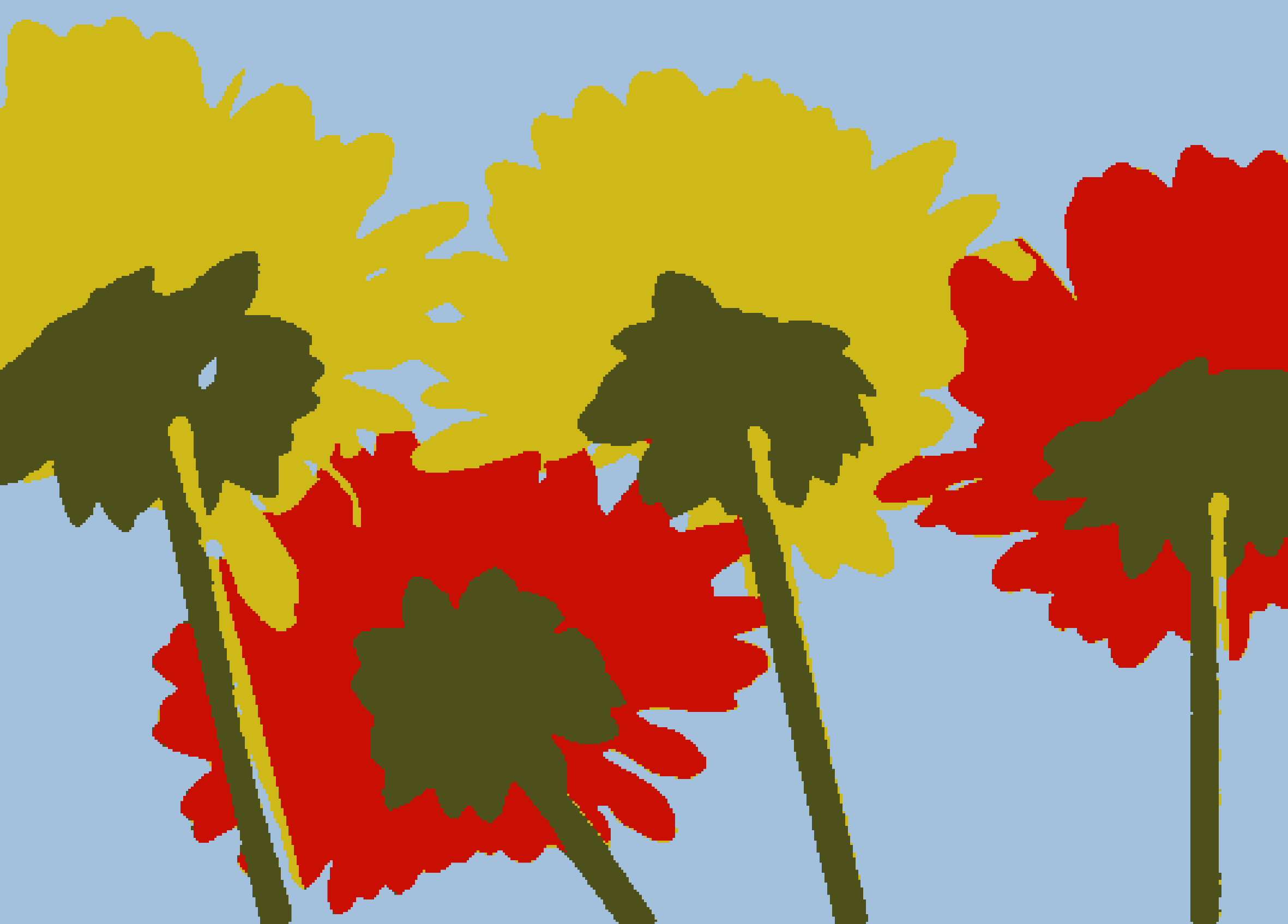}\vspace{1mm}
	\includegraphics[width=2.0cm]{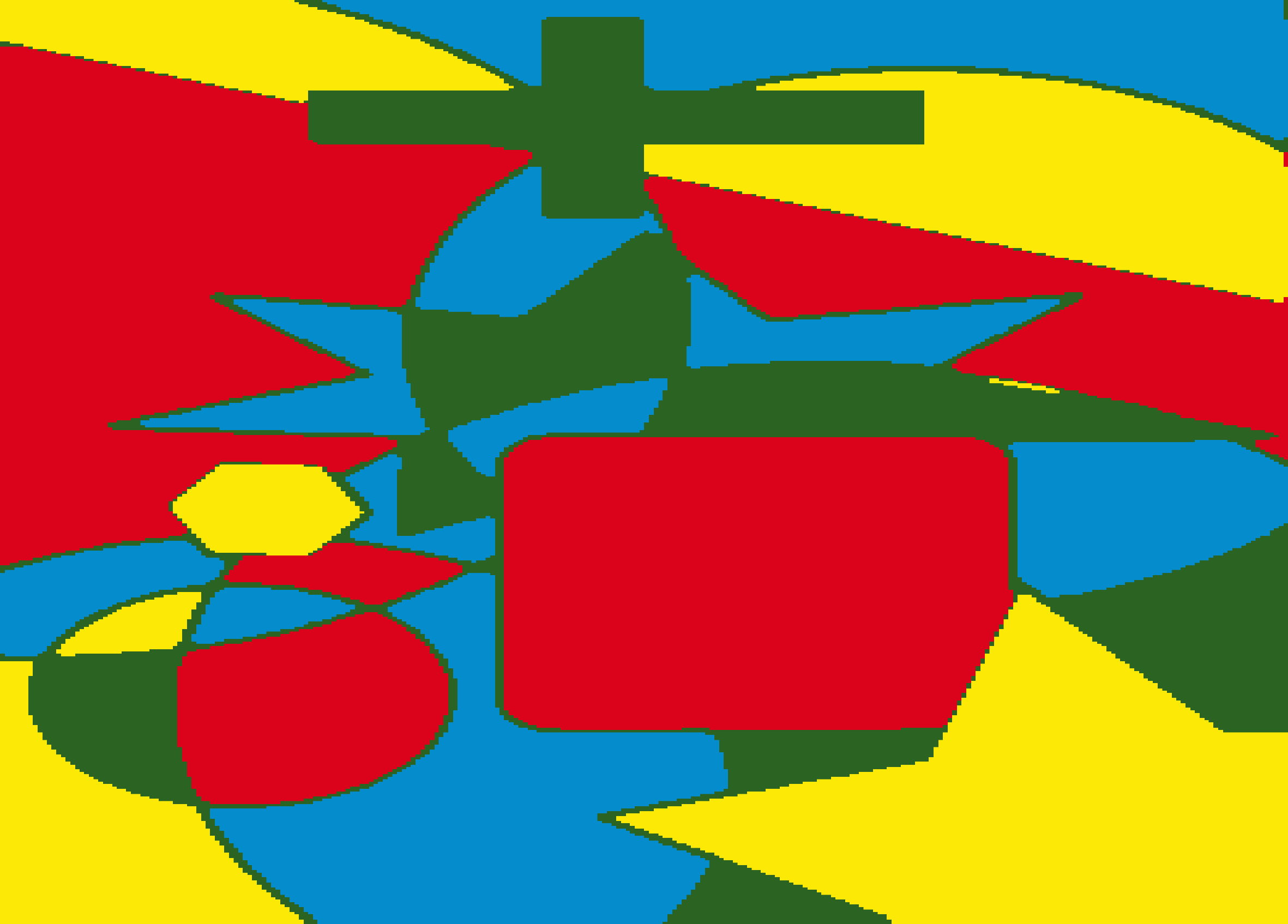} \vspace{1mm}
    \includegraphics[width=2.0cm]{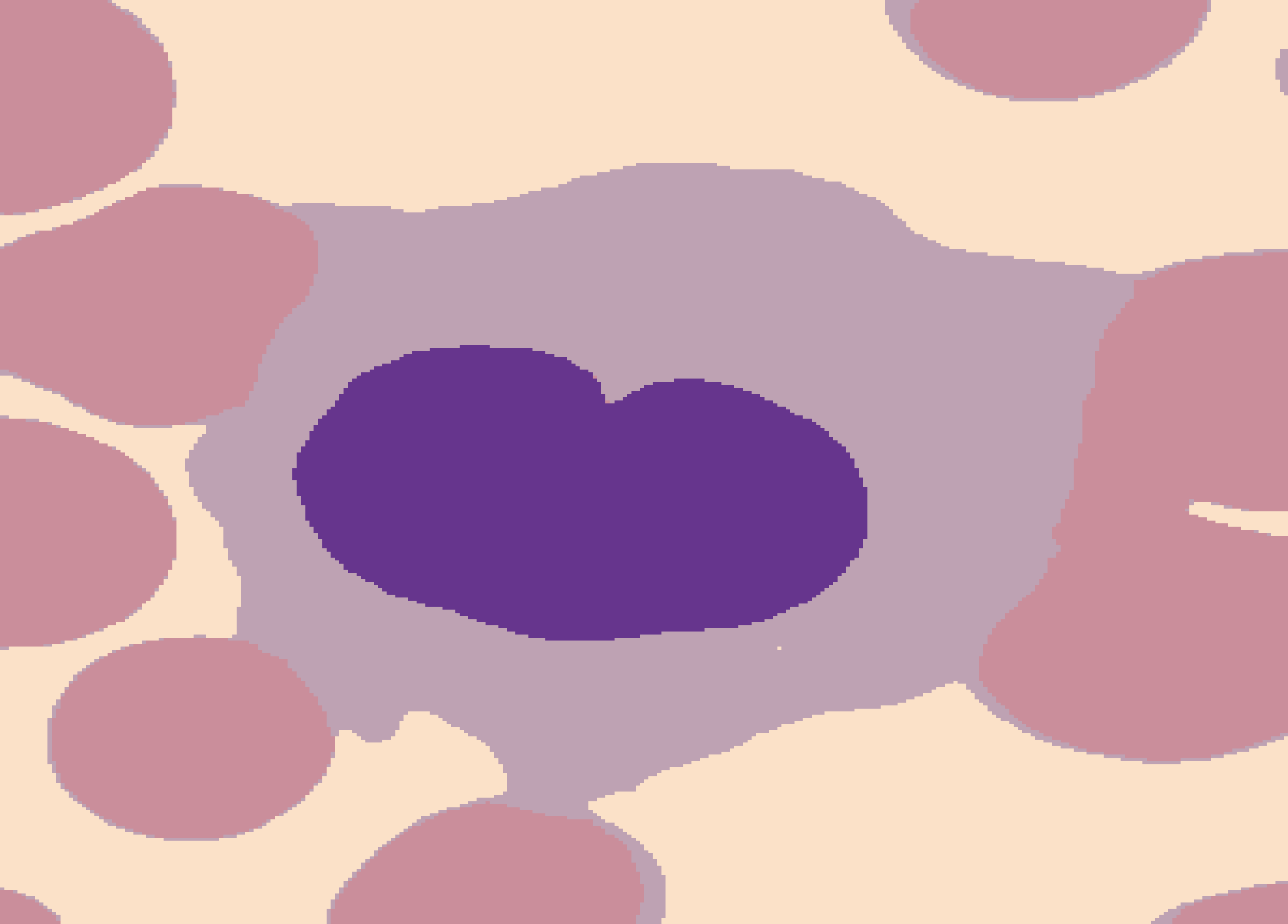}\vspace{1mm}
	\includegraphics[width=2.0cm]{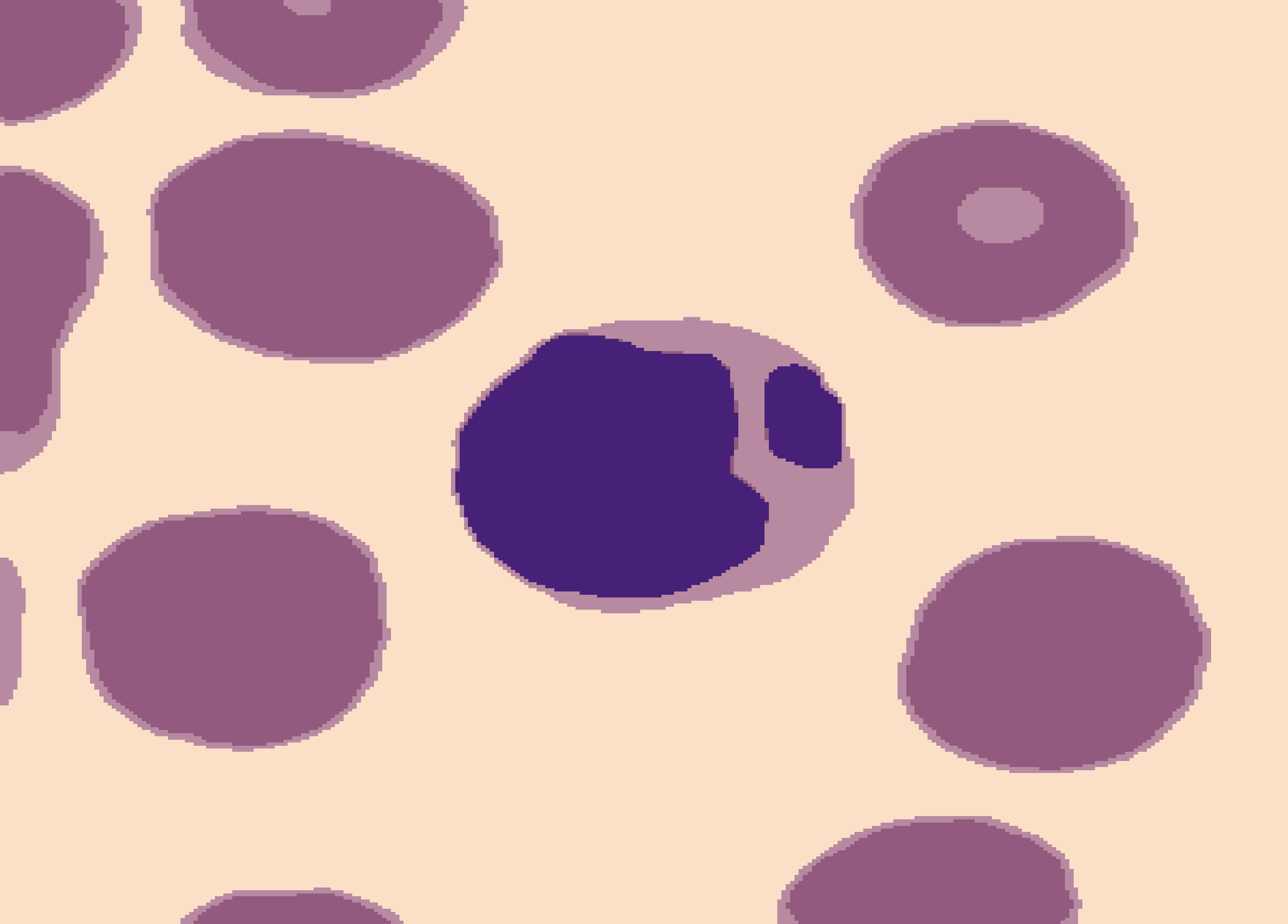} \vspace{1mm}\vspace{1mm}
 \includegraphics[width=2.0cm]{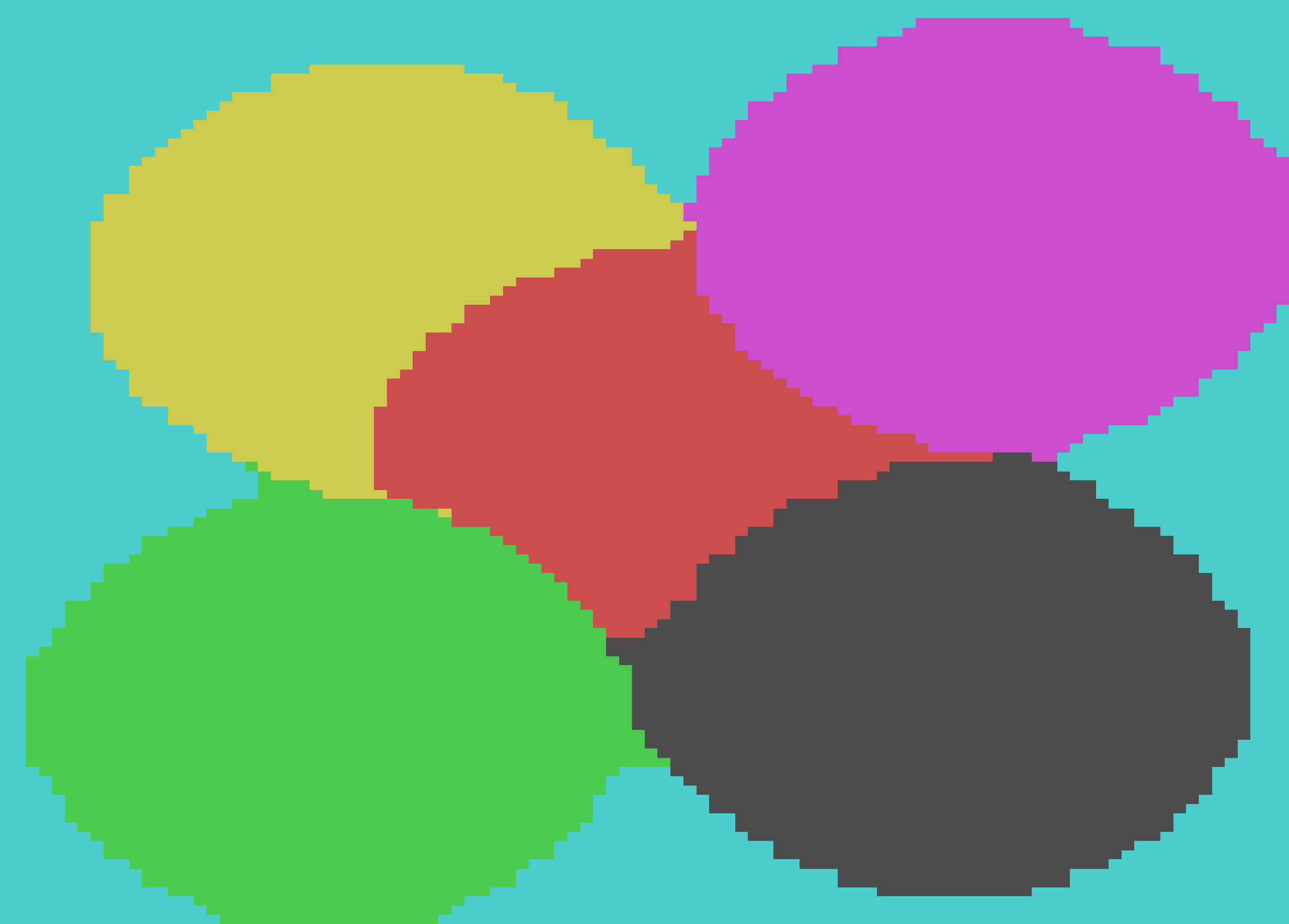} \vspace{1mm}
	\end{minipage}
	}
  \hspace{-11.8mm}
	\subfigure{
	\begin{minipage}[c]{0.2\textwidth}
    \includegraphics[width=2.0cm]{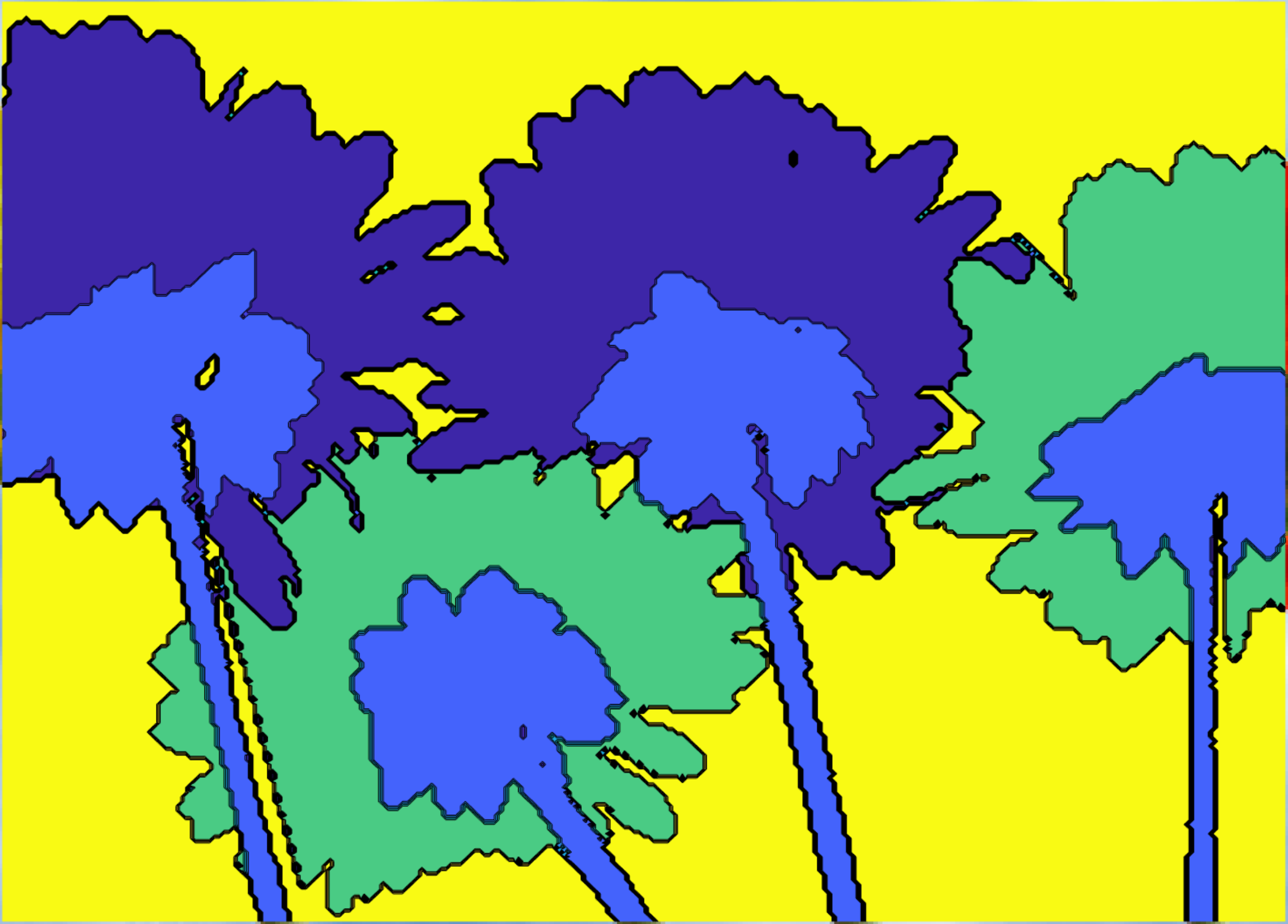}\vspace{1mm}
	\includegraphics[width=2.0cm]{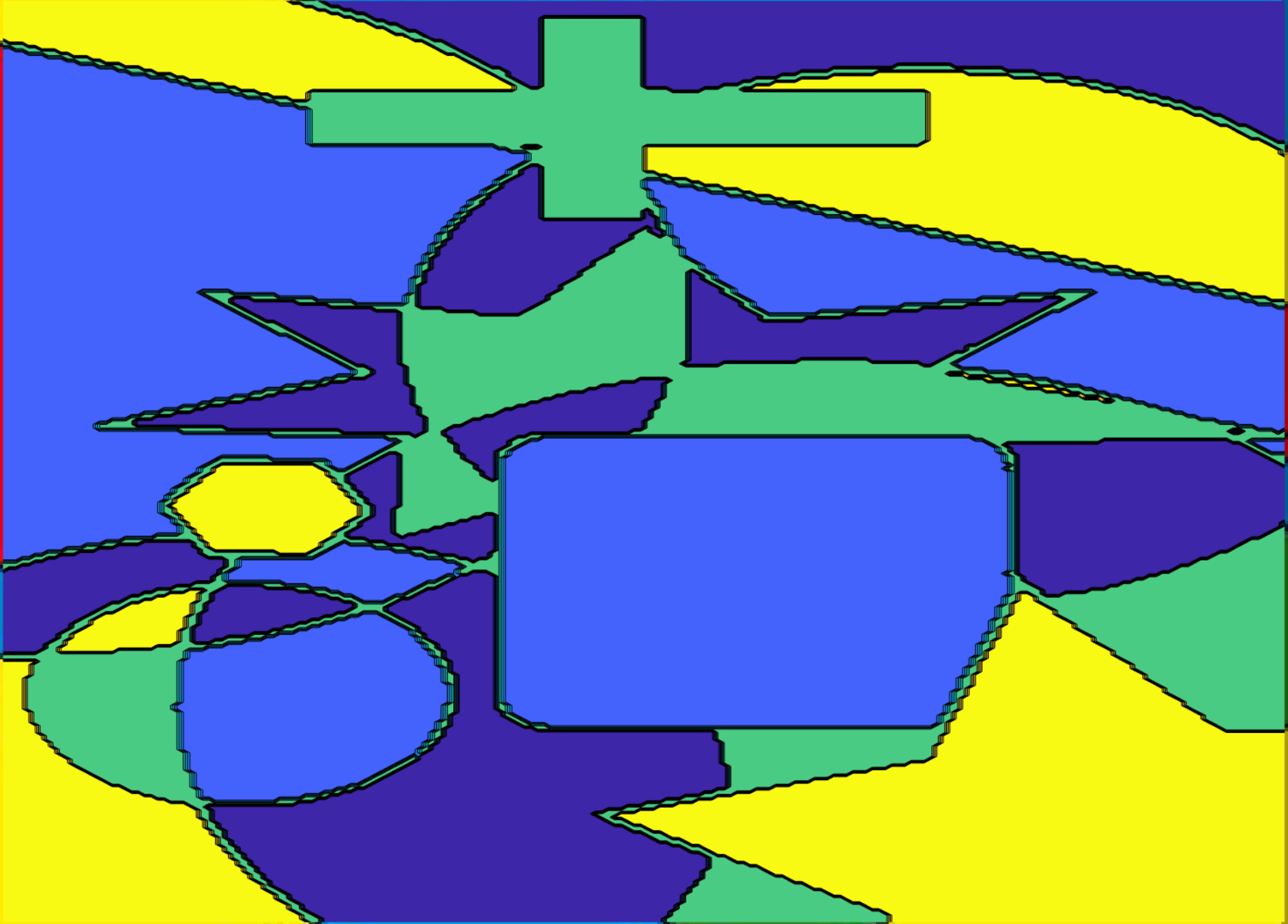}\vspace{1mm}
    \includegraphics[width=2.0cm]{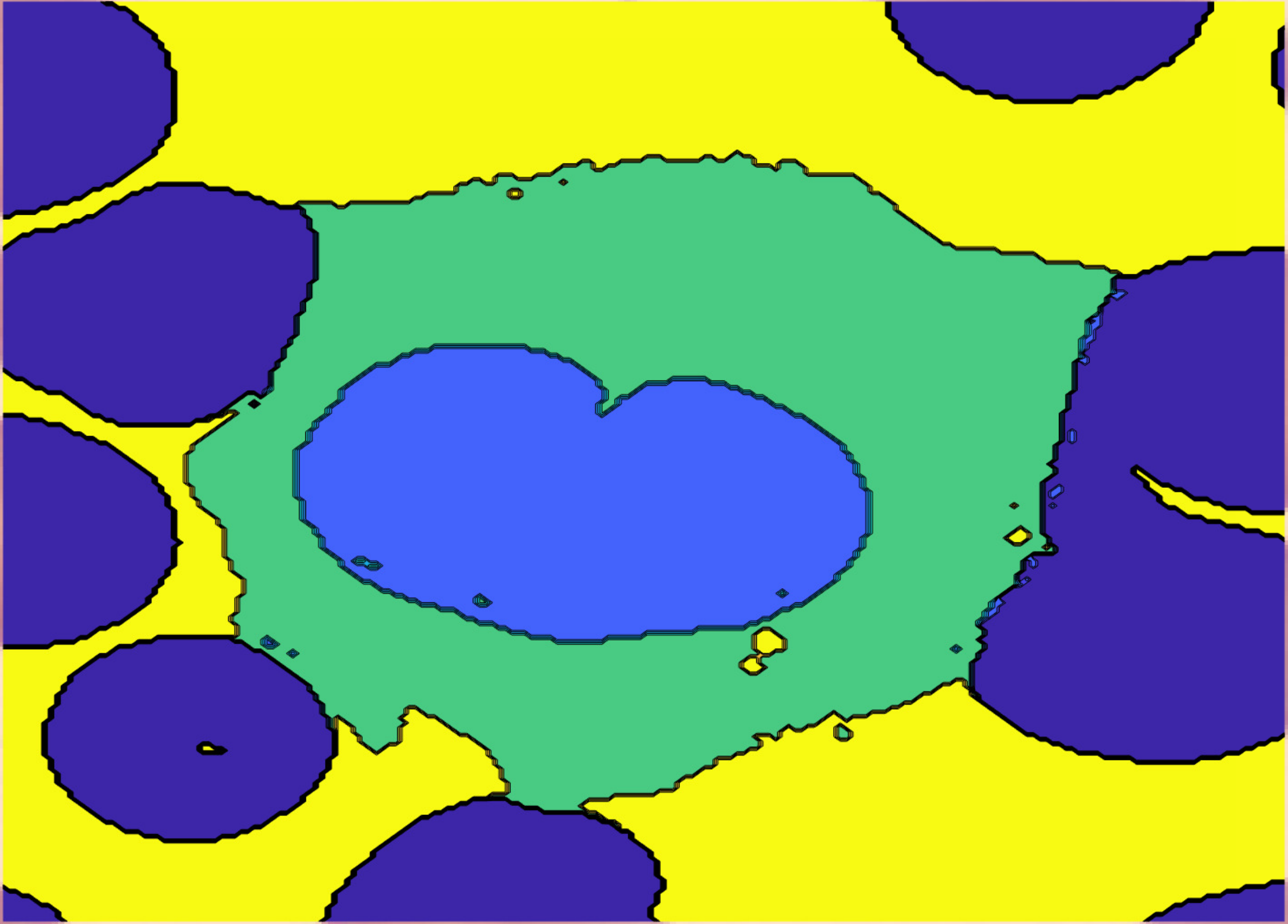}\vspace{1mm}
	\includegraphics[width=2.0cm]{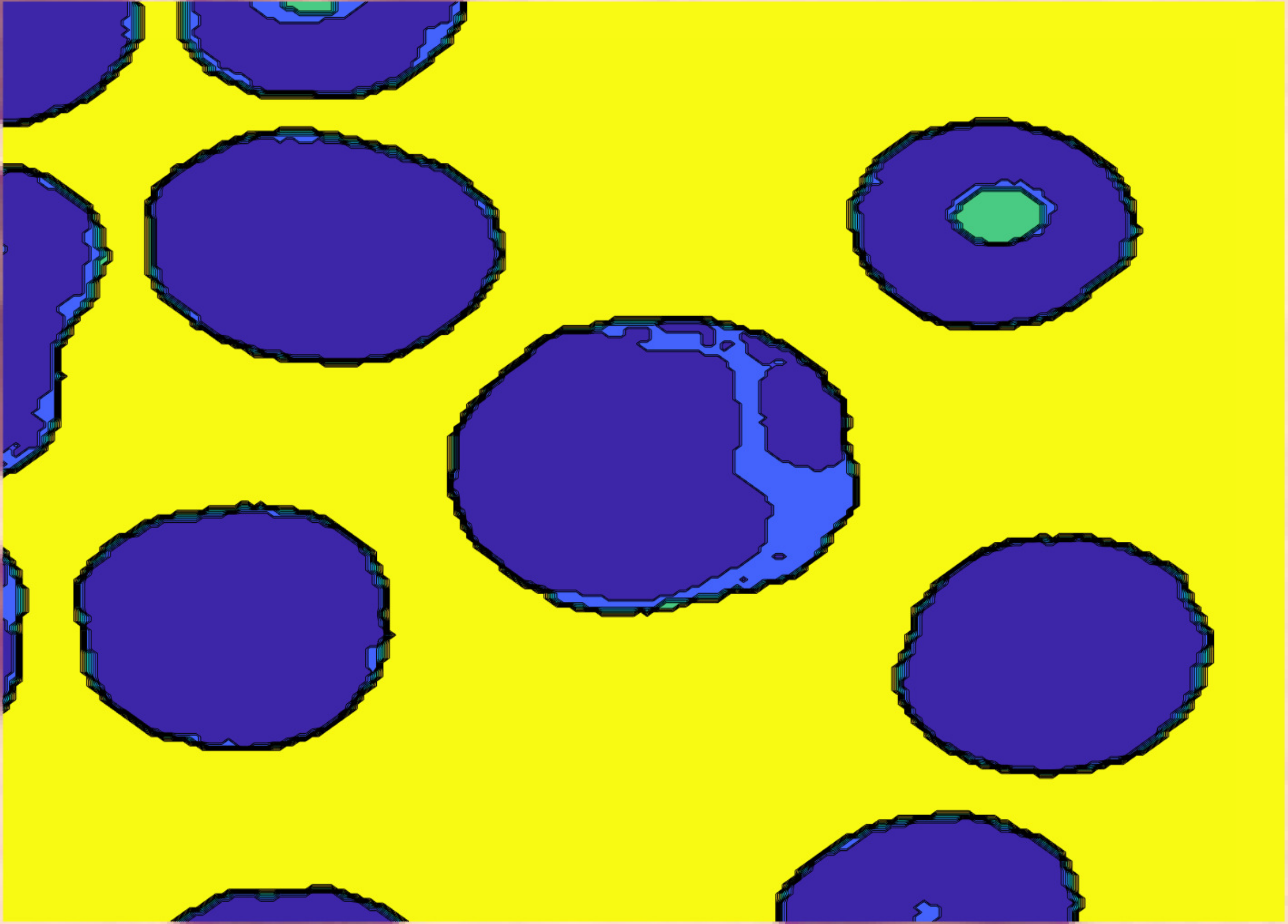}\vspace{1mm}\vspace{1mm}
    \includegraphics[width=2.09cm]{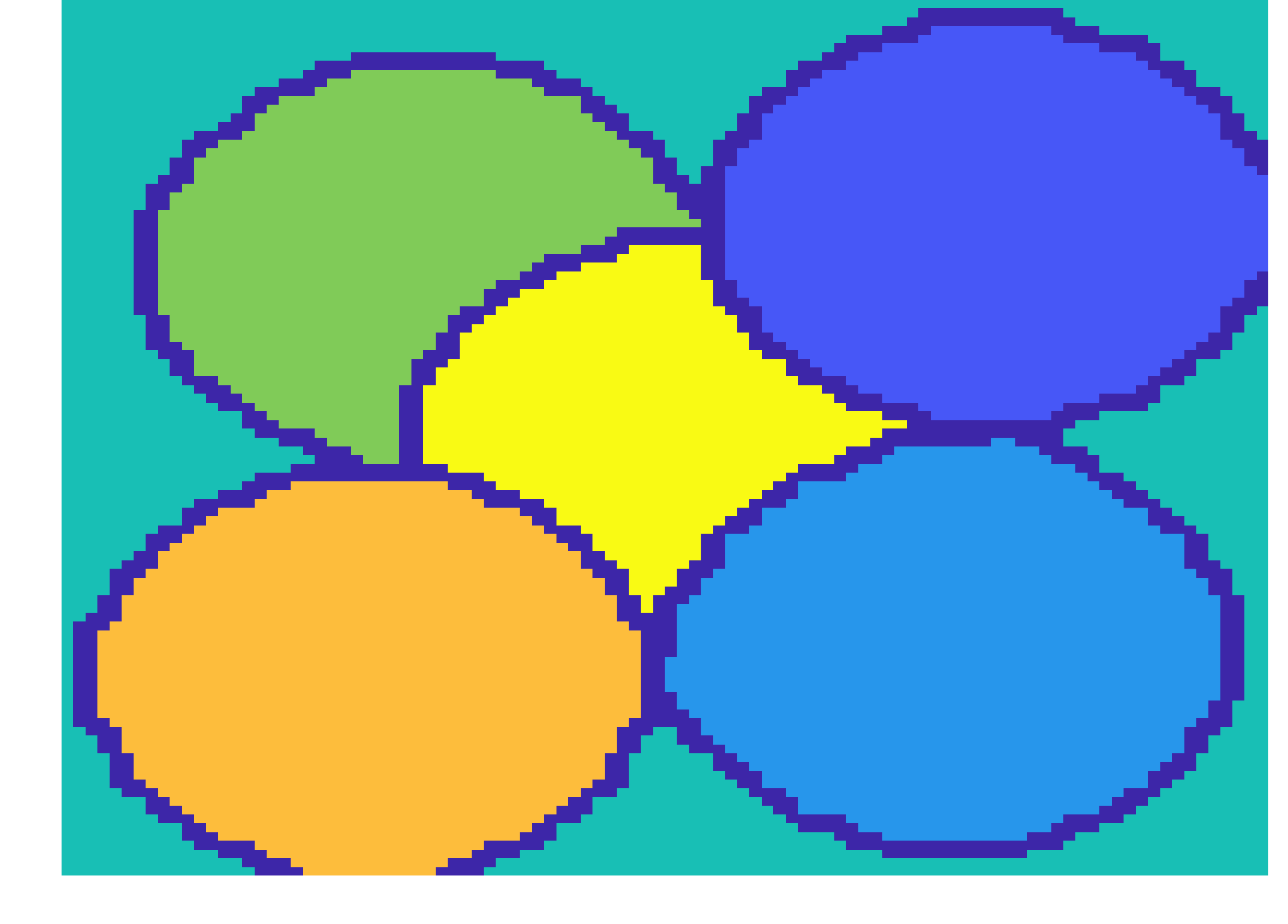}\vspace{1mm}
	\end{minipage}
	}
\hspace{-11.8mm}
	\subfigure{
	\begin{minipage}[c]{0.2\textwidth}
    \includegraphics[width=2.0cm]{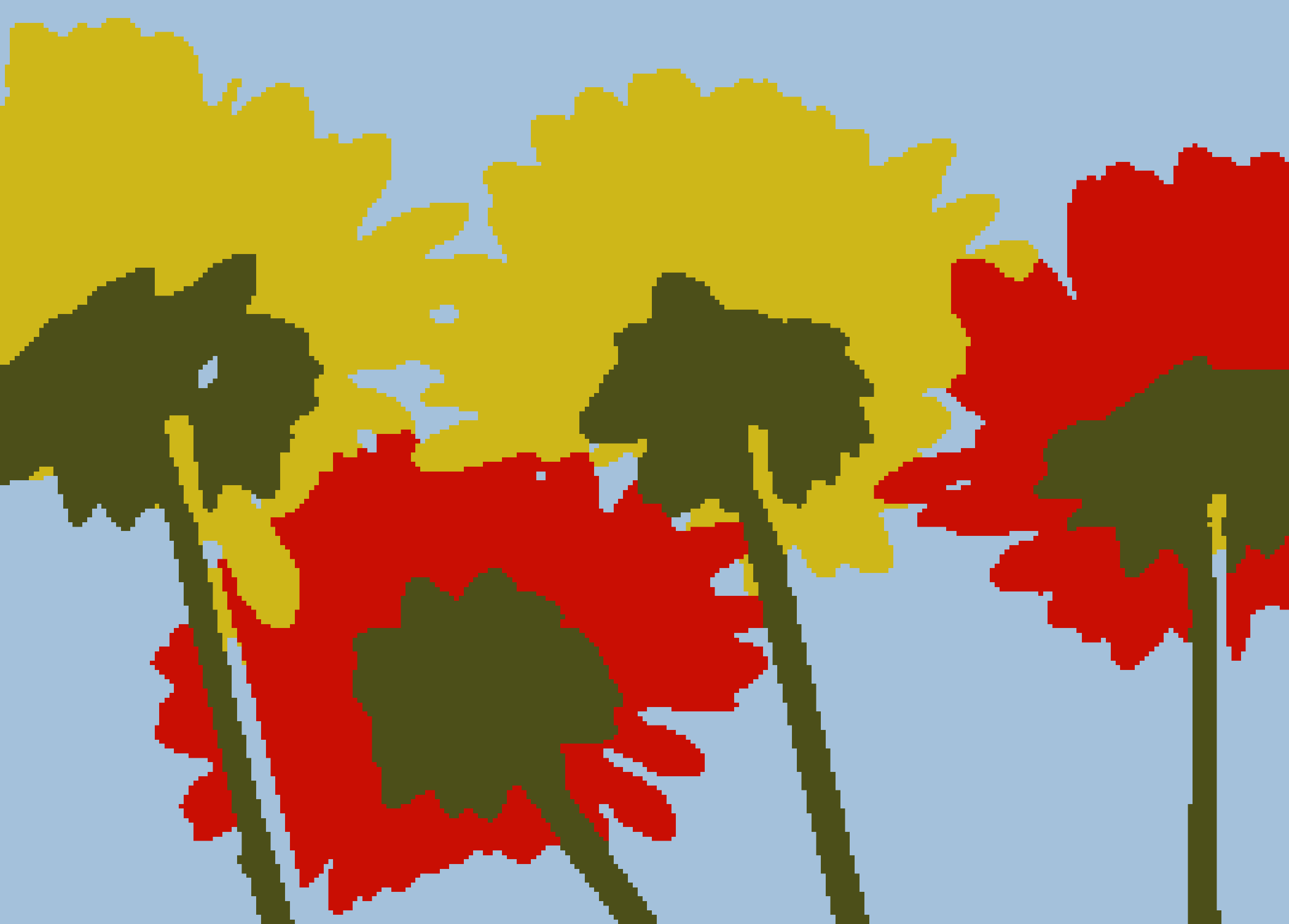}\vspace{1mm} \vspace{1mm}	  
	\includegraphics[width=2.0cm]{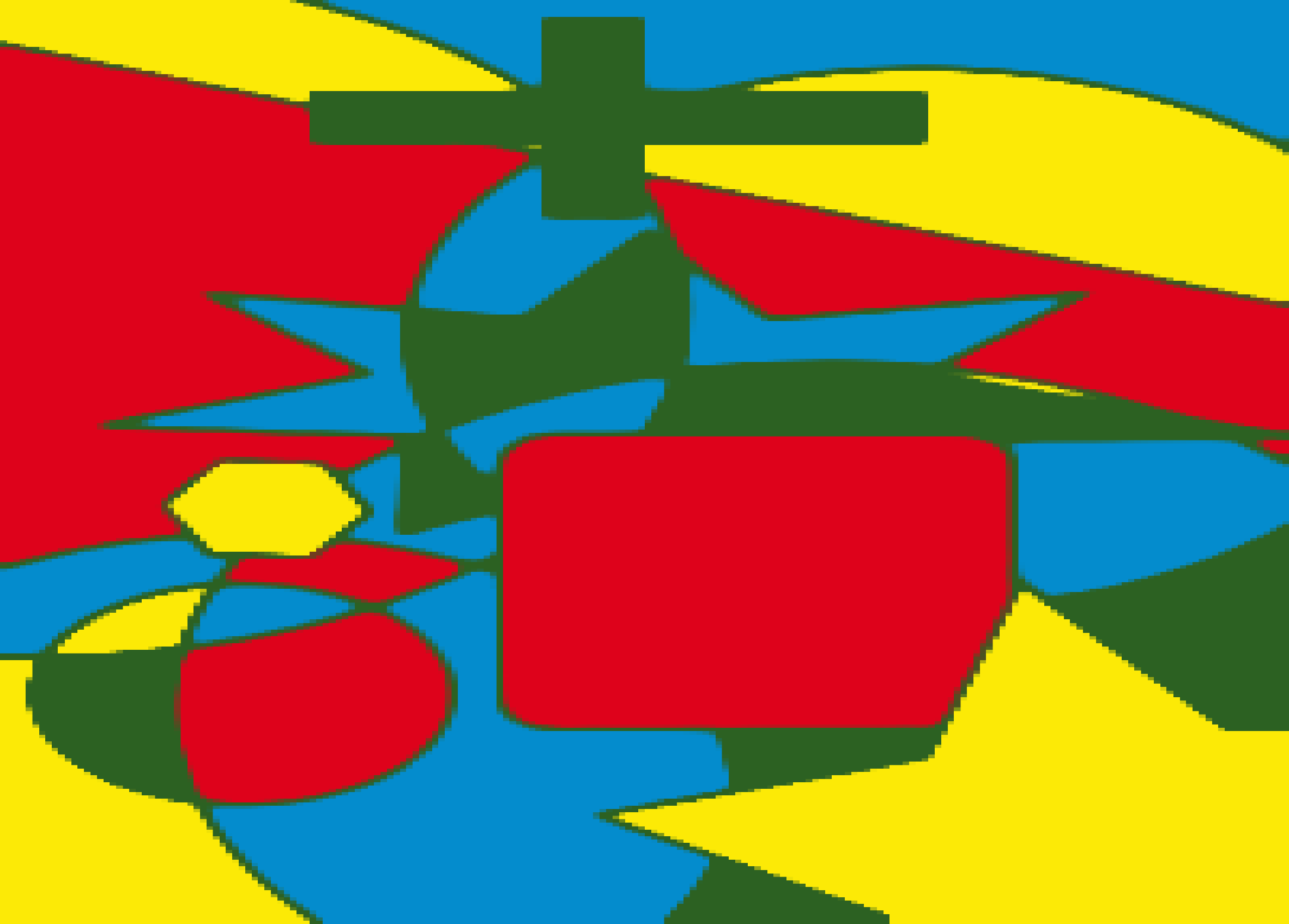} \vspace{1mm} 
    \includegraphics[width=2.0cm]{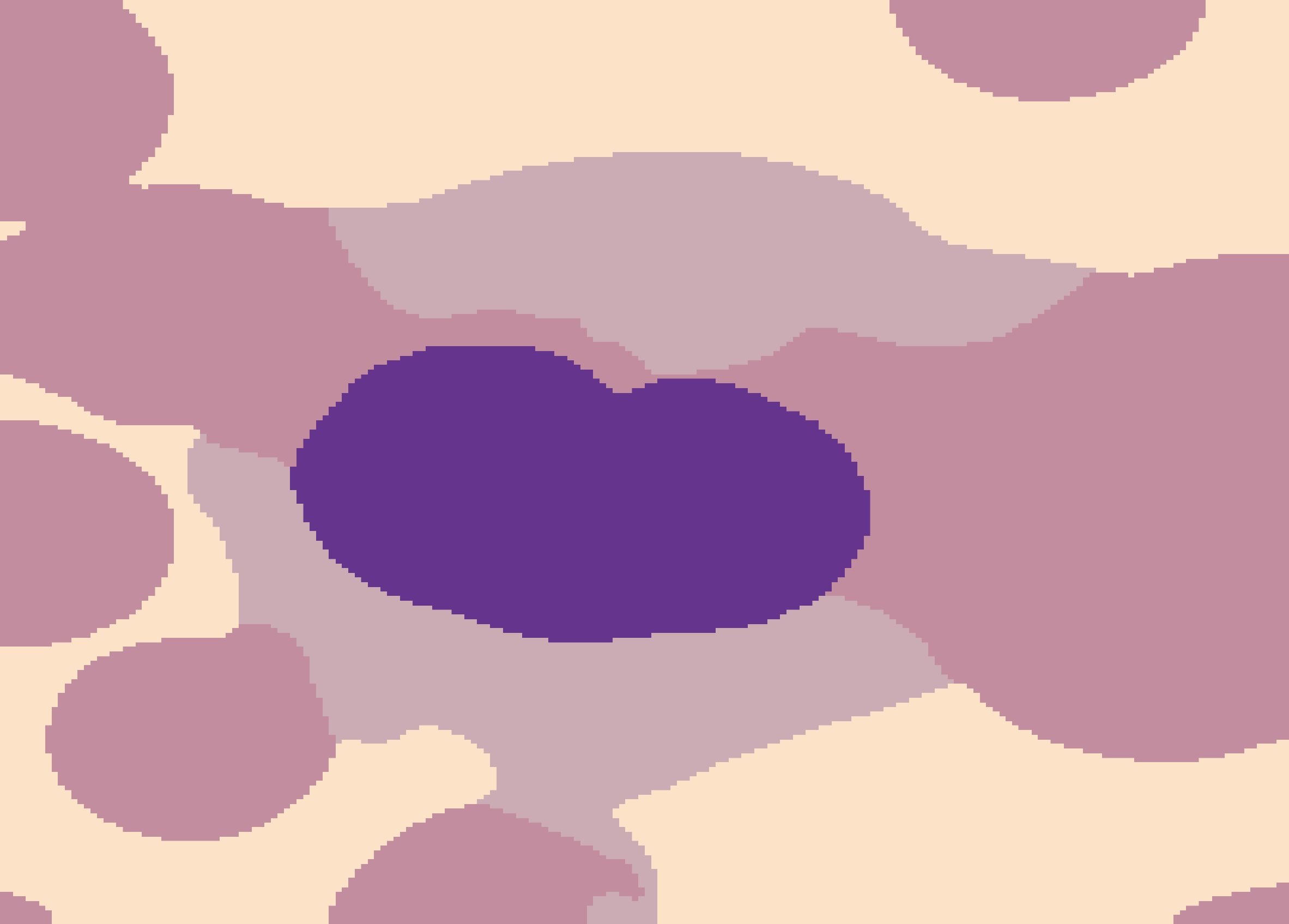}\vspace{0.5mm} 	
	\includegraphics[width=2.0cm]{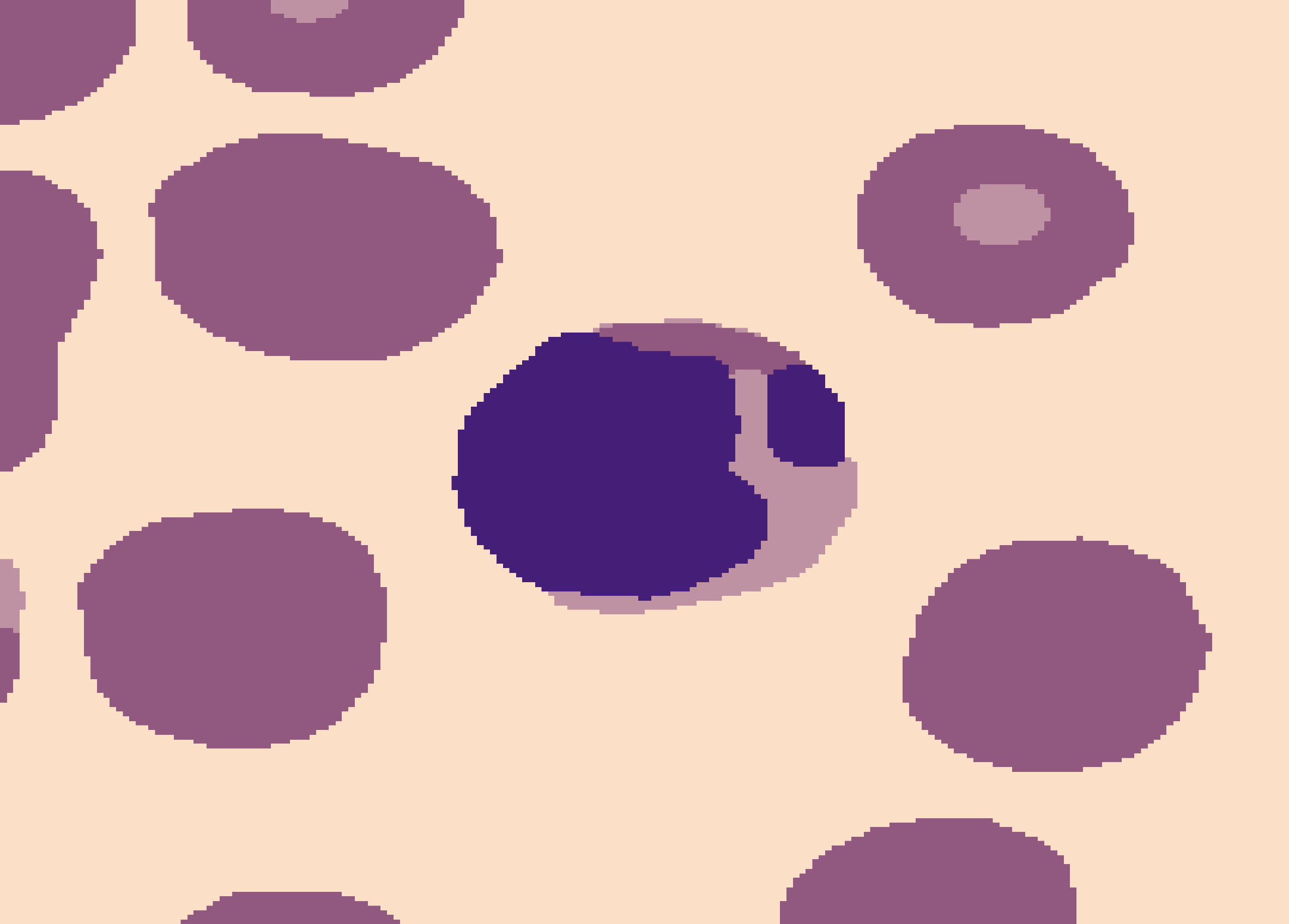} \vspace{1mm}\vspace{1mm}
\includegraphics[width=2.0cm]{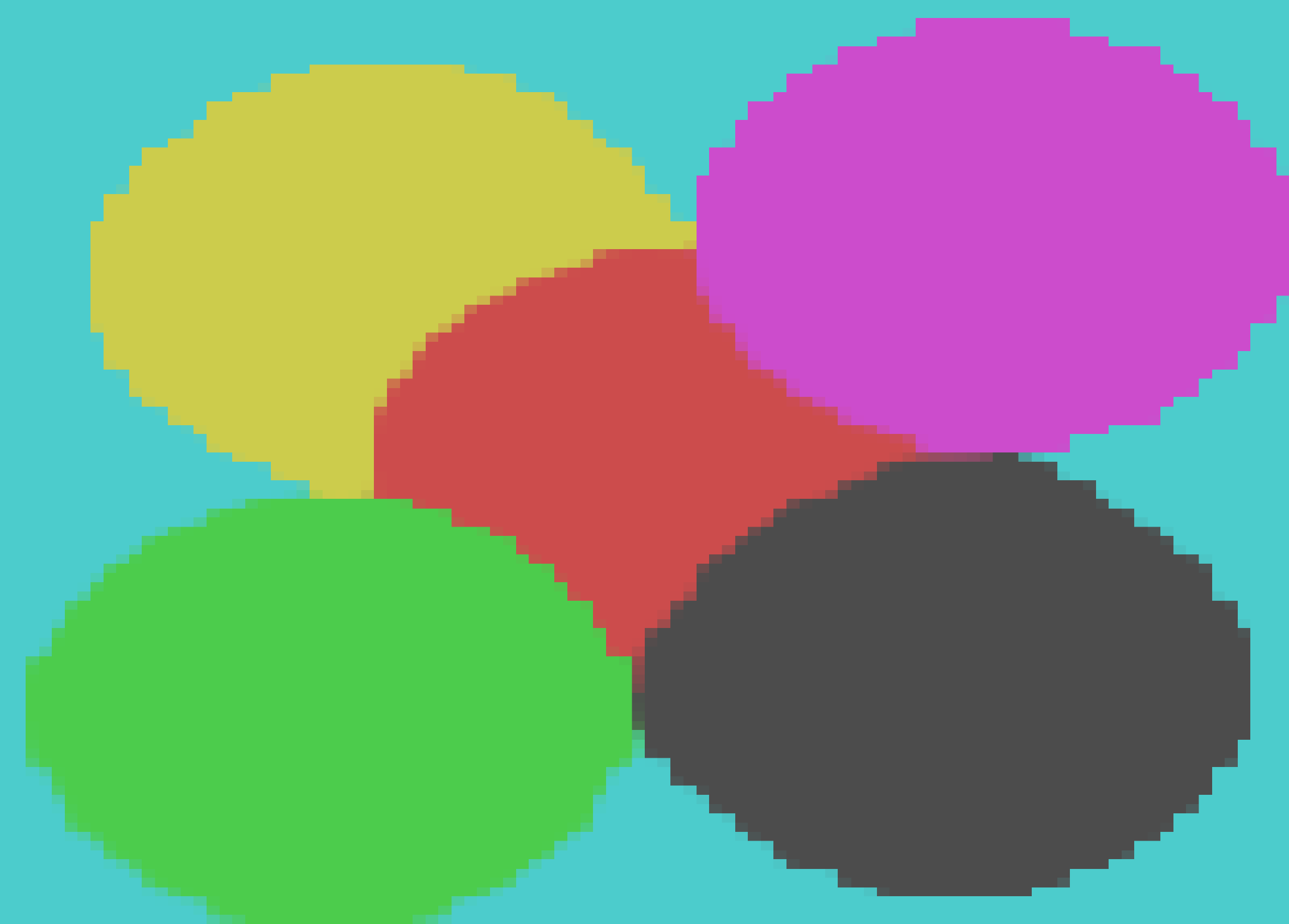} \vspace{1mm}
  \end{minipage}
	}
\hspace{-11.8mm}
	\subfigure{
	\begin{minipage}[c]{0.2\textwidth}
    \includegraphics[width=2.0cm]{ex22_seg-eps-converted-to.pdf}\vspace{1mm}
	\includegraphics[width=2.0cm]{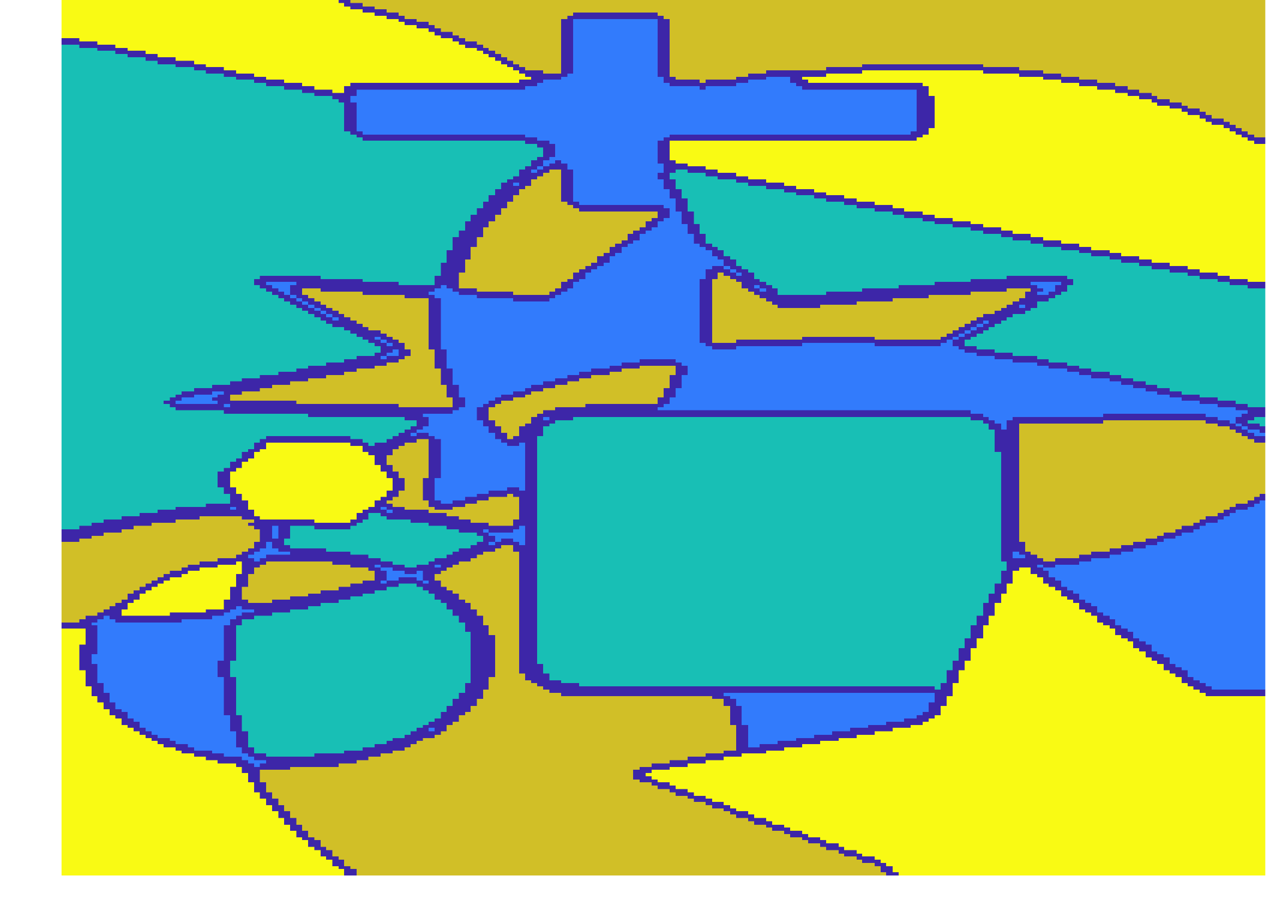} \vspace{1.mm}
    \includegraphics[width=2.0cm]{ex23_seg-eps-converted-to.pdf}\vspace{1.mm}
	\includegraphics[width=2.0cm]{ex24_seg-eps-converted-to.pdf} \vspace{1.mm}
 \includegraphics[width=2.0cm]{ex25_seg-eps-converted-to.pdf} \vspace{1mm}
	\end{minipage}
	}
}
\caption{ \label{fig:WBC} Segmentation comparison. For left to right:  Original image, CH\cite{Yang2019_JSC}, SLaT\cite{Cai2015JSC}, ICTM\cite{Wang_2017}, CKA\cite{Wu2021IET}, and our ACCV model. }
\end{figure}

\begin{table}[!b]
\begin{tabular}{cccccc}
\hline
 Image{\ }           & CH\cite{Yang2019_JSC}     &SLaT\cite{Cai2015JSC} &  ICTM\cite{Wang_2017}   &  CKA\cite{Wu2021IET}     & ACCV\\ \hline
Flowers       &  28.48  & 5.93    &  3.33    &  4.56     & 3.11 \\
 4colors       &  3.97   & 2.56    &  3.11    &  4.37     & 6.12        \\  
 WBC1          &  16.14  & 2.08   &  13.76   & 13.51     & {6.87}  \\ 
 WBC2          &  15.92  & 2.13   &  10.25   & 6.21      &{4.65}\\  
 6phases       &  17.37  & 0.75   &   3.74   & 6.15      & 4.76 \\\hline  
\end{tabular}
\caption{\label{tab:CPU_time_comparison}  Segmentation comparison on CPU time for CH\cite{Yang2019_JSC}, SLaT\cite{Cai2015JSC}, ICTM\cite{Wang_2017}, CKA\cite{Wu2021IET}, and our ACCV model.} 
\end{table}
\subsection{Experiments on the capability for noise images}
{Through two corrupted image ``4blocks-1" and ``4blocks-2"  by different levels of Gaussian noise}, this subsection exhibits the robustness of our model for noise.
For the sake of illustration, three celebrated and effective segmentation methods are adopted for comparison, including SLaT \cite{Cai2015JSC}, ICTM \cite{Wang_2017}, and CKA\cite{Wu2021IET}.
{In this example, we set $\sigma = 0.02, M = 1$ for the Multi-IGLIM, and $\epsilon = 60.0,\lambda = 40, h = 1.0, S=80 $ for the ACCV model.}


We first display the segmentation results of the polluted ``4blocks-1" in Figure \ref{fig:noise2}. The difficulty of the segmentation lies in the intersection of four grayscale squares as well as the intersecting edge with inhomogeneous intensity. Moreover, the close intensities lead that some segmentation methods regard the white and light gray blocks as one segment. From the first row to the fourth row, we gradually increase the Gaussian noise level by taking the variance from 0.01,0.02,0.03,0.04, respectively. ICTM shows not only the four desired objects but also all the noise points, which shows its sensitivity to noise. In spite of the fact that the cost CPU time of SLaT is the least for this image, the segmentation results of SLaT and CKA have blurry edges, which cause inaccurate partition.

{Then a collection of corrupted images ``4blocks-2"  are devoted to doing the multi-phase segmentation in Figure \ref{fig:noise1}. The added Gaussian noise has zero average with variances 0.01,0.03 and 0.05 from top to bottom. The results show that there exists an extra part on the top-right block of the result of SLaT, compared with the ideal segmentation. The behavior of ICTM is still limited by the noise points. For CKA,  the segmented edges of four blocks are covered by some blur. By contrast, our model provides a relatively satisfactory segmentation. The CPU time for this part is recorded in Table \ref{tab:noise}, which also illustrates the efficiency of our ACCV model.}

\begin{figure}[!t]
\resizebox{\textwidth}{!}{\hspace{11.8mm}
	\centering
\subfigure{
	\begin{minipage}[c]{0.2\textwidth}
		\includegraphics[width=2.0cm]{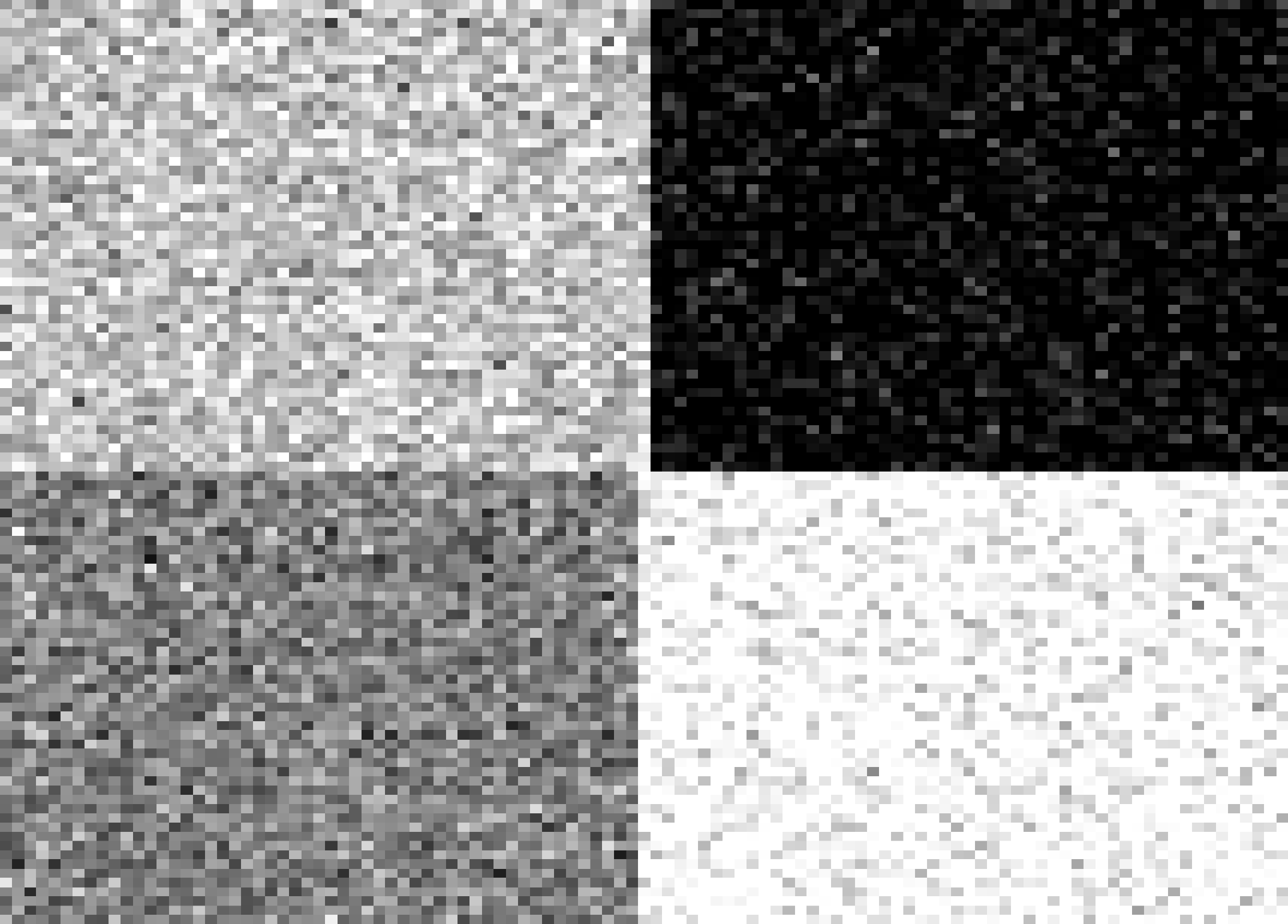} \vspace{1mm}
		\includegraphics[width=2.0cm]{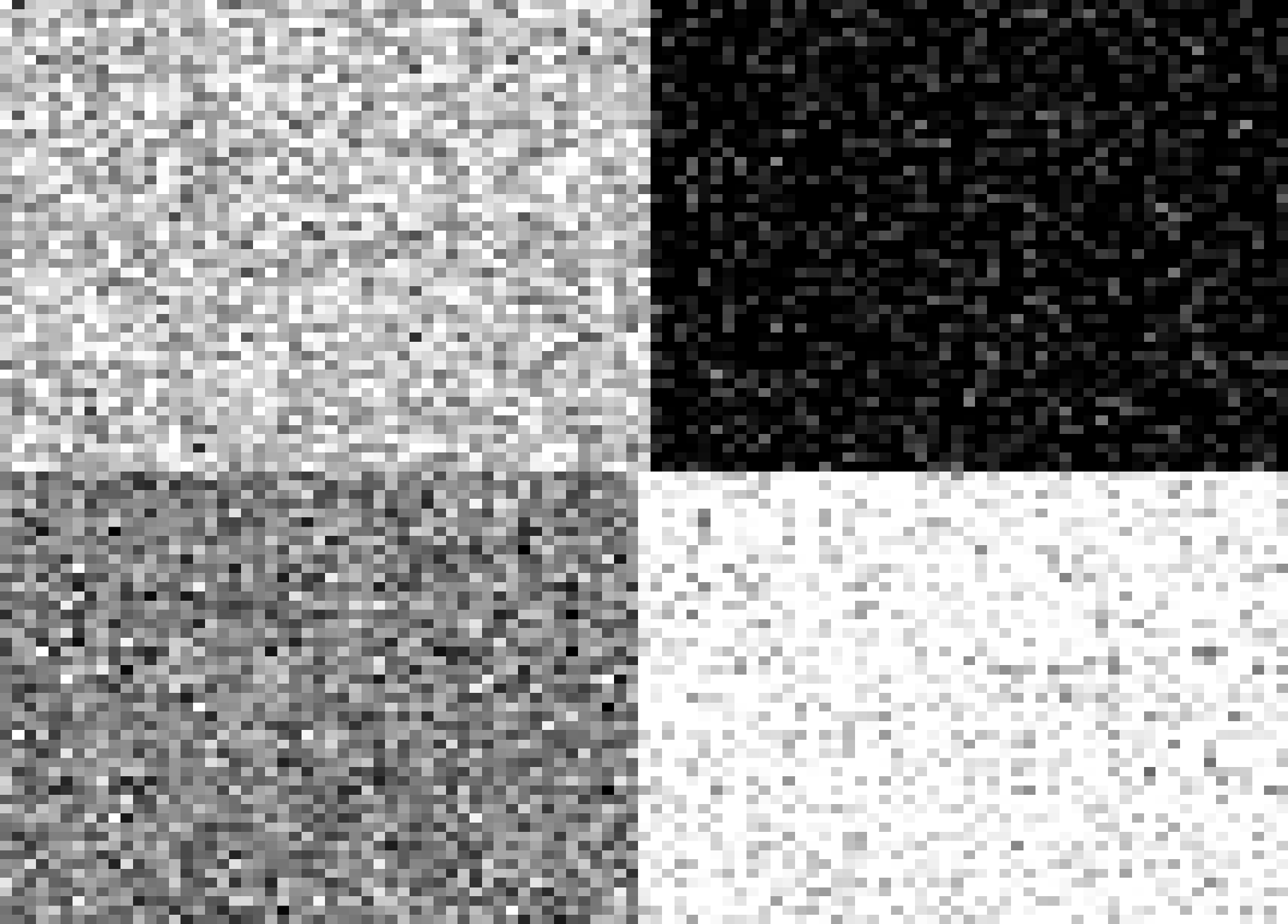} \vspace{1mm}
         \includegraphics[width=2.0cm]{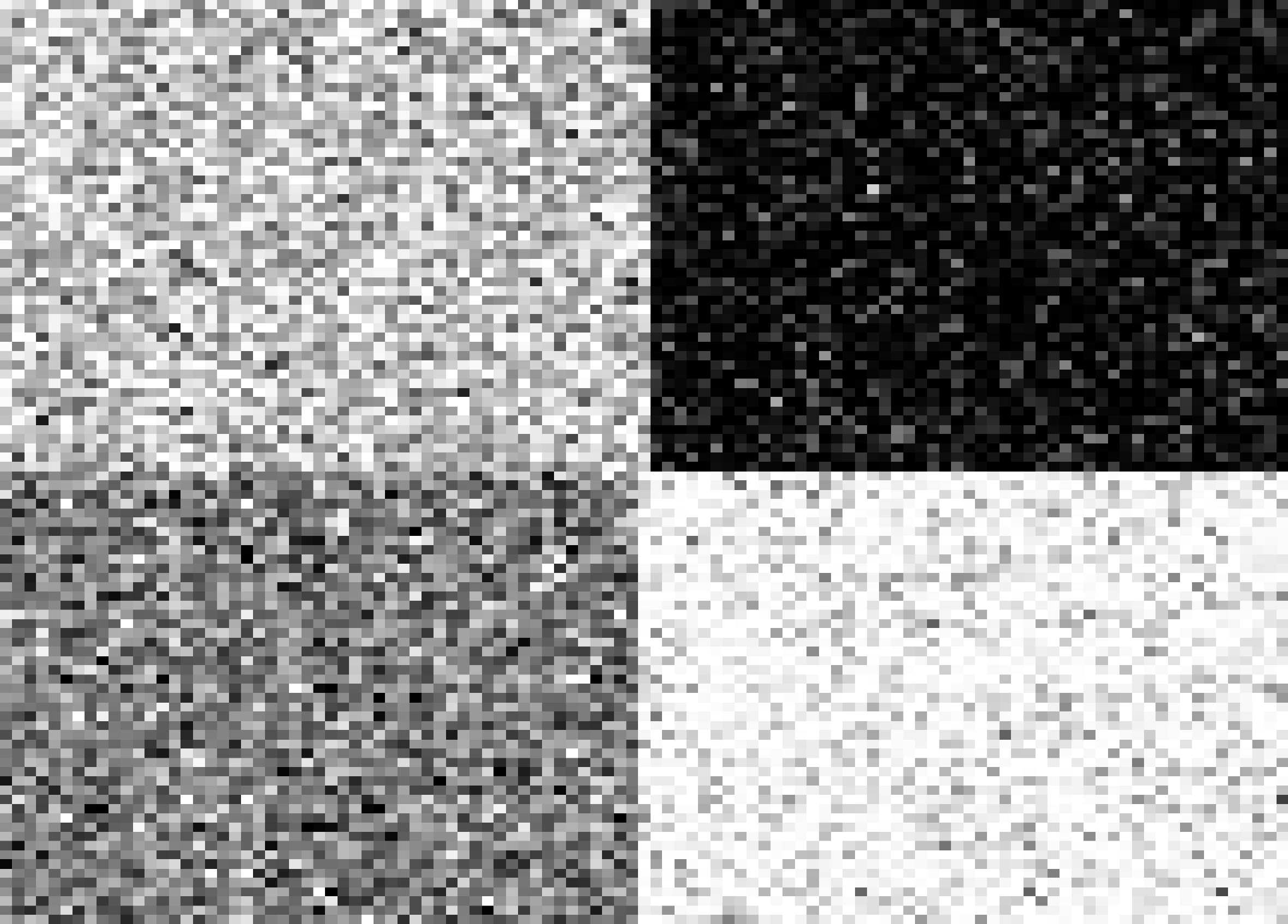}\vspace{2mm}
	\end{minipage}
	}
\hspace{-11.8mm}
	\subfigure{
	\begin{minipage}[c]{0.2\textwidth}
	\includegraphics[width=2.0cm]{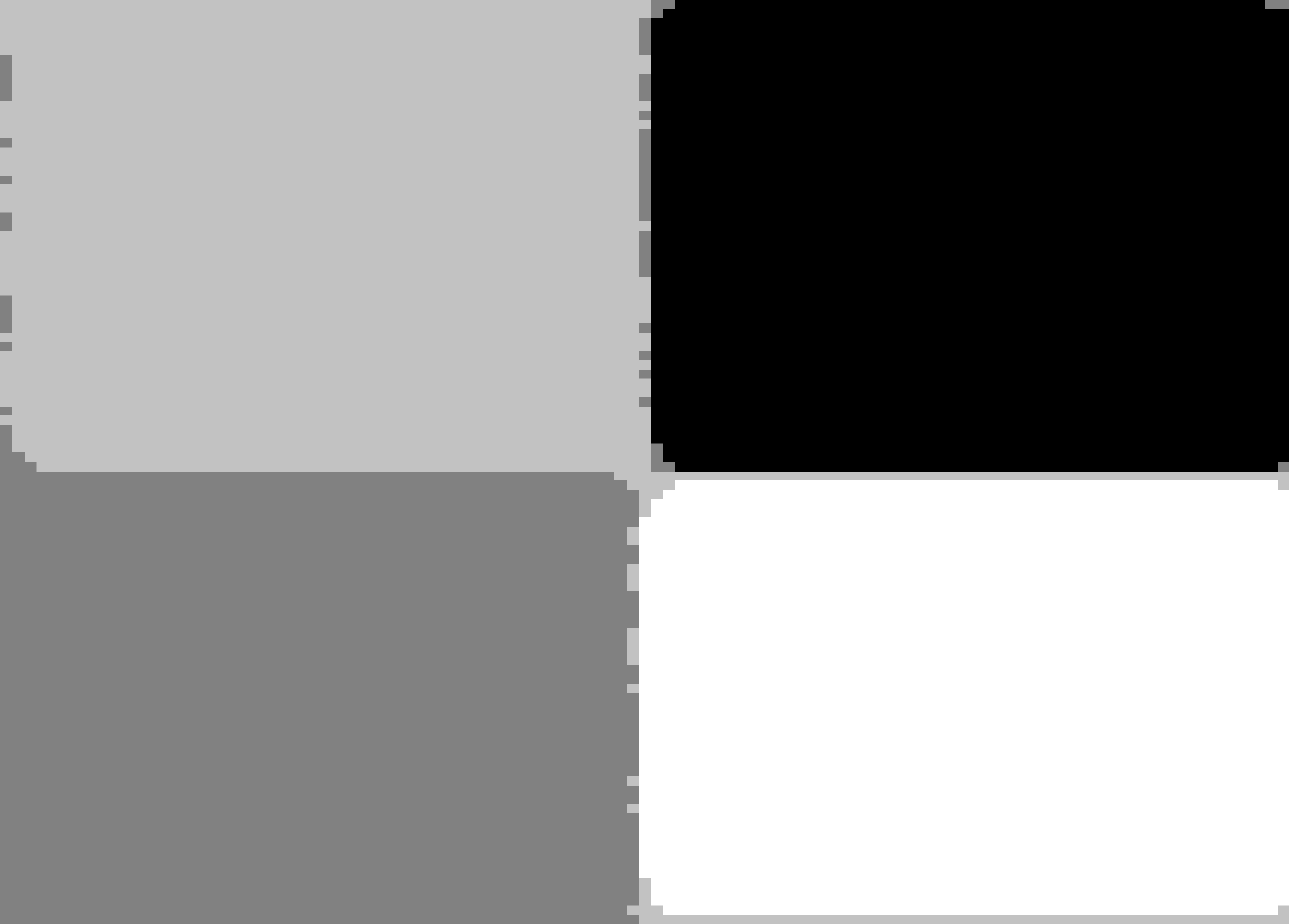}\vspace{1mm}
	\includegraphics[width=2.0cm]{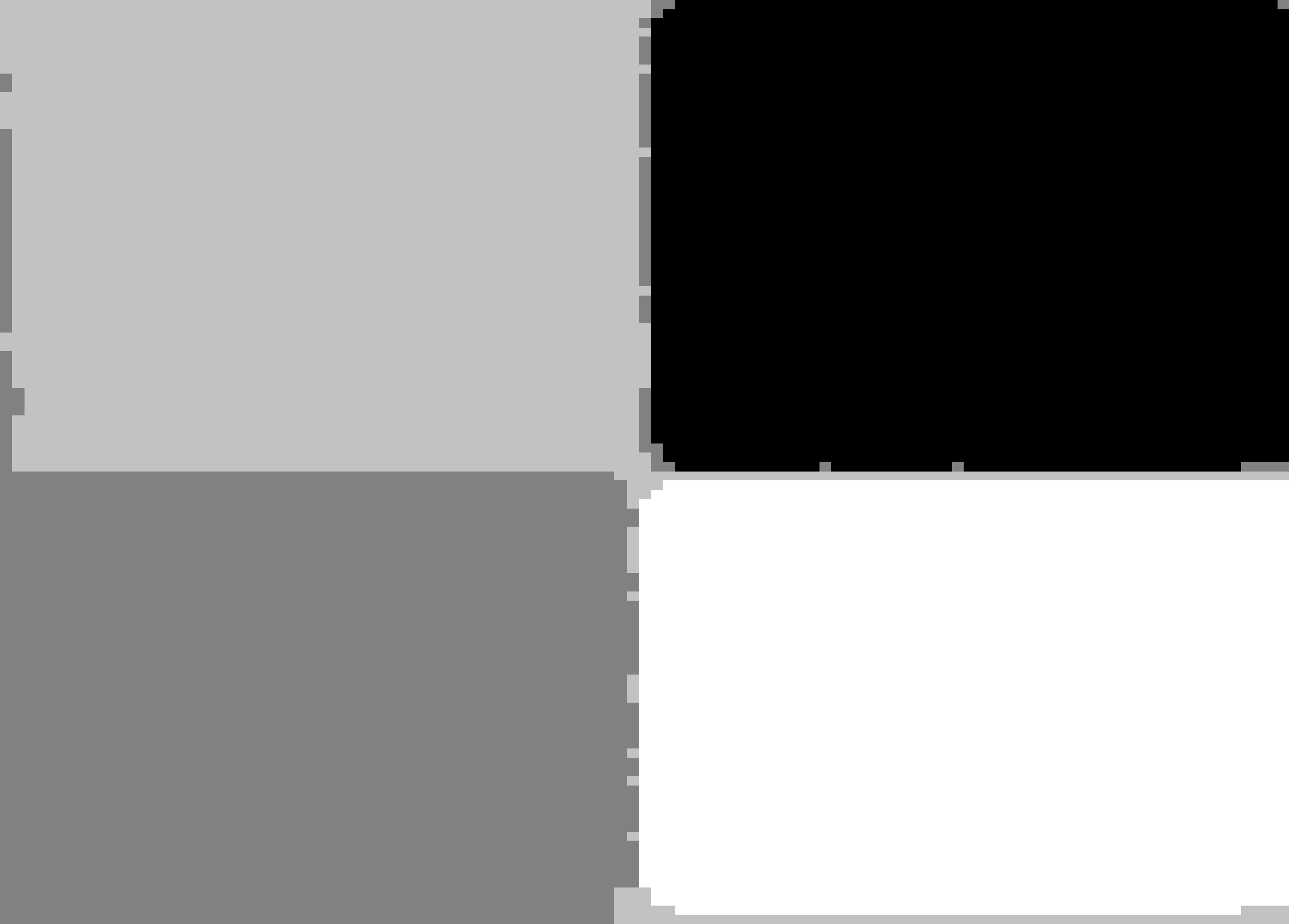}\vspace{1mm}
	\includegraphics[width=2.0cm]{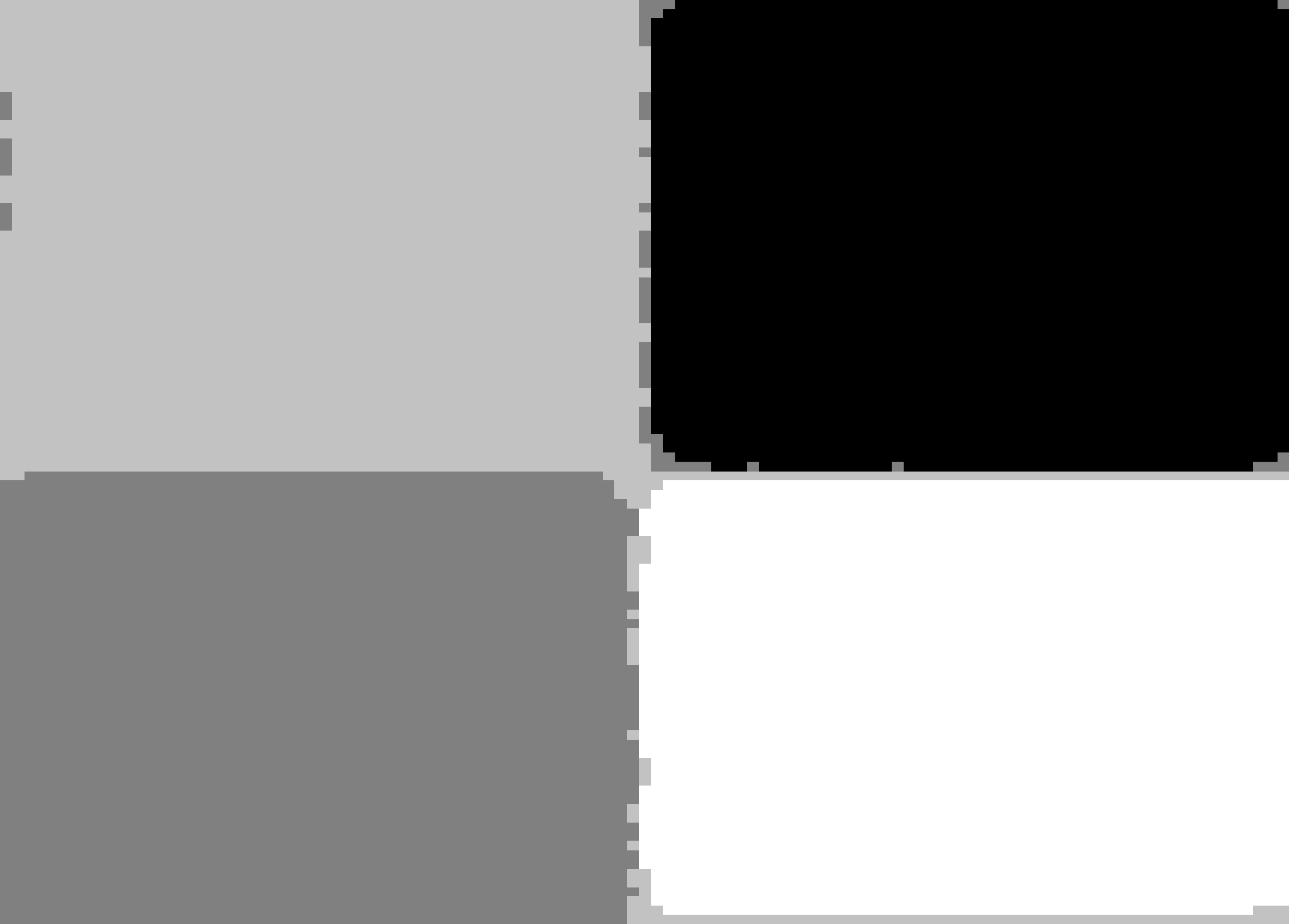} \vspace{2mm}
	\end{minipage}
	}
\hspace{-11.8mm}
	\subfigure{
	\begin{minipage}[c]{0.2\textwidth}
	\includegraphics[width=2.0cm]{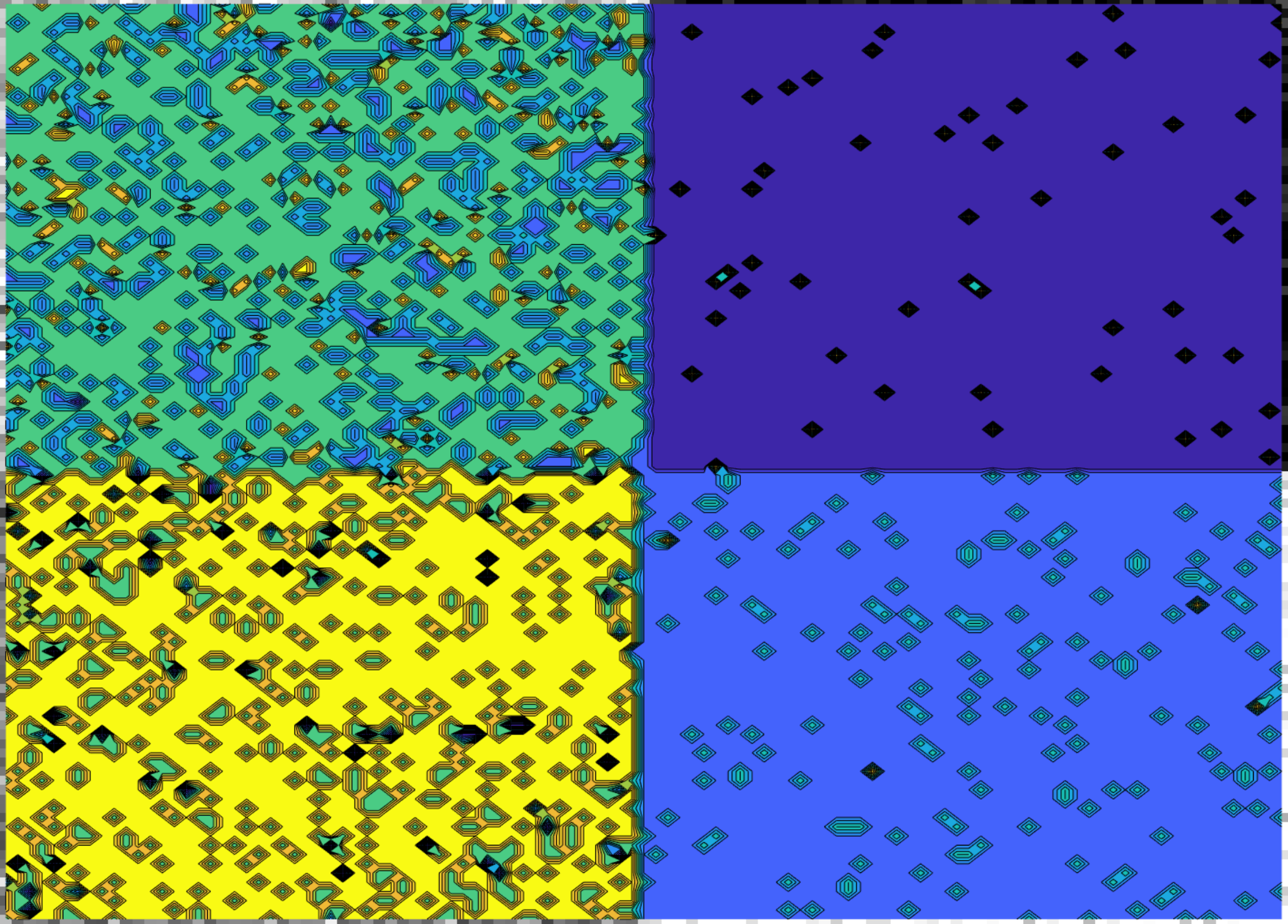}\vspace{1mm}
	\includegraphics[width=2.0cm]{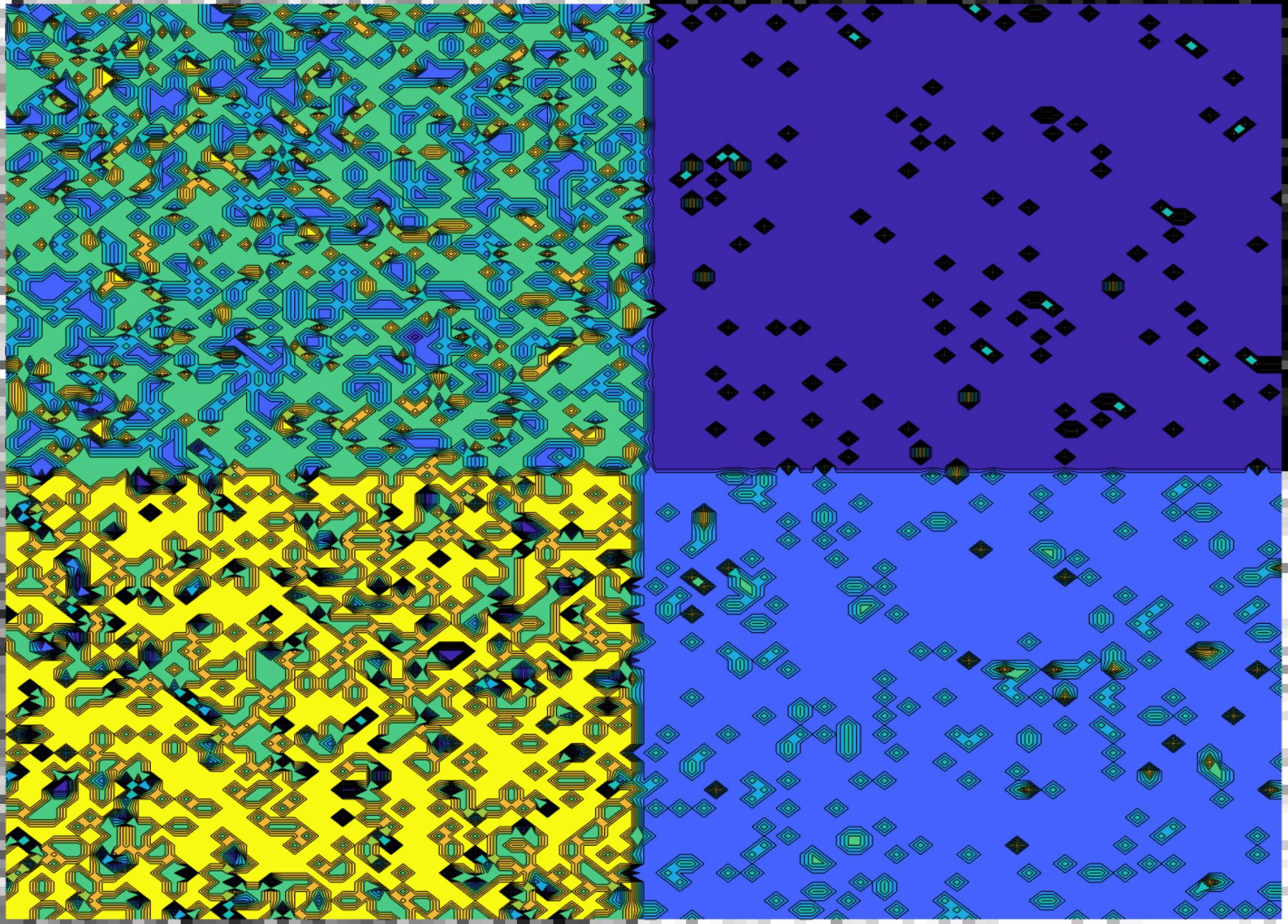}\vspace{1mm}
	\includegraphics[width=2.0cm]{ICTM_block_noise3_seg-eps-converted-to.pdf} \vspace{2mm}
	\end{minipage}
	}
\hspace{-11.8mm}
\subfigure{
	\begin{minipage}[c]{0.2\textwidth}
	\includegraphics[width=2.0cm]{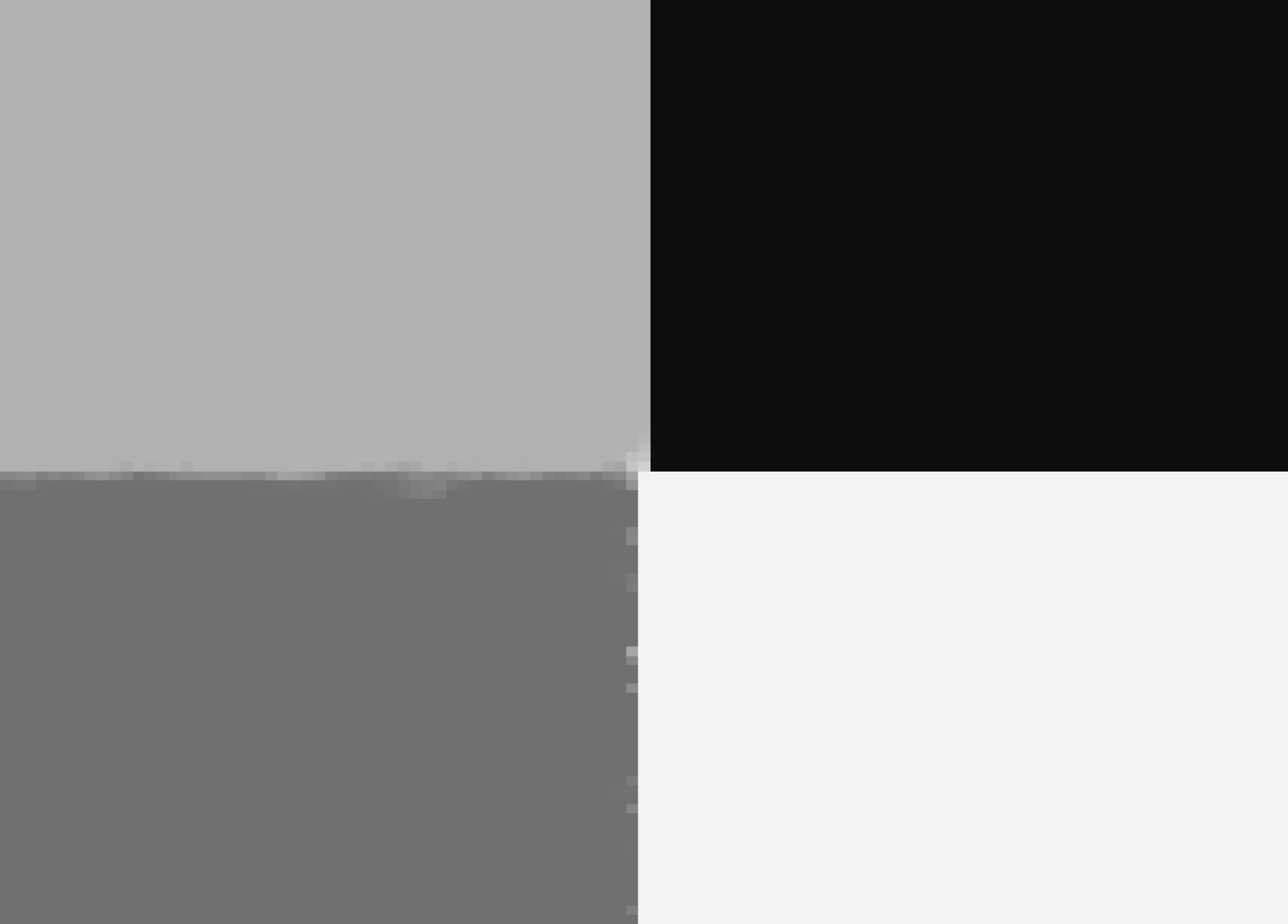}\vspace{1mm}
	\includegraphics[width=2.0cm]{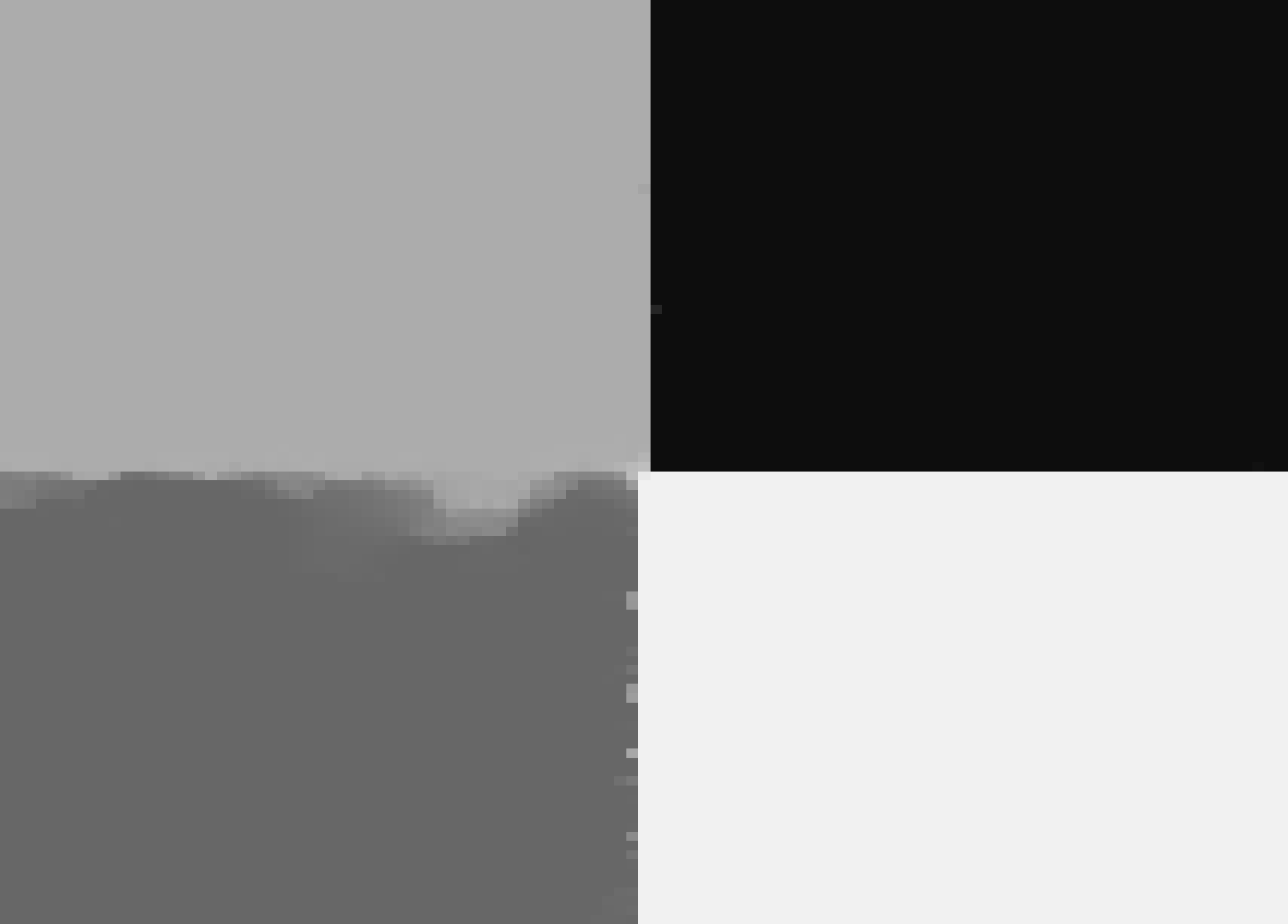}\vspace{1mm}
	\includegraphics[width=2.0cm]{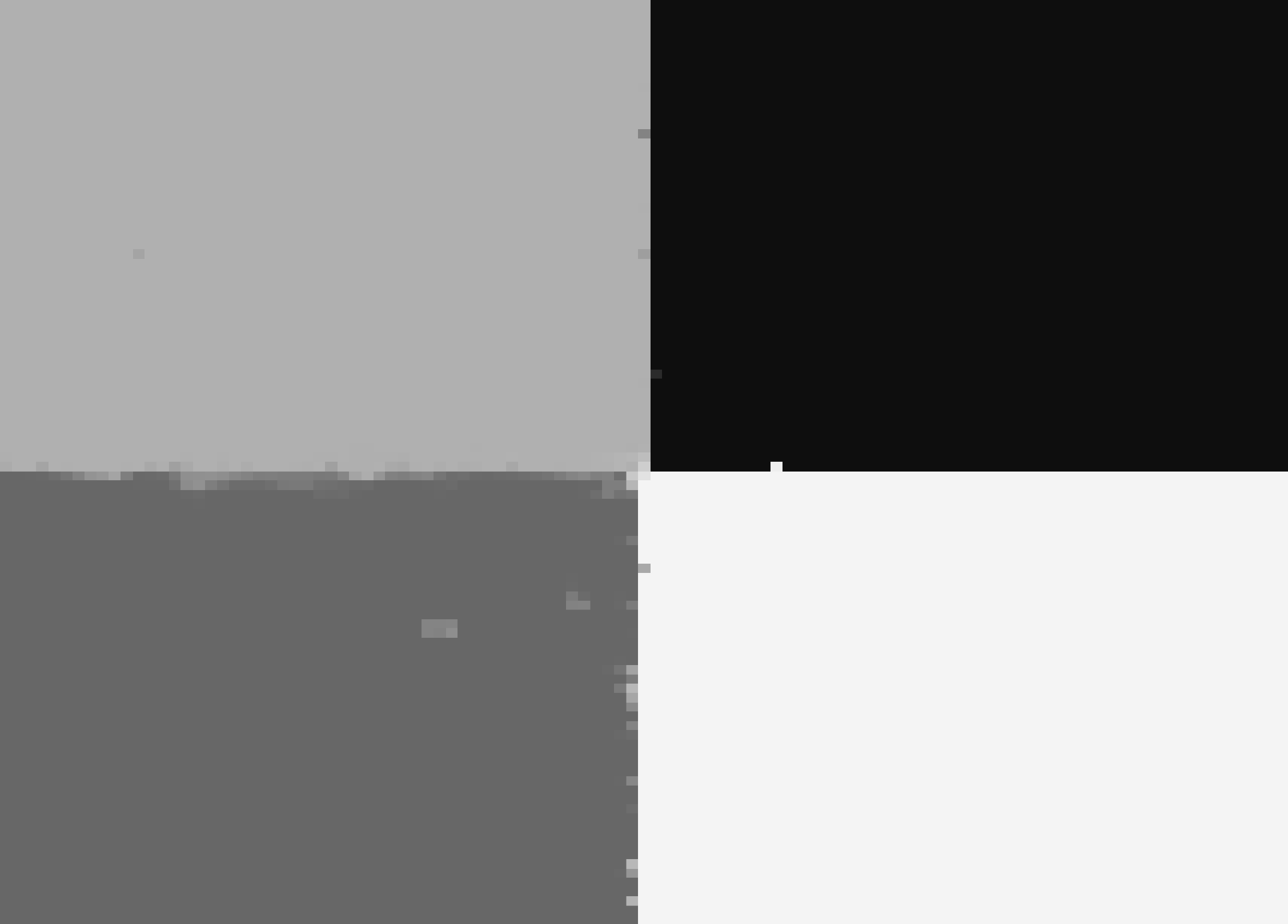} \vspace{2mm}
	\end{minipage}
	}
 \hspace{-11.8mm}
	\subfigure{
	\begin{minipage}[c]{0.2\textwidth}
\includegraphics[width=2.0cm]{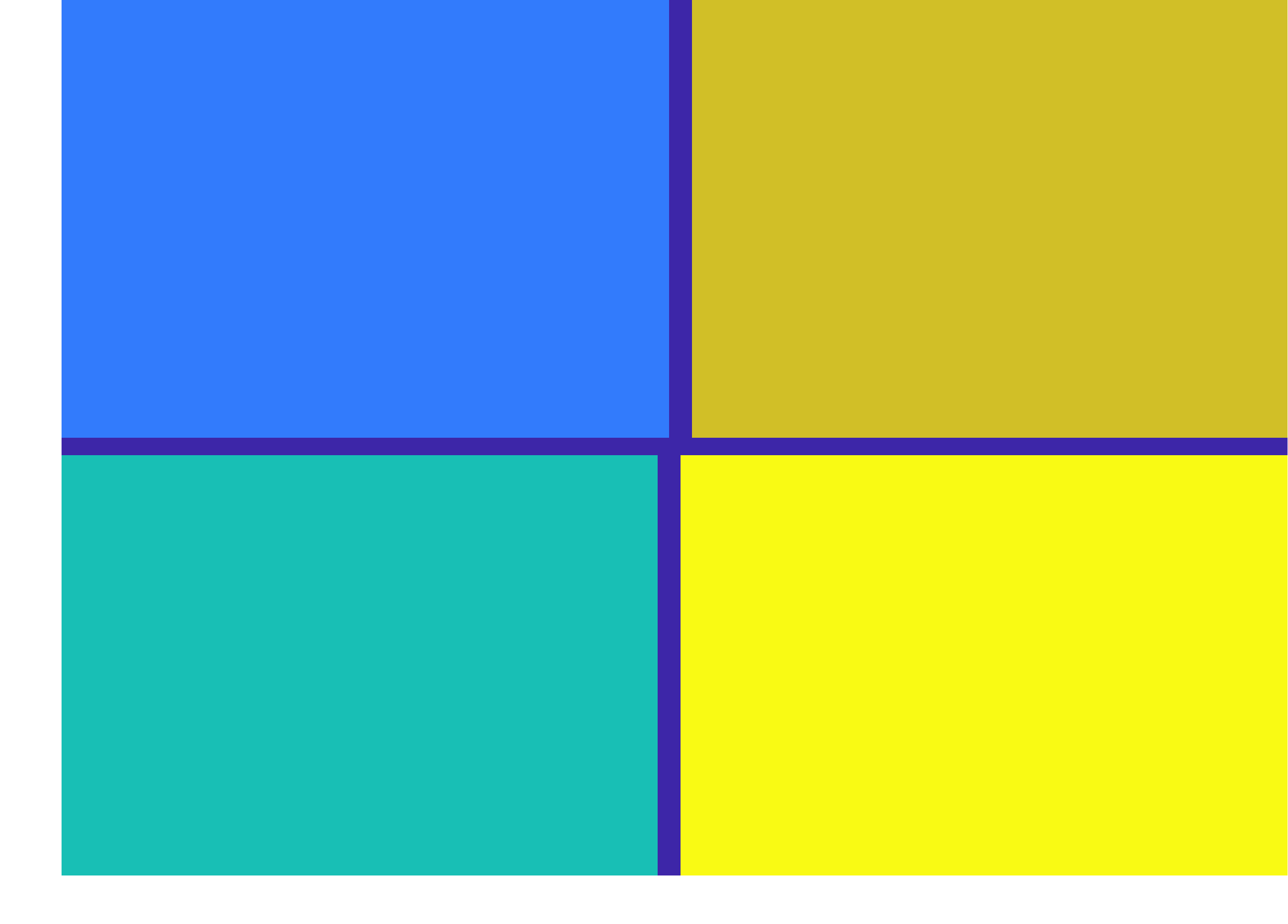}\vspace{1mm}
	\includegraphics[width=2.0cm]{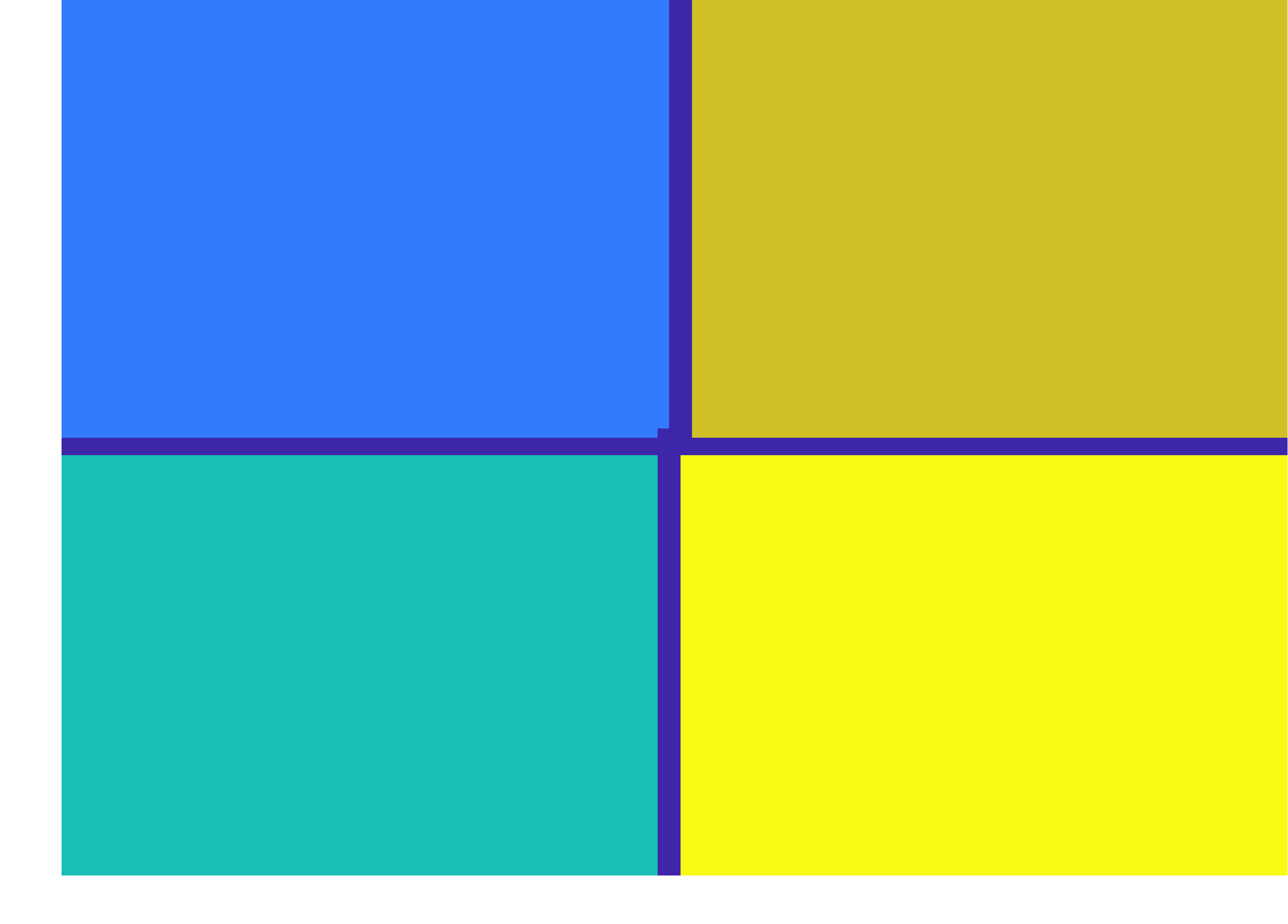}\vspace{1mm}
	\includegraphics[width=2.0cm]{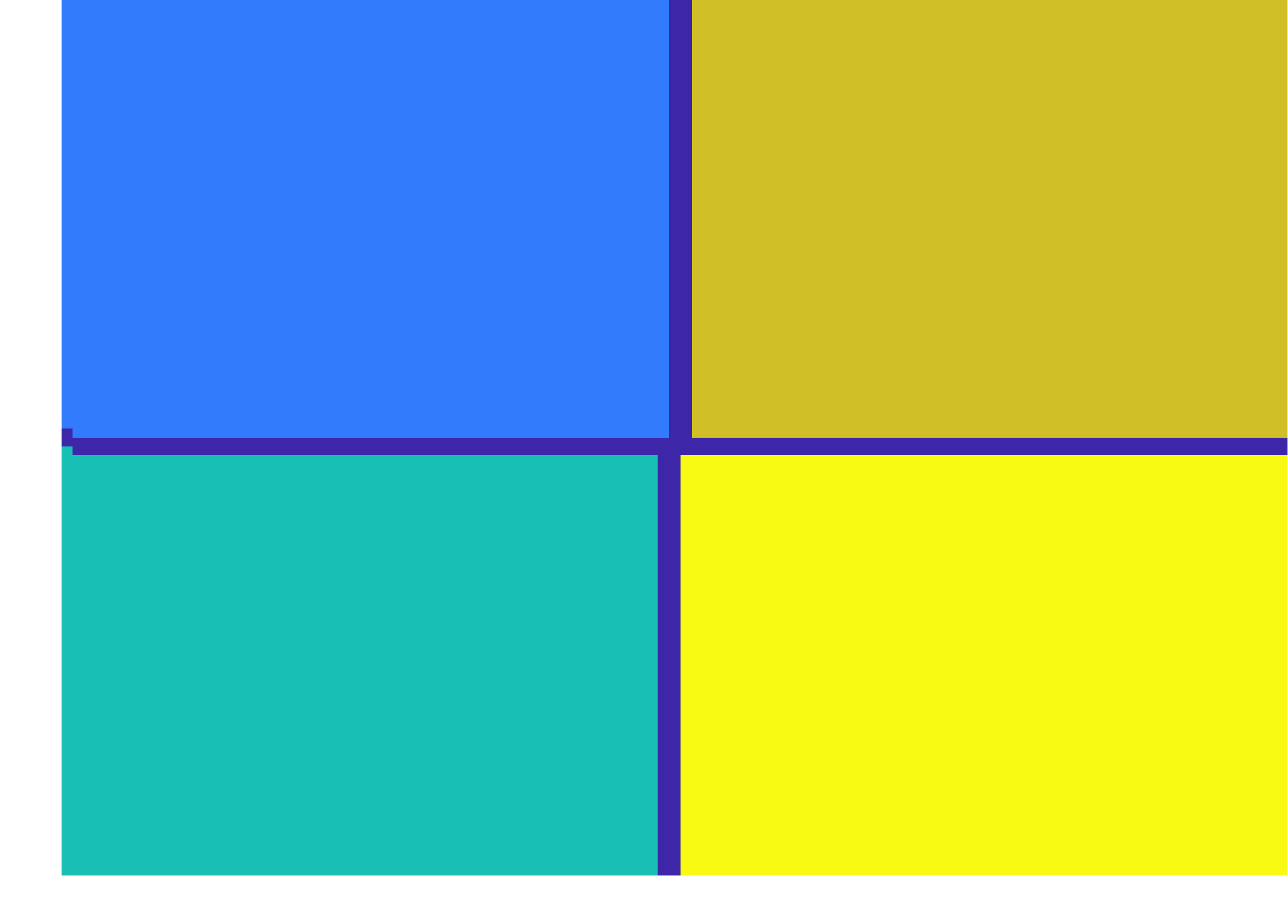} \vspace{2mm}
	\end{minipage}
	}
}
\caption{\label{fig:noise2} Segmentation comparison of image ``4blocks-1". From left to right:  Original image, SLaT\cite{Cai2015JSC}, ICTM\cite{Wang_2017}, CKA\cite{Wu2021IET},  and our ACCV model.}
\end{figure}

\begin{figure}[!t]
\resizebox{\textwidth}{!}{\hspace{11.8mm}
	\centering
\subfigure{
	\begin{minipage}[c]{0.2\textwidth}

\includegraphics[width=2.0cm]{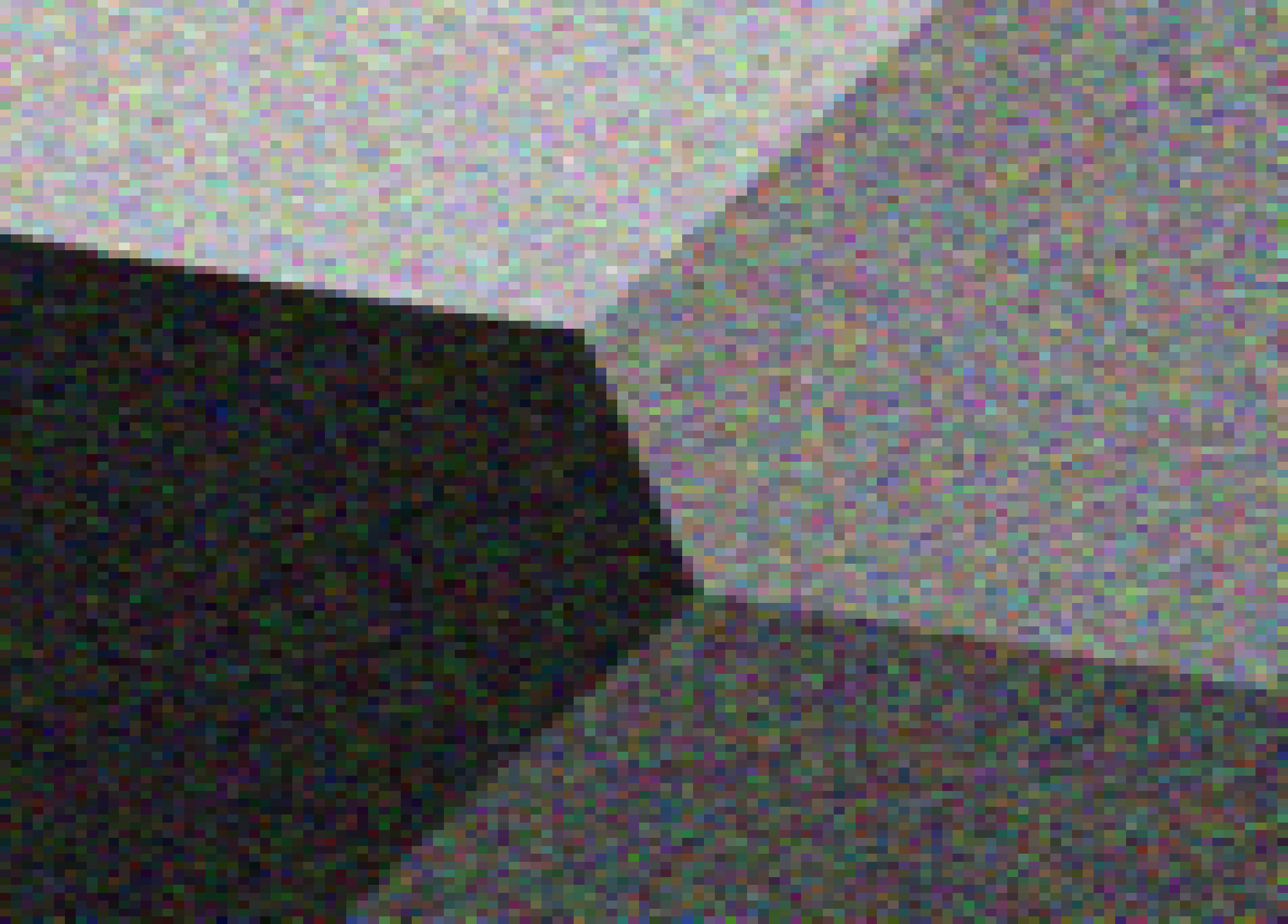}\vspace{1mm} 	
		\includegraphics[width=2.0cm]{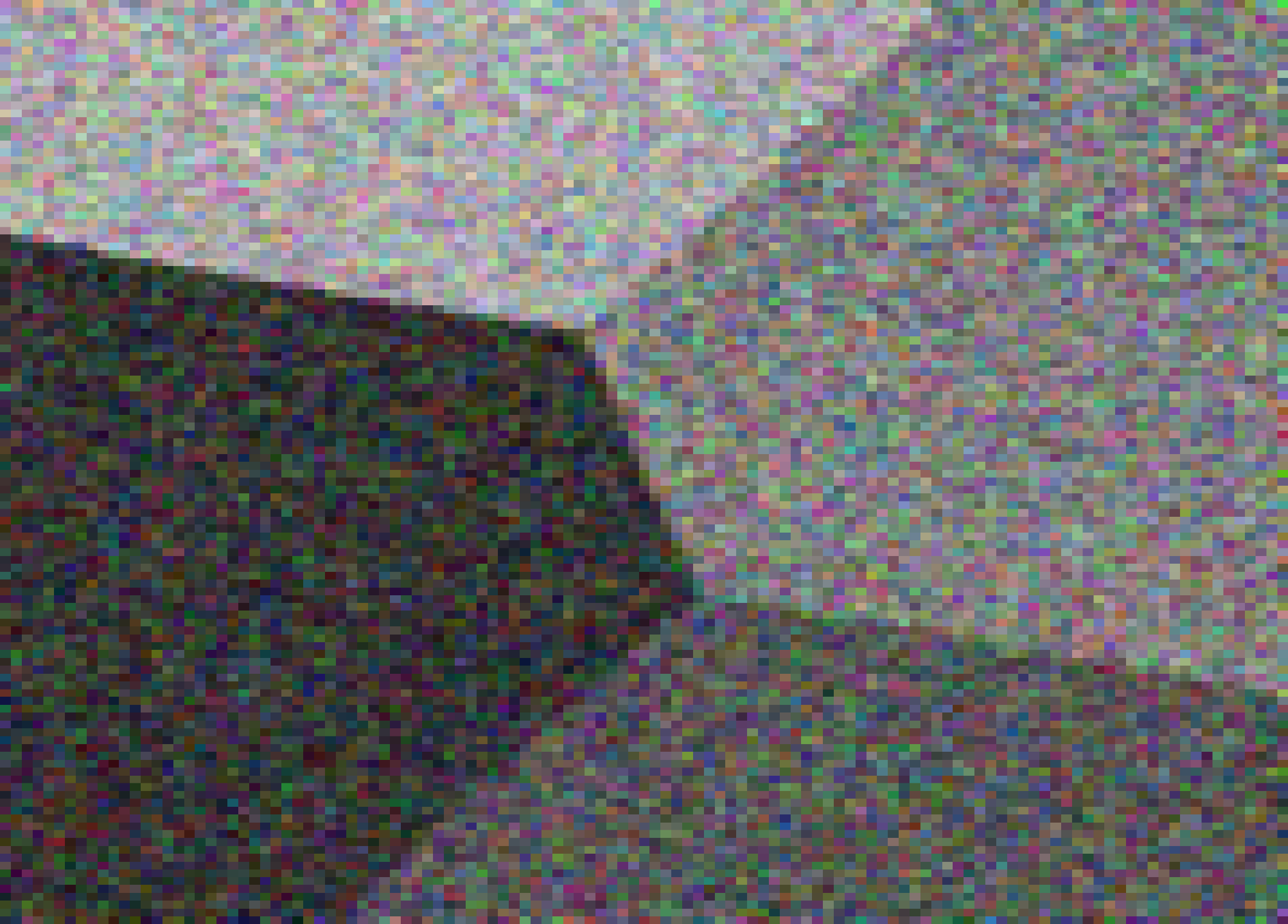} \vspace{1mm}\vspace{1mm}
		\includegraphics[width=2.0cm]{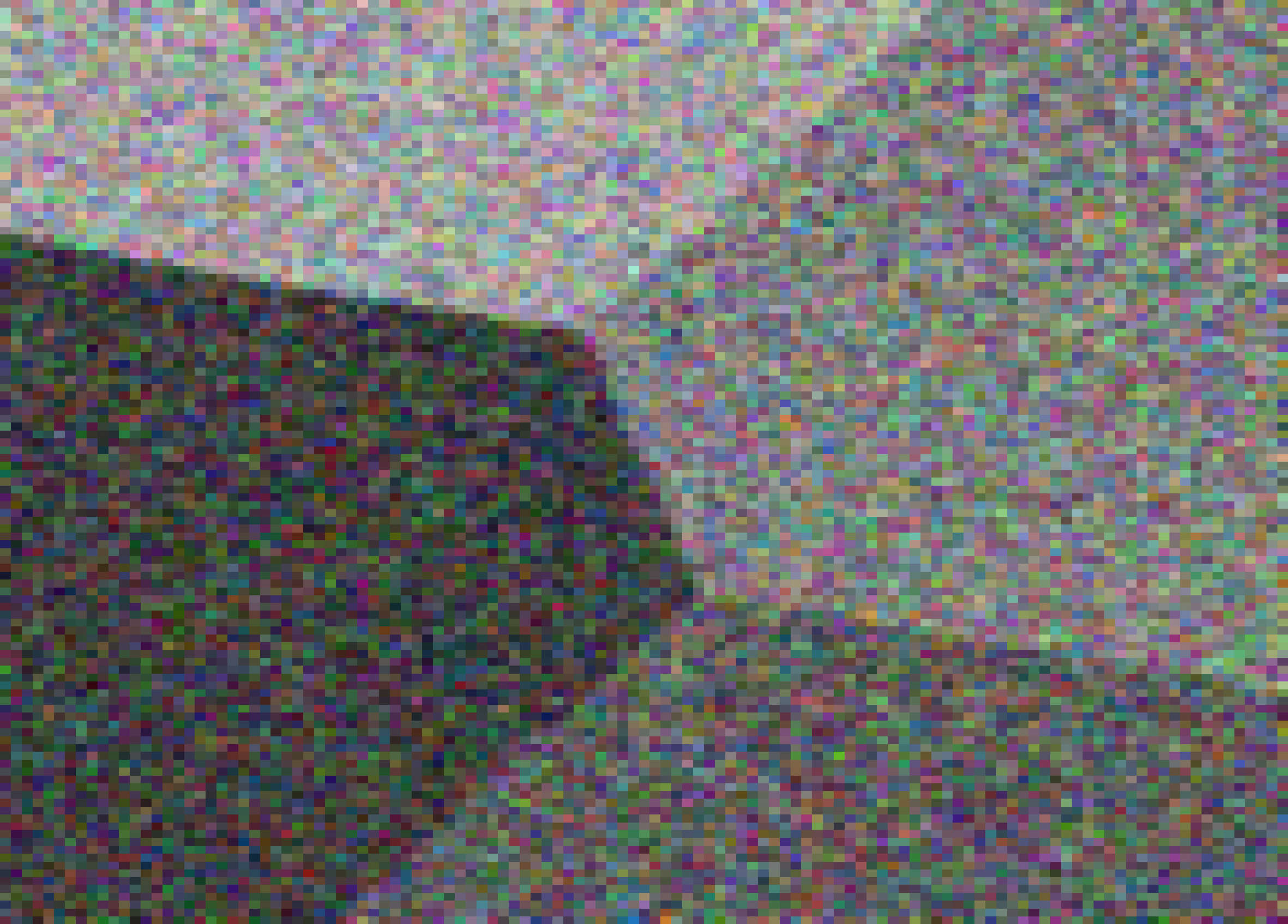} \vspace{1mm}
	\end{minipage}
	}
\hspace{-11.8mm}
	\subfigure{
	\begin{minipage}[c]{0.2\textwidth}

    \includegraphics[width=2.0cm]{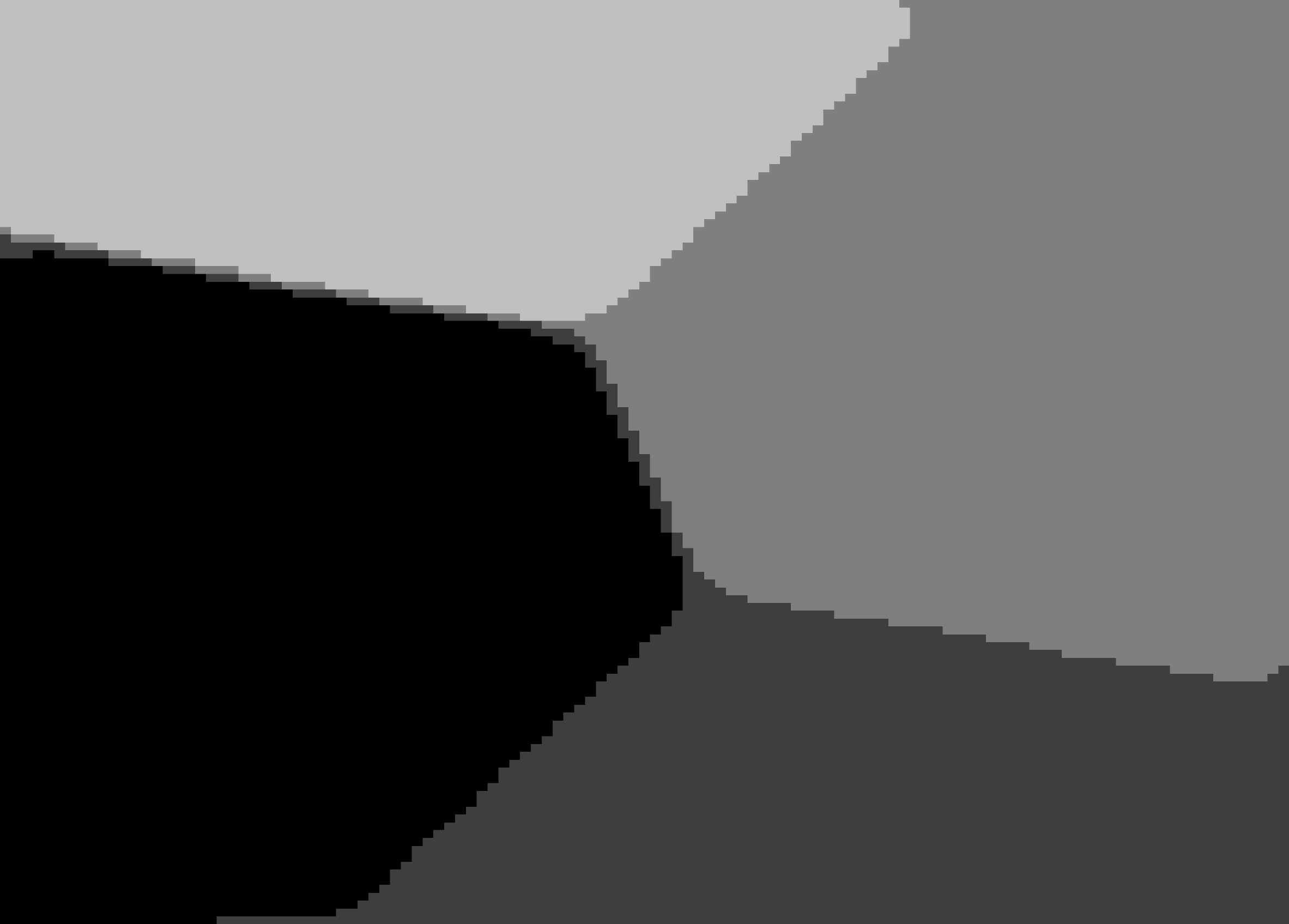}\vspace{1mm}
	\includegraphics[width=2.0cm]{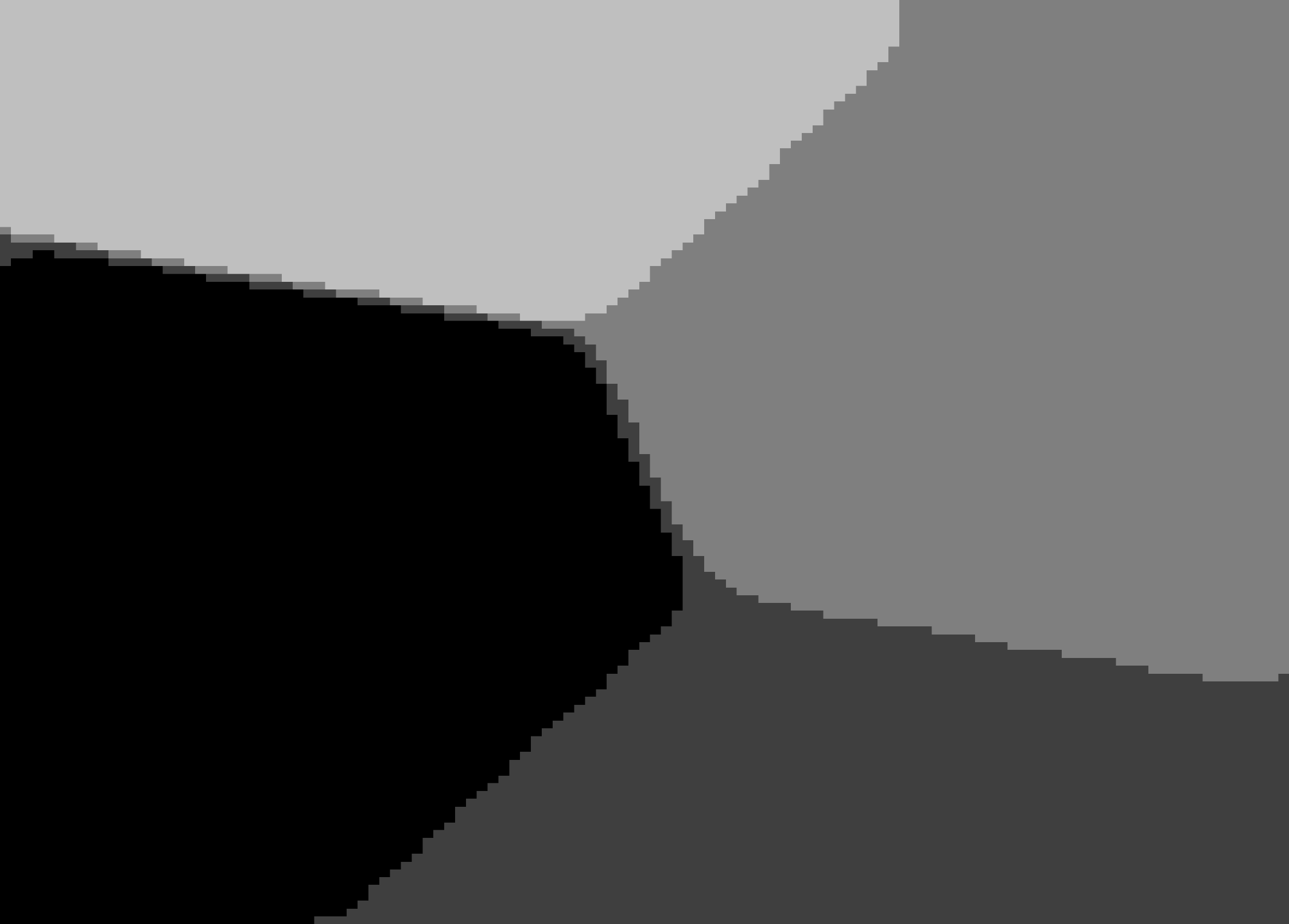}\vspace{1mm}
	\includegraphics[width=2.0cm]{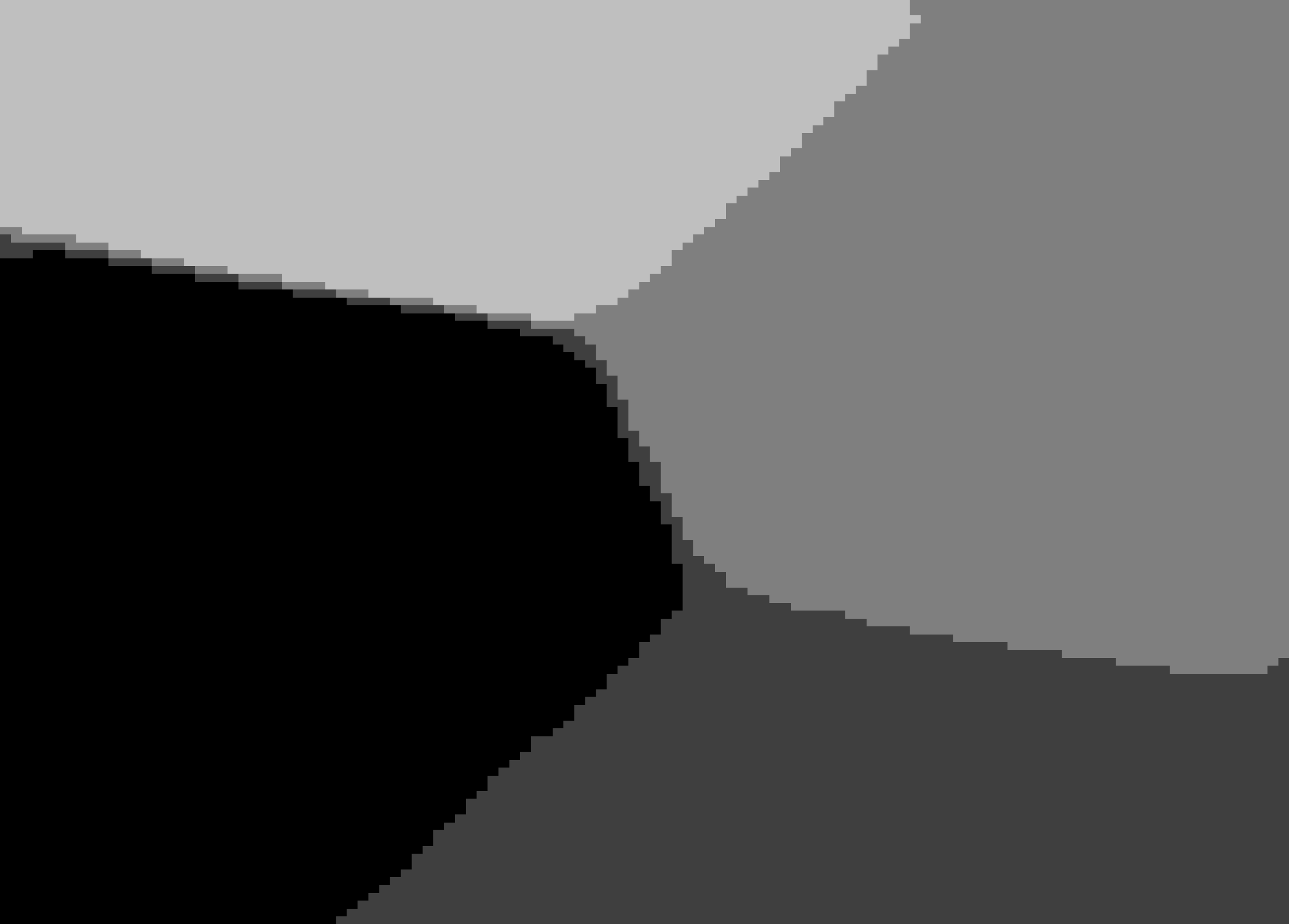}\vspace{1mm}
	\end{minipage}
	}
\hspace{-11.8mm}
	\subfigure{
	\begin{minipage}[c]{0.2\textwidth}

    \includegraphics[width=2.0cm]{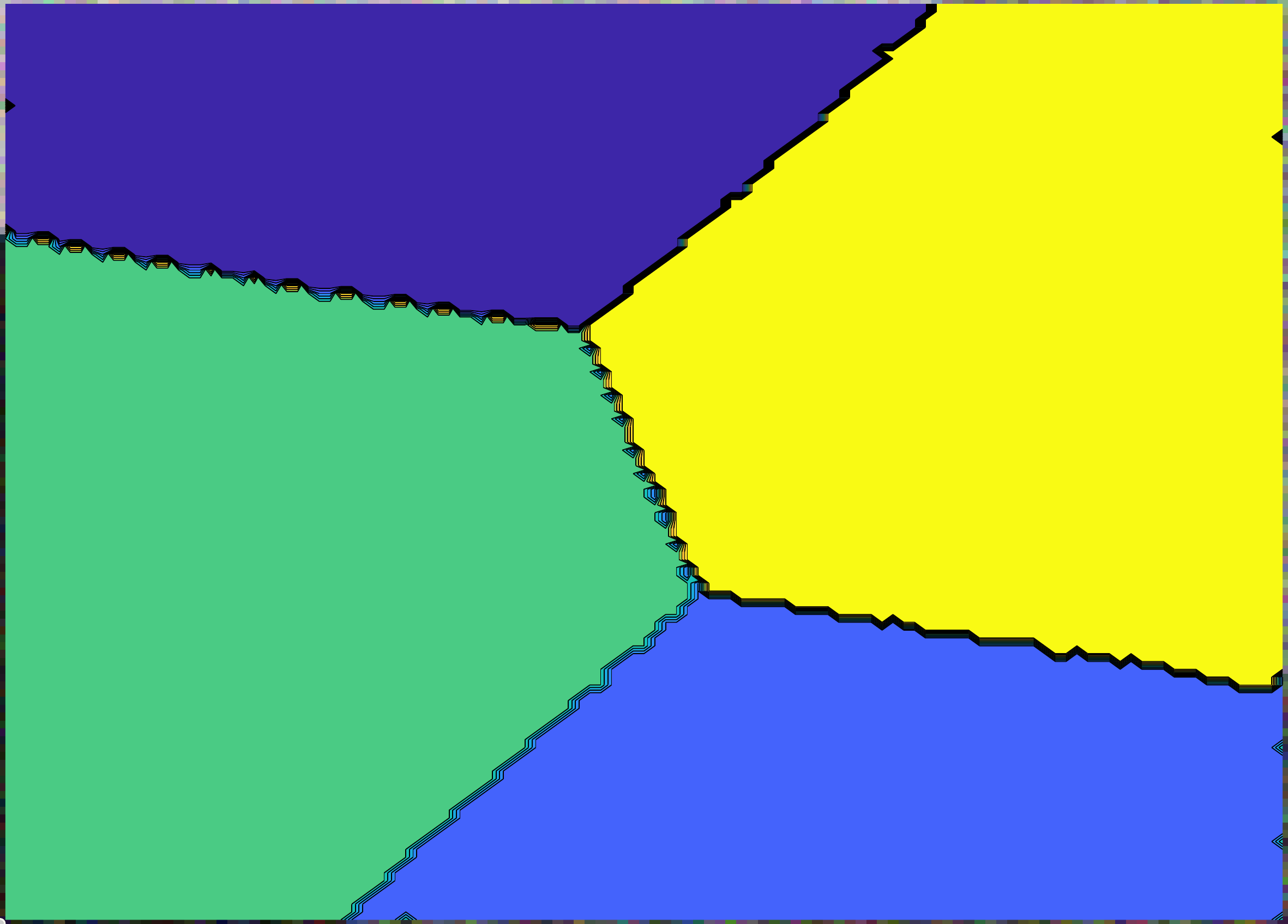}\vspace{1mm}
	\includegraphics[width=2.0cm]{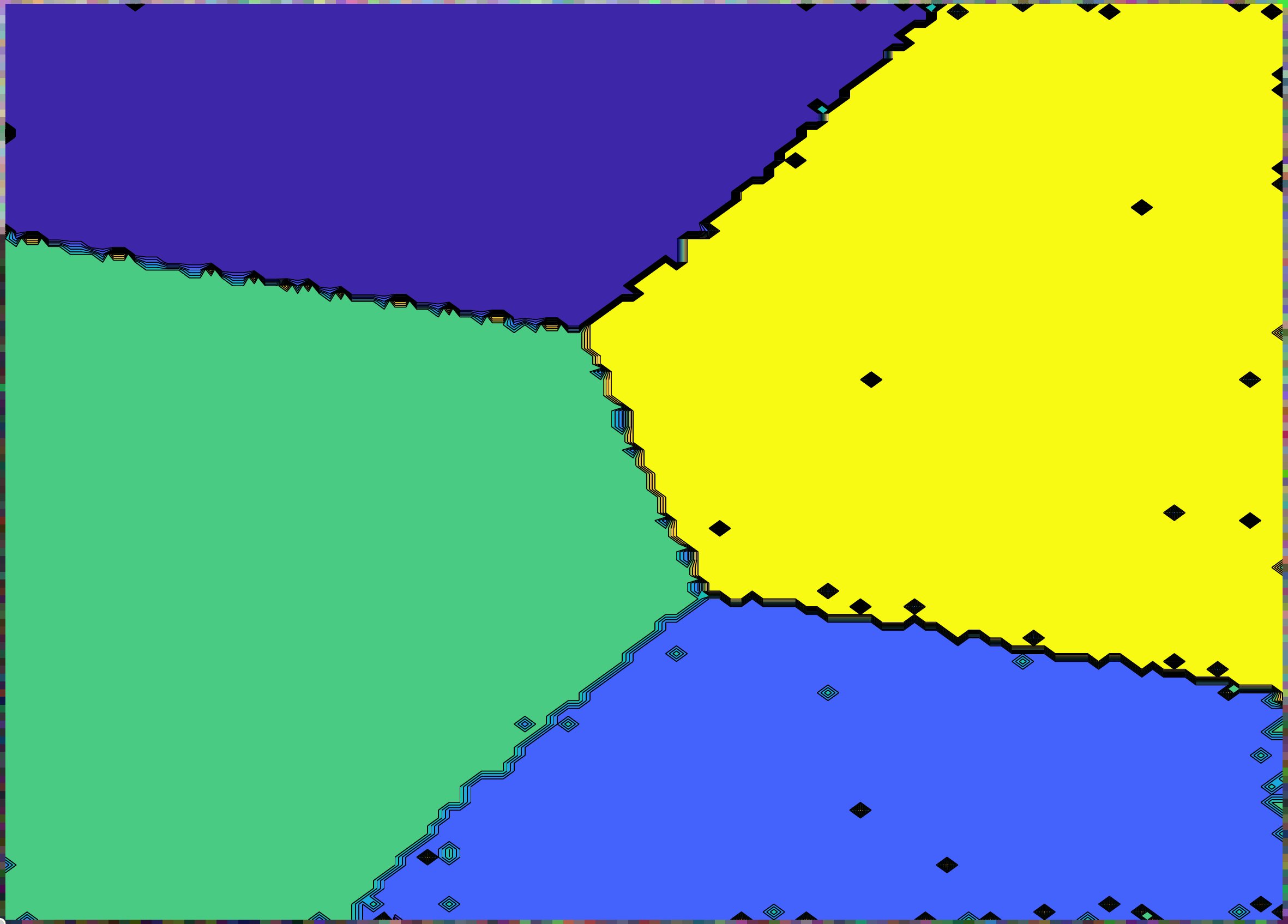}\vspace{1mm}
	\includegraphics[width=2.0cm]{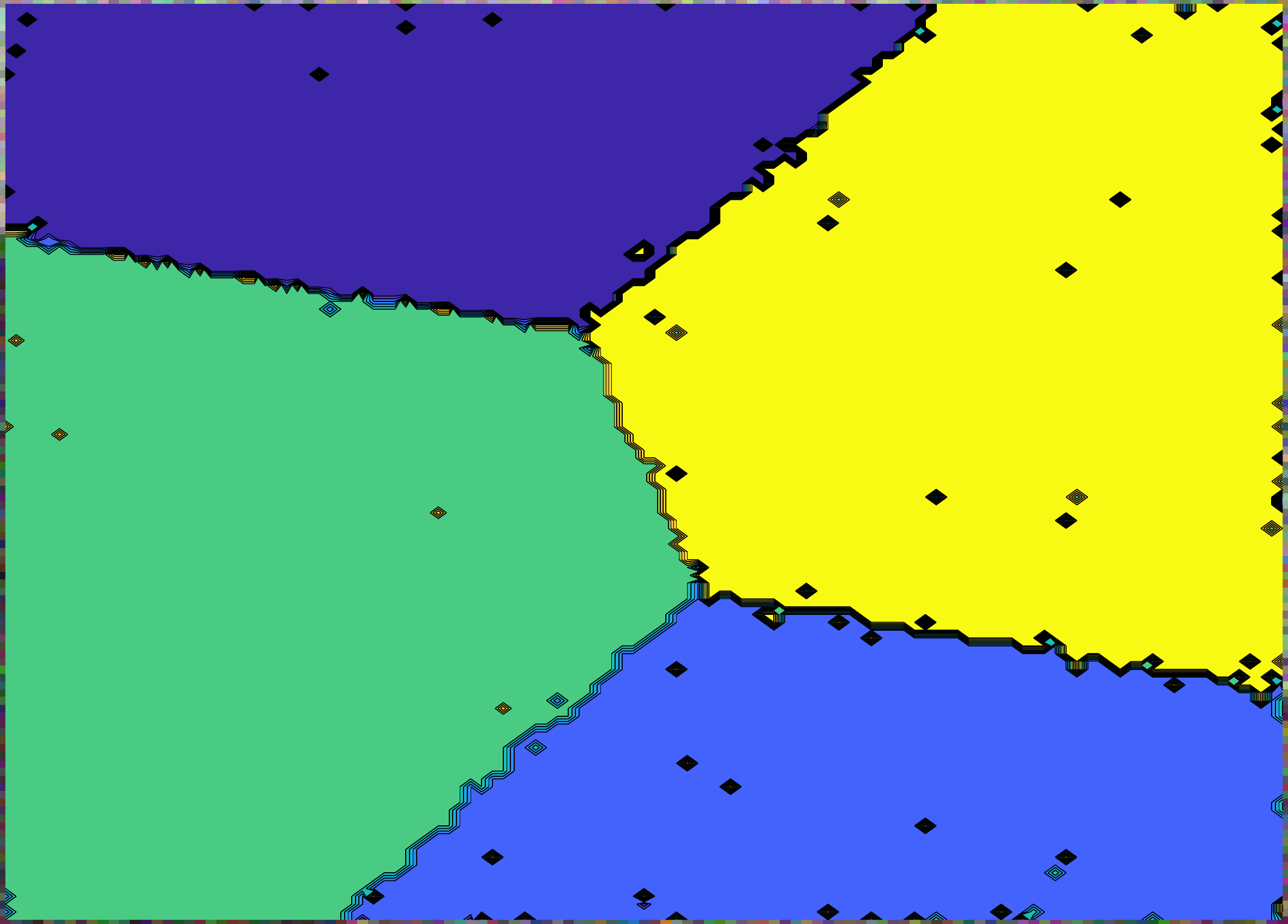}\vspace{1mm}
	\end{minipage}
	}
\hspace{-11.8mm}
\subfigure{
	\begin{minipage}[c]{0.2\textwidth}
	
    \includegraphics[width=2.0cm]{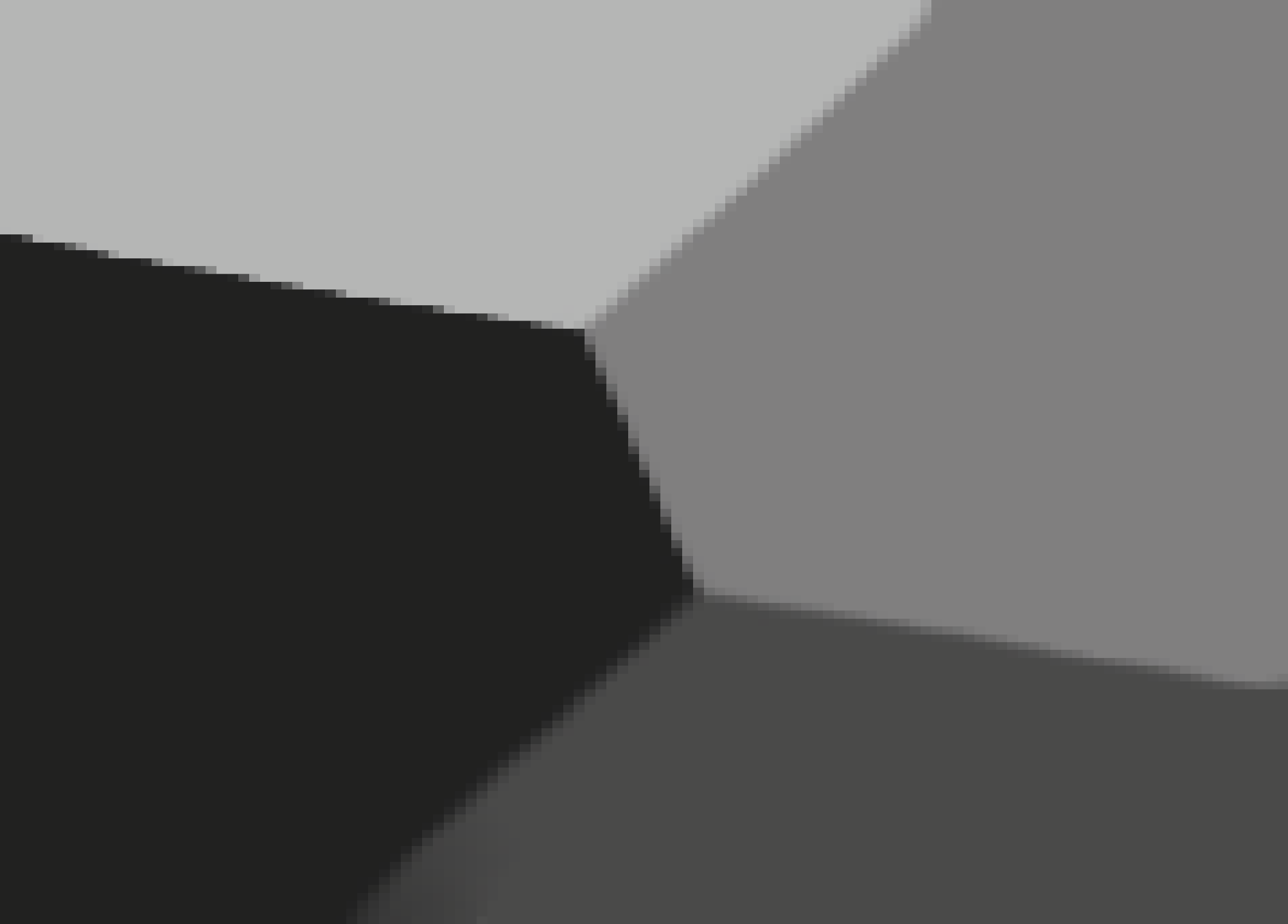}\vspace{1mm}
	\includegraphics[width=2.0cm]{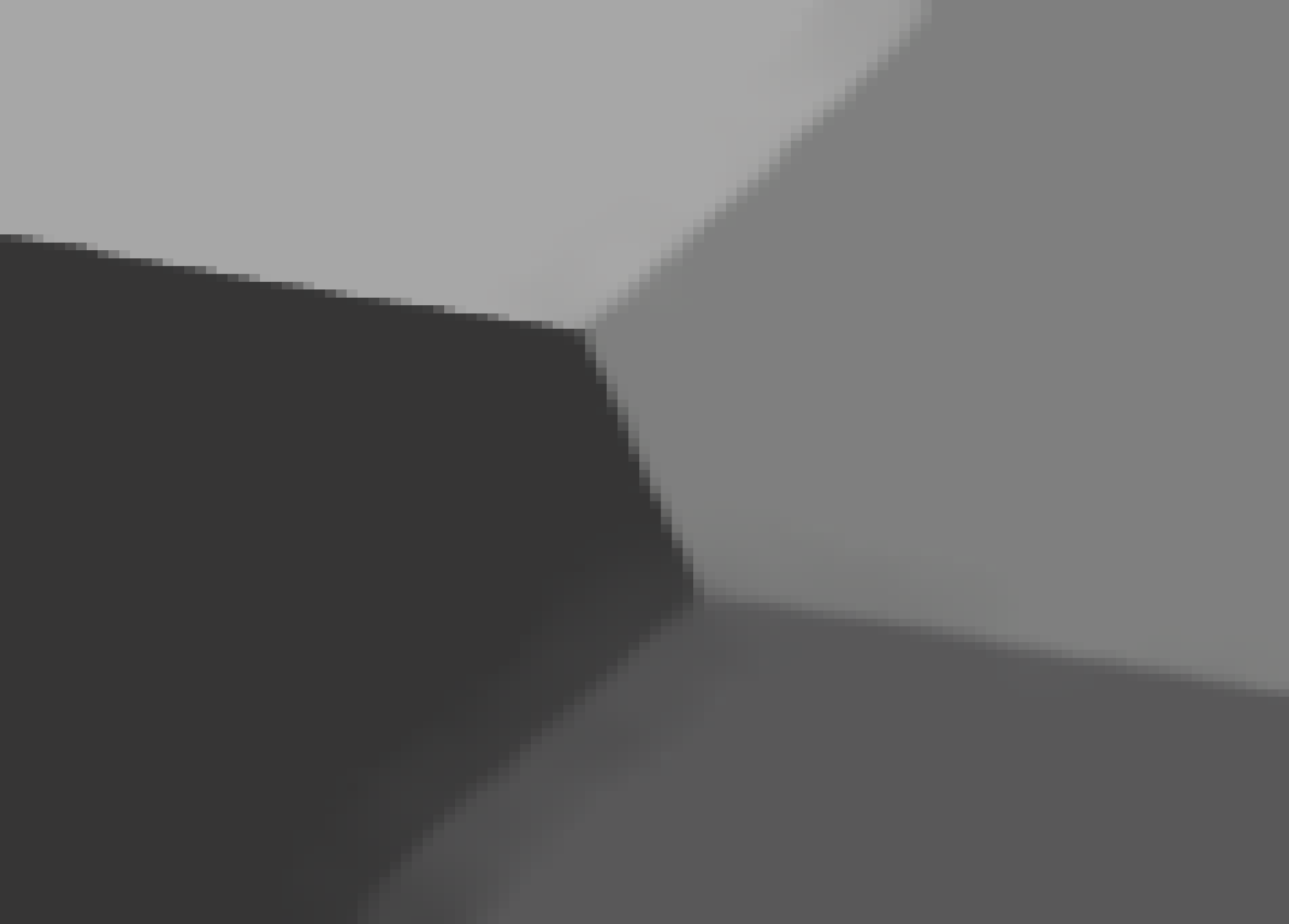}\vspace{1mm}
	\includegraphics[width=2.0cm]{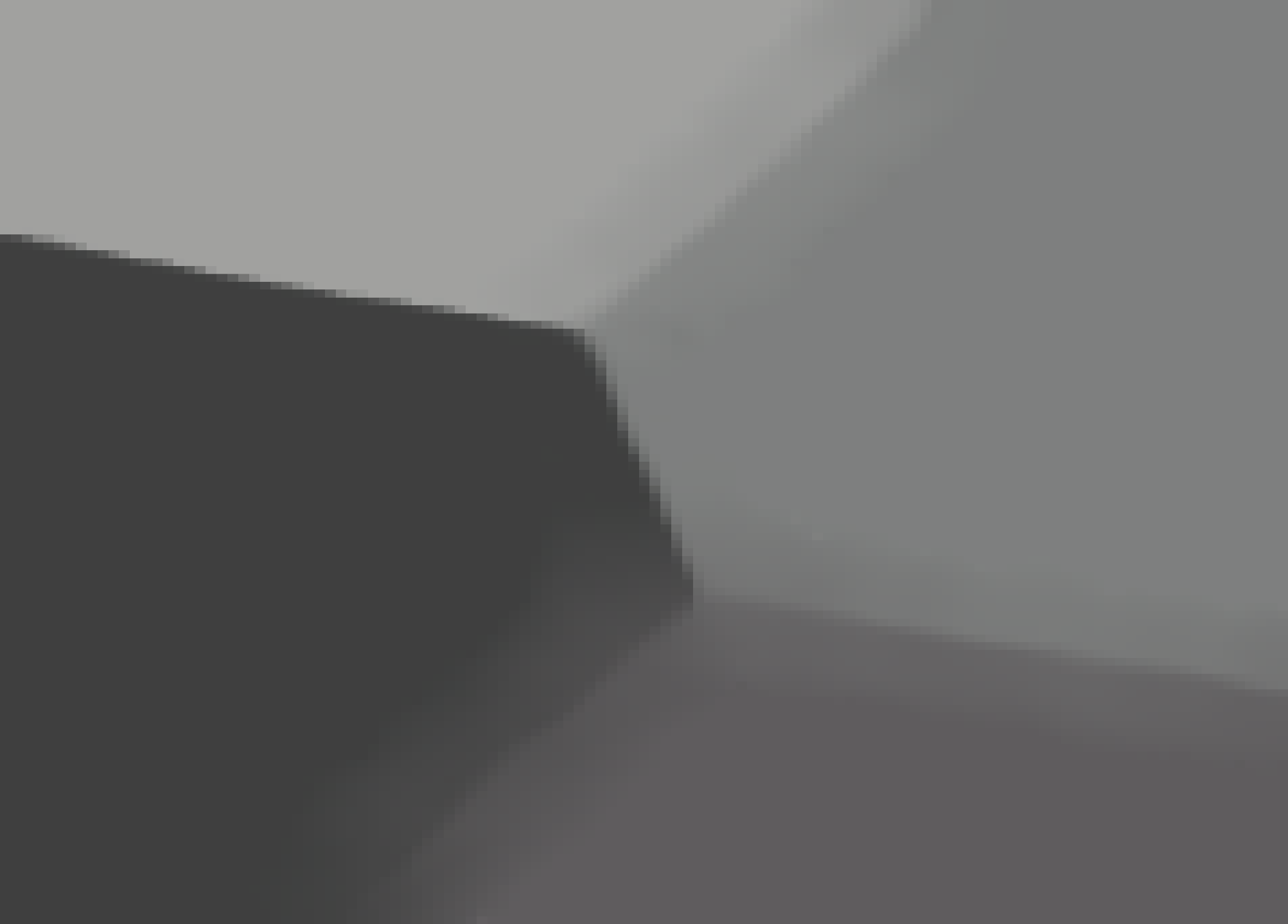}\vspace{1mm}
	\end{minipage}
	}
 \hspace{-11.8mm}
	\subfigure{
	\begin{minipage}[c]{0.2\textwidth}

    \includegraphics[width=2.0cm]{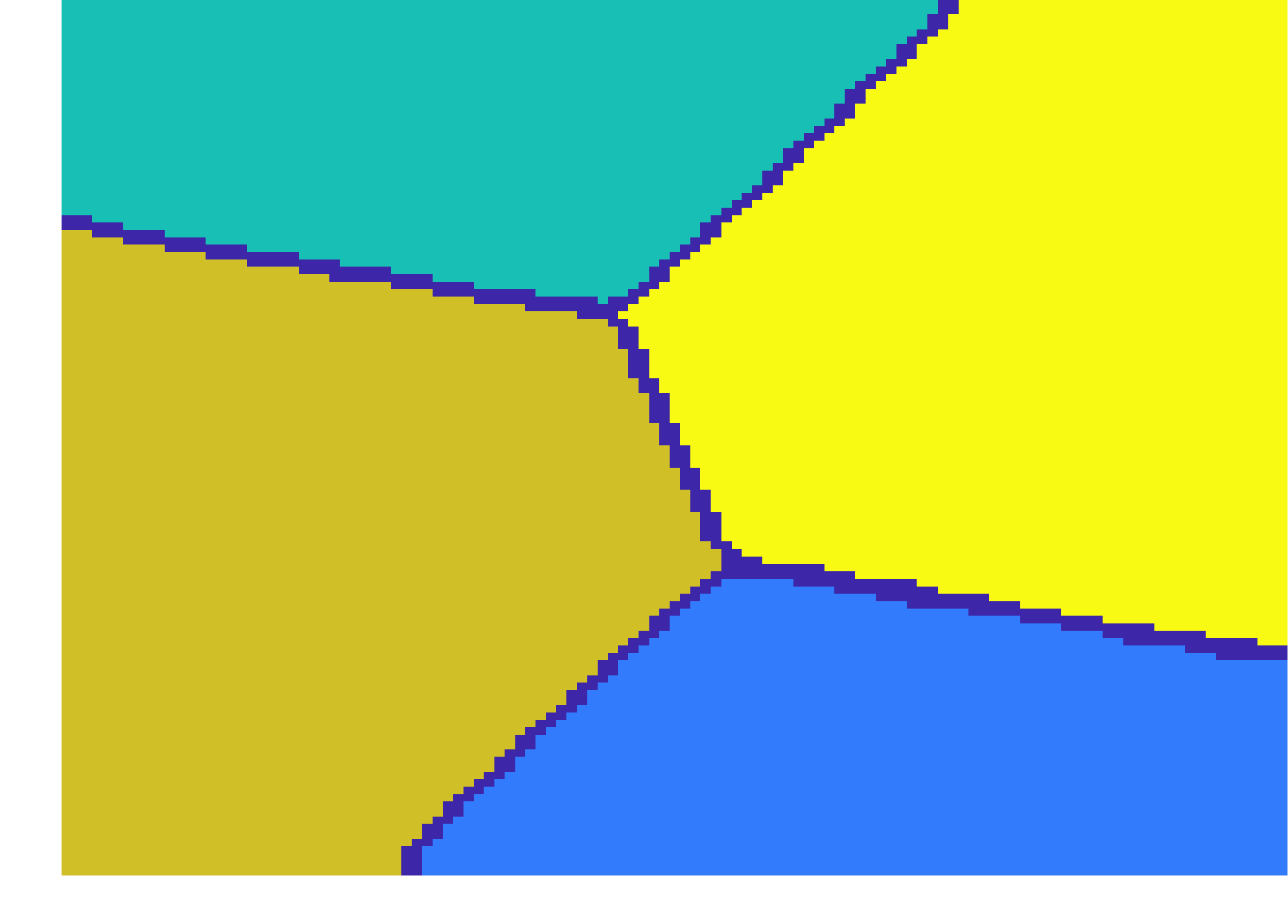}\vspace{1mm}
	\includegraphics[width=2.0cm]{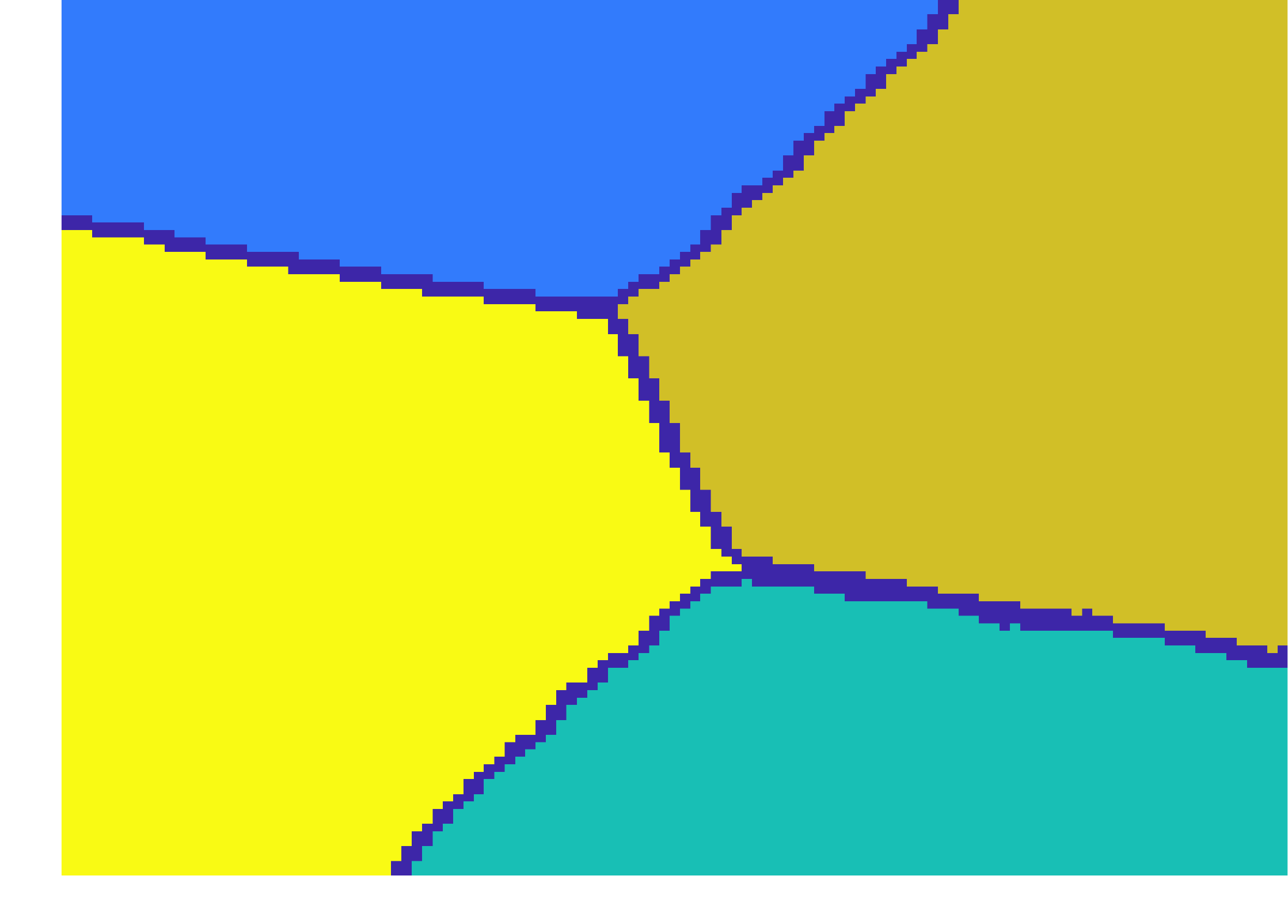}\vspace{1mm}
	\includegraphics[width=2.0cm]{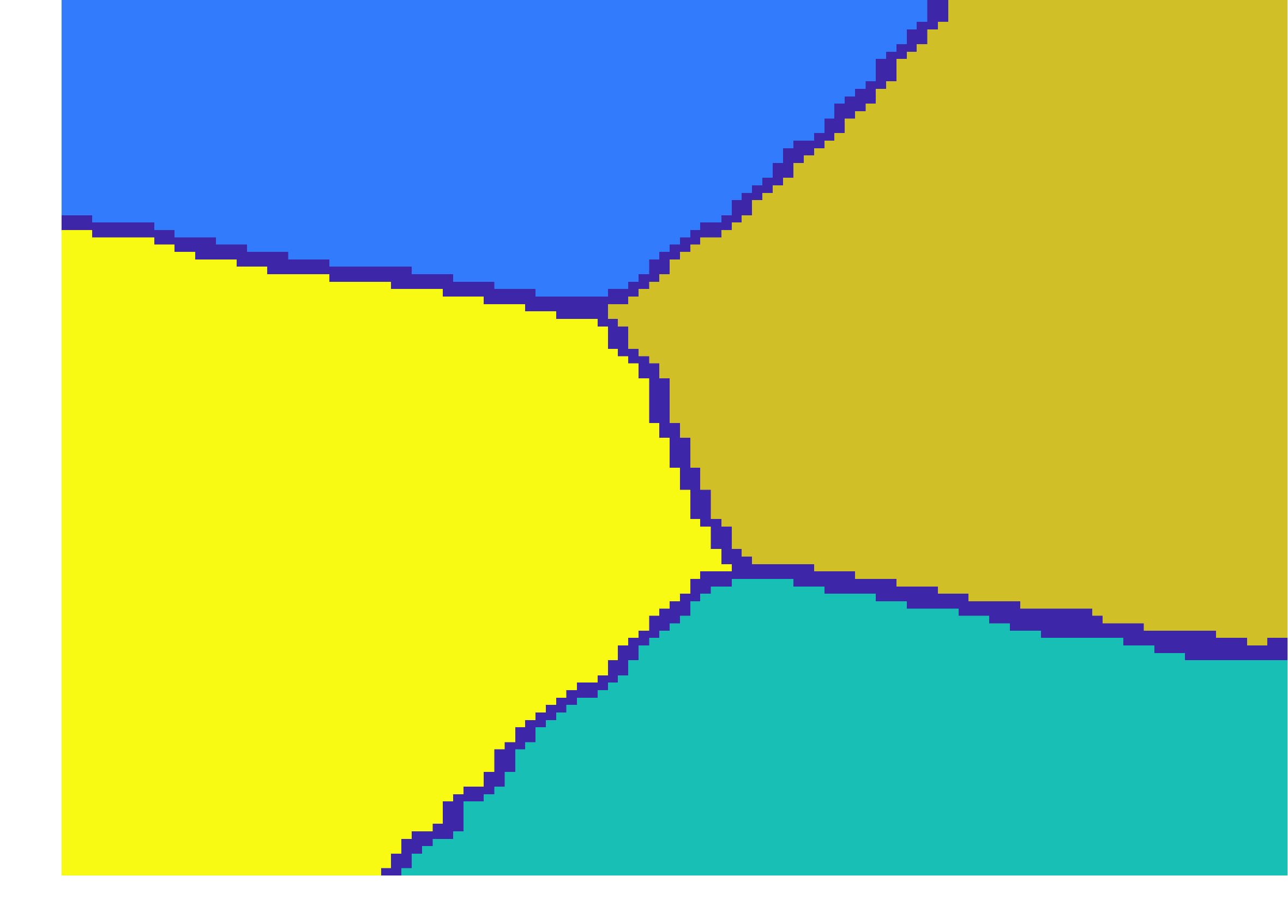}\vspace{1mm}
	\end{minipage}
	}
}
\caption{\label{fig:noise1} Segmentation comparison of image ``4blocks-2". From left to right:  Original image, SLaT\cite{Cai2015JSC}, ICTM\cite{Wang_2017}, CKA\cite{Wu2021IET}, and our ACCV model.}
\end{figure}

\begin{table}[!t]
\centering
\begin{tabular}{cccccc}
\hline
Figures & Rows     &SLaT\cite{Cai2015JSC}      &  ICTM\cite{Wang_2017}         &  CKA\cite{Wu2021IET}       &   ACCV\\ \hline
          &Row 1        & 0.99        &   8.17         & 8.38       &3.00\\  
Figure \ref{fig:noise2} &Row 2        & 0.98       &   17.57        & 8.11      & 3.61 \\ 
       &Row 3        & 0.98       &   31.04        & 8.25       & 3.77 \\\hline  
       &Row 1       & 1.34       &   2.35         &  4.33      &  1.52 \\
 Figure \ref{fig:noise1}       &Row 2        & 1.32       &   2.25         &  4.34     & 1.78       \\  
        &Row 3        & 1.34       &   2.82         & 4.36      & 1.59  \\ \hline 

\end{tabular}
 \caption{\label{tab:noise} Comparison on CPU time of Figures \ref{fig:noise2} and \ref{fig:noise1} for SLaT\cite{Cai2015JSC} , ICTM\cite{Wang_2017}, CKA\cite{Wu2021IET}, and our ACCV model.}

\end{table}
\section{Conclusion}

This paper develops the phase-field model in conjunction with the Chan--Vese fitting term, i.e., the ACCV model, to deal with the multi-phase image segmentation problems. An ADMM has been applied to minimize the objective energy functional. In the minimization process, an initialization method, Multi-IGLIM, is employed to give the initial contours. Variational calculus results in $n$ Allen--Cahn equations if there are $m = 2^n$ phases to be partitioned.  For Allen--Cahn equations,  the ETD1 and ETDRK2 schemes are adopted and analyzed for temporal discretization. Theoretically, we prove that our proposed method satisfies the discrete maximum bound principle and energy decay for ETD1 and ETDRK2 schemes, which are validated in the numerical experiments. Comparisons with other models show the effectiveness of our phase-field approach and the efficiency of the ADMM-ETD solver for segmenting various images and the robustness for handling noise.

\section*{Acknowledgments}
  We thank Prof. Zhongyi Huang at Tsinghua University for providing the MATLAB codes of \cite{Yang2019_JSC}, and thank Dr. Tingting Wu at Nanjing University of Posts and Telecommunications for providing the MATLAB codes of \cite{Wu2021IET}.

\end{document}